\documentclass[reqno]{amsart}
\usepackage{amssymb, amsmath, amsthm, mathrsfs, ascmac, color}
\usepackage[pdftex]{graphicx}
\usepackage{harvard}
\usepackage[foot]{amsaddr}
\usepackage[
top=2cm, bottom=2cm, left=2.5cm,right=2.5cm
]{geometry}
\newtheoremstyle{mystyle}% % Name
  {}%                      % Space above
  {}%                      % Space below
  {\normalfont}%           % Body font
  { }%                  % Indent amount
  {\bfseries}%             % Theorem head font
  {}%                      % Punctuation after theorem headｘ
  {10pt}%                  % Space after theorem head, ' ', or \newline
  { }%                     % Theorem head spec (can be left empty, meaning `normal')
\theoremstyle{mystyle}
\newtheorem{thm}{Theorem}

\newtheorem{lem}{Lemma}
\newtheorem{cor}[thm]{Corollary}
\newtheorem{rmk}{Remark}
\makeatletter

\@addtoreset{equation}{section}
\makeatother
\newcommand{\dd}{\mathrm d}

\newcommand{\EE}{\mathbb E}
\newcommand{\EEt}{\mathbb E_{\theta_0}}
\newcommand{\Ep}{\mathbb E_{\theta}}
\newcommand{\Pt}{P_{\theta_0}}
\newcommand{\GG}{\mathscr G_{i-1}^n}

\newcommand{\DeX}{\Delta X_i}
\newcommand{\Xt}{X_{t_i^n}}
\newcommand{\Xs}{X_{t_{i-1}^n}}
\newcommand{\Aa}{\mathcal A_{i-1}^{\ell_1,\ell_2}}

\newcommand{\Bd}{\mathcal B_{2,i-1}^{\ell,\ell_1}}
\newcommand{\Bm}{\mathcal B_{m,i-1}^{\ell,\ell_1}}
\newcommand{\TT}{\mathsf T}

\newcommand{\Hi}{\mathcal H_{1,i}}
\newcommand{\Hd}{\mathcal H_{2,i}}
\newcommand{\KK}{\mathcal K_i^\ell}

\newcommand{\Qi}{\mathcal Q_{1,i}}
\newcommand{\Qd}{\mathcal Q_{2,i}}
\newcommand{\Qt}{\mathcal Q_{3,i}}
\newcommand{\Qf}{\mathcal Q_{4,i}}
\newcommand{\Qg}{\mathcal Q_{5,i}}

\newcommand{\Zi}{\mathcal Z_{1,i}^\ell}
\newcommand{\Zd}{\mathcal Z_{2,i}^\ell}
\newcommand{\Zt}{\mathcal Z_{3,i}^\ell}
\newcommand{\Zf}{\mathcal Z_{4,i}^\ell}
\newcommand{\Zg}{\mathcal Z_{5,i}^\ell}
\newcommand{\Zs}{\mathcal Z_{6,i}^\ell}
\newcommand{\Yj}{\mathcal Y_{j,i}^\ell}
\newcommand{\Ym}{\mathcal Y_{m,i}^\ell}
\newcommand{\tr}{\mathrm{tr}}
\newcommand{\Ro}{R_{i-1}(1,\theta)}
\newcommand{\Ri}{R_{i-1}(h_n,\theta)}
\newcommand{\Rd}{R_{i-1}(h_n^2,\theta)}
\newcommand{\Rt}{R_{i-1}(h_n^3,\theta)}
\newcommand{\Rf}{R_{i-1}(h_n^4,\theta)}
\newcommand{\lto}{\longrightarrow}
\newcommand{\pto}{\stackrel{p}{\longrightarrow}}
\newcommand{\dto}{\stackrel{d}{\longrightarrow}}
\newcommand{\wto}{\stackrel{w}{\longrightarrow}}
\newcommand{\ato}{\stackrel{\mathrm{a.s.}}{\longrightarrow}}

\allowdisplaybreaks
\usepackage{comment}
\includecomment{jp-text}
\excludecomment{en-text}

\begin{document}
\title[Adaptive tests for parameter changes in ergodic diffusion processes]
{
Adaptive tests for parameter changes in ergodic diffusion processes from discrete observations
}
\author[Y. Tonaki]{Yozo Tonaki}%\and Yusuke Kaino \and
\author[Y. Kaino]{Yusuke Kaino}
\address[Y. Tonaki]
{Graduate School of Engineering Science, Osaka University, 1-3, 
Machikaneyama, Toyonaka, Osaka, 560-8531, Japan}
\address[Y. Kaino]{
Graduate School of Engineering Science, Osaka University, 1-3, 
Machikaneyama, Toyonaka, Osaka, 560-8531, Japan}
\author[M. Uchida]{Masayuki Uchida}
\address[M. Uchida]{
Graduate School of Engineering Science, 
and Center for Mathematical Modeling and Date Science,
Osaka University, 1-3, Machikaneyama, Toyonaka, Osaka, 560-8531, Japan}

%\date{\color{red} \today}

%%

%\input{change_point1}


%\begin{center}
%{K\scriptsize{EYWORDS}}.
%\end{center}
%Adaptive test, Brownian bridge, cusum test, diffusion processes, 
%{\color{red} discrete observations}, test for parameter change

\maketitle

\begin{abstract}
%In this paper, 
We consider the adaptive test for the parameter change in discretely
observed ergodic diffusion processes based on the cusum test. 
%Using the test statistics focusing on that the derivative of the quasi-log likelihood functions  have asymptotic normality, 
Using two test statistics based on 
the two quasi-log likelihood functions of the diffusion parameter and the drift parameter, 
we perform the change point tests for both diffusion and drift parameters of 
the diffusion process.
It is shown that the test statistics have the limiting distribution of 
the sup of the norm of a Brownian bridge. 
Simulation results are illustrated for the 1-dimensional Ornstein-Uhlenbeck process.
\end{abstract}

%%%%%%%%%%%%%%%%%%%%%%%%%%%%%%%%%%%%%%%%%%%%%%%%%%%%%%%%%%%%%%%%%%%%%%%%%%%%%%%%%%%%%%%%
%%%%%%%%%%%%%%%%%%%%%%%%%%%%%%%%%%%%%%%%%%%%%%%%%%%%%%%%%%%%%%%%%%%%%%%%%%%%%%%%%%%%%%%%
%%%%%%%%%%%%%%%%%%%%%%%%%%%%%%%%%%%%%%%%%%%%%%%%%%%%%%%%%%%%%%%%%%%%%%%%%%%%%%%%%%%%%%%%
%%%%%%%%%%%%%%%%%%%%%%%%%%%%%%%%%%%%%%%%%%%%%%%%%%%%%%%%%%%%%%%%%%%%%%%%%%%%%%%%%%%%%%%%
%%%%%%%%%%%%%%%%%%%%%%%%%%%%%%%%%%%%%%%%%%%%%%%%%%%%%%%%%%%%%%%%%%%%%%%%%%%%%%%%%%%%%%%%
\section{Introduction}
We consider a $d$-dimensional diffusion process $\{X_t\}_{t\ge0}$ 
satisfying the stochastic differential equation:
\begin{align}\label{eq2.1}
\begin{cases}
\dd X_t=b(X_t,\beta)\dd t+a(X_t,\alpha)\dd W_t,\\
X_0=x_0,
\end{cases}
\end{align}
where parameter space
$\Theta=\Theta_A\times\Theta_B$, which is a %convex 
compact subset of 
$\mathbb R^p\times\mathbb R^q$, 
$\theta=(\alpha,\beta)\in\Theta$ is an unknown parameter and
$\{W_t\}_{t\ge0}$ 
is a $d$-dimensional standard Wiener process.
%%%
The diffusion coefficient 
$a:\mathbb R^d\times\Theta_A\lto\mathbb R^d\otimes\mathbb R^d$ and
the drift coefficient $b:\mathbb R^d\times\Theta_B\lto\mathbb R^d$
are known
except for the parameter $\theta$, 
and the true parameter $\theta_0=(\alpha_0,\beta_0)$ belongs to $\mathrm{Int\,}\Theta$.
We assume that the solution of \eqref{eq2.1} exists, and 
$P_\theta$ and $\EE_\theta$ denote the law of the 
solution and the expectation with respect to $P_\theta$, respectively. 
%%%%%%%%%%%%%%%%%%%%%%%%%%

We deal with the testing problem for parameter change in diffusion processes
from discrete observations $\{\Xt\}_{i=0}^n$, 
where $t_i^n=ih_n$ and $\{h_n\}$ is a positive sequence 
such that $h_n\lto0$ and $nh_n\lto\infty$. We further assume $nh_n^2\lto0$.
%%%%%%%%
In order to analyze  stochastic phenomena, 
we construct a statistical model with statistical inference for unknown parameters.
%we assume some model and make statistical inferences. 
%At that time, 
In some cases, the values of the parameters of the model 
may change due to some influence.
For example, if we consider a model of the stock price fluctuation, then the parameters of the statistical model 
may change due to factors such as interest rates, economic conditions, or political trends.
It is important to see if the values of the parameters of this model have changed. 
That is, we want to determine if there are any changes in such paths. 
As we see from the paths shown in Figure \ref{fig1}, 
%are obtained from a stock price fluctuation,
%At this time, 
the parameters of this model may have changed.
In particular, it would be nice to be able to determine 
whether there is a change in the paths as shown in Figure \ref{fig2}. 
However, it seems difficult to determine whether there are changes in these paths.
In this paper, we show that our proposed test statistics detect whether these paths have changed  
and the tests also tell us whether either the diffusion parameter or the drift parameter changes. 
\begin{figure}[h]
\centering
\caption{
Sample paths of the 1-dimensional Ornstein-Uhlenbeck process
$\dd X_t=-\beta(X_t-\gamma)\dd t+\alpha\dd W_t$
whose parameter
%where the parameter 
changes from $(\alpha,\beta,\gamma)=(1,1,1)$ to 
$(3,1,1)$ (left) and changes to $(1,1,-1)$ (right) at $t=50$.}
\includegraphics[height=7.5cm, width=15cm]{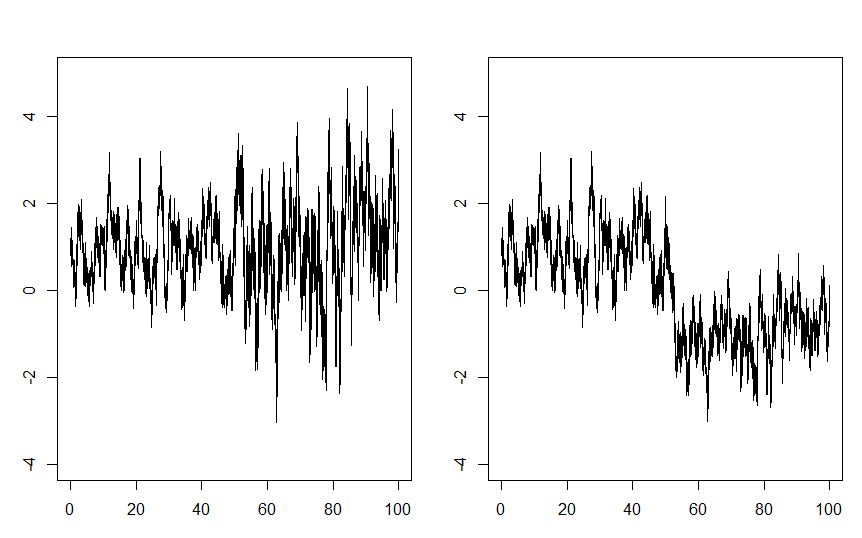}
\label{fig1}
\end{figure}

%%%%%%%%%%%%%%%%%%
\begin{figure}[h]
\centering
\caption{
Sample paths of the 1-dimensional Ornstein-Uhlenbeck process
$\dd X_t=-\beta(X_t-\gamma)\dd t+\alpha\dd W_t$
whose parameter
from $(\alpha,\beta,\gamma)=(1,1,1)$ to $(1.05,1,1)$ (upper right),
$(1,2,1)$ (lower left),
and 
$(1,1,0.3)$ (lower right)
at $t=50$.
In the upper left figure, the parameter does not change
with $(\alpha,\beta,\gamma)=(1,1,1)$.
}
\includegraphics[height=10cm, width=15cm]{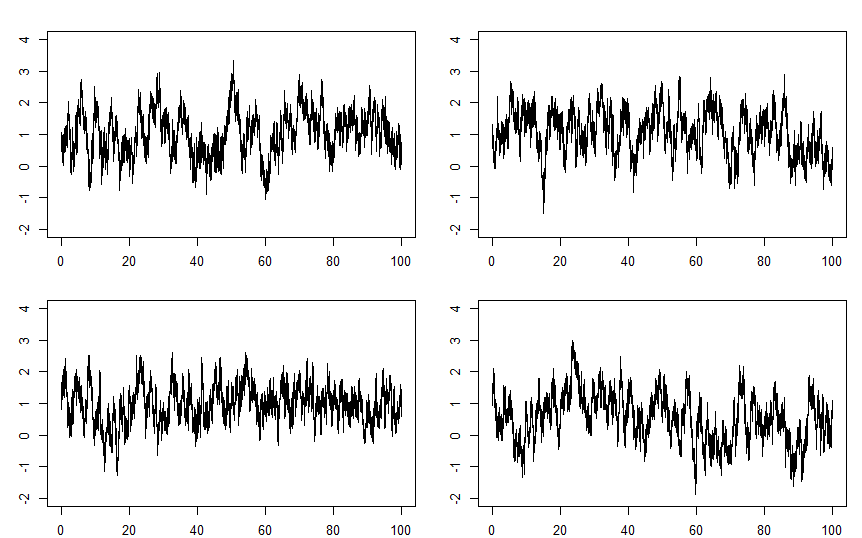}
\label{fig2}
\end{figure}

%%%%%%%%%%%%%%%%%%

Change point problems for diffusion processes has been developed by many researchers.
Kutoyants (2004), Lee et al. (2006), 
Negri and Nishiyama (2012) and
Tsukuda (2017) 
treated the problem of testing for a change of the drift parameter
in continuously observed ergodic diffusion processes.
On the other hand, 
De Gregorio and Iacus (2008),
Song and Lee (2009), 
Lee (2011)
and
Song (2020)
%Negri and Nishiyama (2017)
studied 
the problem of testing for a change of the diffusion parameter 
%{the testing problem for parameter changes}
for discretely observed diffusion processes,
see also Iacus and Yoshida (2012)
for the estimation problem for the change point of the volatility in a stochastic differential equation. 
%%%%
Song and Lee (2009) considered the testing problem for a change of diffusion parameter $\alpha$ 
for the following stochastic differential equation with known drift function $b(x, \beta)=b(x)$
%%
%\begin{align*}
%\dd X_t=b(X_t)\dd t+a(X_t,\alpha)\dd W_t,
%\end{align*}
under the assumption $nh_n^p\lto\infty,\ nh_n^q\lto0\ (4< p< q)$.
%, and
%shows that the proposed test statistic  of diffusion parameter $\alpha$ 
%weakly converges to the sup of the square of a Brownian bridge.
%%%
%Further, Lee (2011) deals with testing for changes in diffusion parameter $\alpha$
%of the stochastic differential equation
%\begin{align*}
%\dd X_t=b(X_t,\beta)\dd t+a(X_t,\alpha)\dd W_t,
%\end{align*}
%under $nh_n^2\lto0$, and considers that the proposed test statistic 
%of diffusion parameter $\alpha$ weakly converges to the sup of a Brownian bridge,
%unaffected by the drift parameter estimation.
%%
Furthermore, %Lee (2011) deal with
Song (2020) considered the test statistic for a change in the diffusion parameter $\alpha$
of the stochastic differential equation with $a(x,\alpha)=\alpha$
%\begin{align*}%\label{sde-lee}
%\dd X_t=b(X_t,\beta)\dd t+\alpha\dd W_t,
%\end{align*}
under $nh_n^2\lto0$. 
%Song (2020) consider when $a(x,\alpha)=\alpha$. 
%that the proposed test statistic 
%for diffusion parameter $\alpha$ weakly converges to the sup of a Brownian bridge. 
%Note that the test statistic of diffusion parameter $\alpha$ proposed by them 
%requires an estimator for drift parameter $\beta$.
%
Negri and Nishiyama (2017) proposed the test statistics for changes of the  diffusion 
parameter or  the drift parameter in the stochastic differential equation.  
%\eqref{sde-lee} by applying Z-process method.
%%%%%
In this paper, we construct  the two kinds of adaptive test statistics for 
the diffusion parameter $\alpha$ and the drift parameter $\beta$.
Using the proposed test statistics, 
we can determine whether either the diffusion parameter or the drift parameter changes. 

This paper is organized as follows.
In Section 2, we state the main results.
We propose the  adaptive test statistics and show that the test statistics have
limiting distribution of the sup of the norm of a Brownian bridge. 
In Section 3, examples and simulation studies are given.
In Section 4, we provide the proofs.
%%%%%%%%%%%%%%%%%%%%%%%%%%%%%%%%%%%%%%%%%%%%%%%%%%%%%%%%%%%%%%%%%%%%%%%%%%%%%%%%%%%%%%%%
%%%%%%%%%%%%%%%%%%%%%%%%%%%%%%%%%%%%%%%%%%%%%%%%%%%%%%%%%%%%%%%%%%%%%%%%%%%%%%%%%%%%%%%%
%%%%%%%%%%%%%%%%%%%%%%%%%%%%%%%%%%%%%%%%%%%%%%%%%%%%%%%%%%%%%%%%%%%%%%%%%%%%%%%%%%%%%%%%
%%%%%%%%%%%%%%%%%%%%%%%%%%%%%%%%%%%%%%%%%%%%%%%%%%%%%%%%%%%%%%%%%%%%%%%%%%%%%%%%%%%%%%%%
%%%%%%%%%%%%%%%%%%%%%%%%%%%%%%%%%%%%%%%%%%%%%%%%%%%%%%%%%%%%%%%%%%%%%%%%%%%%%%%%%%%%%%%%
\section{Main results}
Let
%\begin{align*}
$A=aa^\TT$, $\DeX= X_{t_i^n}-X_{t_{i-1}^n}$,
$(\DeX)^{\otimes2}=(\DeX)(\DeX)^\TT$,
%\end{align*}
where 
$\TT$ is the transpose.
Let $C^{k,\ell}_{\uparrow}(\mathbb R^d\times\Theta)$ be 
the space of all functions $f$ satisfying the following conditions:
\begin{enumerate}
\item[(i)] 
$f$ is continuously differentiable with respect to $x\in\mathbb R^d$ up to
order $k$ for all $\theta\in\Theta$;
\item[(ii)]  
$f$ and all its $x$-derivatives up to order $k$ are 
$\ell$ times continuously differentiable with respect to $\theta\in\Theta$;
\item[(iii)]  
$f$ and all derivatives are of polynomial growth in $x\in\mathbb R^d$ 
uniformly in $\theta\in\Theta$, i.e.,  
$g$ is of polynomial growth in $x\in\mathbb R^d$ uniformly 
in $\theta\in\Theta$ if, for some $C>0$, we have
\begin{align*}
\sup_{\theta\in\Theta}\|g(x,\theta)\|\le C(1+\|x\|)^C.
\end{align*}
\end{enumerate}
We assume the following conditions:
\begin{enumerate}
\renewcommand{\labelenumi}{{\textbf{[A\arabic{enumi}]}}}
%%%
\item 
There exists a constant $C$ such that for any $x,y\in\mathbb R^d$, 
\begin{align*}
\sup_{\alpha\in\Theta_A}\|a(x,\alpha)-a(y,\alpha)\|
+\sup_{\beta\in\Theta_B}\|b(x,\beta)-b(y,\beta)\|\le C\|x-y\|.
\end{align*}
%%%
\item   
$\displaystyle\sup_t\EE_{\theta}\left[\|X_t\|^k\right]<\infty$
for all $k\ge0$ and $\theta\in\Theta$.
%%%
\item $\displaystyle\inf_{x,\alpha}\det \left(A(x,\alpha)\right)>0$. 
%%%
\item 
$a\in C^{4,2}_{\uparrow}(\mathbb R^d\times\Theta_A)$
and
$b\in C^{4,2}_{\uparrow}(\mathbb R^d\times\Theta_B)$.
%%%
\item 
The solution of \eqref{eq2.1} is ergodic with its invariant measure $\mu_\theta$ 
such that 
\begin{align*}
\int_{\mathbb R^d}\|x\|^k\dd\mu_\theta(x)<\infty\quad\text{for all }\ k\ge0 
\text{ and } \theta\in\Theta.
\end{align*}
\end{enumerate}
%%%%%%%%%%%%%%%%%%%%%%%%%%%%%%%%%%%%%%%%%%%%%%%%%%%%%%%%%%%%%%%%%%%%%%%%%%%%%%%%%%%%%%
\subsection{Test for the diffusion parameter}
\ 

We consider the following testing problem: 
\begin{align*}
\begin{cases}
H_0^\alpha: \alpha_0 \text{ does not change over } 0\le t\le nh_n, \\
H_1^\alpha:\text{not } H_0^\alpha.
\end{cases}
\end{align*}
Now, we additionally assume the following condition:
\begin{enumerate}
\renewcommand{\labelenumi}{{\textbf{[A\arabic{enumi}]}}}
\setcounter{enumi}{5}
\item Under $H_0^\alpha$, there exists an estimator $\hat\alpha_n$ such that 
\begin{align*}
\sqrt n(\hat\alpha_n-\alpha_0)=O_p(1).
\end{align*}
\end{enumerate}
%Let
%\begin{align*}
%A=aa^\TT,\quad
%\DeX= X_{t_i^n}-X_{t_{i-1}^n},\quad
%(\DeX)^{\otimes2}=(\DeX)(\DeX)^\TT,
%\end{align*}
%and
Let
\begin{align*}
U_n^{(1)}(\alpha)=-\frac12\sum_{i=1}^n
\left({\mathrm{tr}}\left(
A^{-1}(X_{t_{i-1}^n},\alpha)\frac{(\Delta X_i)^{\otimes2}}{h_n}
\right)
+\log \det A(X_{t_{i-1}^n},\alpha)\right).
\end{align*}
Note that
\begin{align*}
\partial_{\alpha^j} U_n^{(1)}(\alpha)=\frac12\sum_{i=1}^n
\tr\left(
A^{-1}(\Xs,\alpha)\partial_{\alpha^j}A(\Xs,\alpha)
\left(
A^{-1}(\Xs,\alpha)\frac{(\DeX)^{\otimes2}}{h_n}-I_d
\right)\right)
\end{align*}
and $\displaystyle\frac{1}{\sqrt n}\partial_\alpha U_n^{(1)}(\alpha_0)$ 
has asymptotic normality, let 
\begin{align*}
\hat\eta_i
&=\tr\left(A^{-1}(\Xs,\hat\alpha_n)
\frac{(\Delta X_i)^{\otimes2}}{h_n}\right)
=\sum_{\ell_1,\ell_2=1}^d
\left(A^{-1}(\Xs,\hat\alpha_n)\right)^{\ell_1,\ell_2}
\frac{(\DeX)^{\ell_1}(\DeX)^{\ell_2}}{h_n}.
\end{align*}
The test statistic for $\alpha$ is as follows:
\begin{align*}
T_n^\alpha=\frac{1}{\sqrt{2dn}}\max_{1\le k\le n}
\left|\sum_{i=1}^k\hat\eta_i-\frac{k}{n}\sum_{i=1}^n\hat\eta_i\right|.
\end{align*}
For $k=1,2, \ldots$, 
let $\boldsymbol B_k^0$ be a $k$-dimensional standard Brownian bridge.
%%%%%%%%%%%%%%%%%%%%%%%%%%%%%%%%%%%%%%%%%%%%%%%%%%%%%%%%%%%%%%%
\begin{thm}\label{th1}
Suppose that
{\textbf{[A1]}}-{\textbf{[A6]}} hold. Then under $H_0^\alpha$, as $nh_n^2\lto0$,
\begin{align*}
T_n^\alpha\dto \sup_{0\le s\le 1}| \boldsymbol B_1^0(s) |.
\end{align*}
%where $B^0$ is a $1$-dimensional standard Brownian bridge.
\end{thm}
%%%%%%%%%%%%%%%%
\begin{rmk}
The test statistic $T_n^\alpha$ is a multidimensional version of the test statistic proposed in Remark
on page 842 of Lee (2011).
% without  the proof of the asymptotic properties of the test statistics. 
The asymptotic distribution of the test statistic $T_n^\alpha$ is shown in Theorem \ref{th1}.
Moreover, it follows from Theorem \ref{th4} below
that the test is consistent.
\end{rmk}
%%%%%%%%%%%%%%%%%%%%%%%%%%%%%%%%%%%%%%%%%%%%%%%%%%%%%%%%%%%%%%%%%%%%%%%%%%%%%%%
\subsection{Test for the drift parameter}
\ 

We make the following assumption:
\begin{enumerate}
\renewcommand{\labelenumi}{{\textbf{[A\arabic{enumi}]}}}
\setcounter{enumi}{6}
\item $\alpha_0$ does not change over $0\le t\le nh_n$. 
\end{enumerate}
Then we consider the following testing problem:
\begin{align*}
\begin{cases}
H_0^\beta:\beta_0 \text{ does not change over } 0\le t\le nh_n,\\
H_1^\beta:\text{ not } H_0^\beta.
\end{cases}
\end{align*}
%%%
Now, we additionally assume the following condition:
\begin{enumerate}
\renewcommand{\labelenumi}{{\textbf{[A\arabic{enumi}]}}}
\setcounter{enumi}{7}
\item Under $H_0^\beta$, there exists an estimator $(\hat\alpha_n,\hat\beta_n)$
such that  
\begin{align*}
\sqrt n(\hat\alpha_n-\alpha_0)=O_p(1),\quad\sqrt{nh_n}(\hat\beta_n-\beta_0)=O_p(1).
\end{align*}
\end{enumerate}
%%%
Let
\begin{align*}
U_n^{(2)}(\beta|\alpha)=-\frac12\sum_{i=1}^n
\frac{1}{h_n}
\tr\left(
A^{-1}(\Xs,\alpha)
\left(\Xt-\Xs-h_nb(\Xs,\beta)\right)^{\otimes2}
\right).
\end{align*}
Note that 
\begin{align*}
\partial_{\beta^j} U_n^{(2)}(\beta|\alpha)=\sum_{i=1}^n
\partial_{\beta^j} b(\Xs,\beta)^\TT A^{-1}(\Xs,\alpha)
\left(\Xt-\Xs-h_nb(\Xs,\beta)\right),
\end{align*}
and $\displaystyle\frac{1}{\sqrt{nh_n}}\partial_\beta U_n^{(2)}(\beta_0|\alpha_0)$
has asymptotic normality, let
\begin{align*}
\hat\xi_i=
1_d^\TT a^{-1}(\Xs,\hat\alpha_n)
\left(\Xt-\Xs-h_nb(\Xs,\hat\beta_n)\right).
\end{align*}
The test statistic for $\beta$ is as follows:
\begin{align*}
T_{1,n}^\beta=\frac{1}{\sqrt{dnh_n}}\max_{1\le k\le n}
\left|\sum_{i=1}^k\hat\xi_i-\frac{k}{n}\sum_{i=1}^n\hat\xi_i\right|.
\end{align*}
%%%%%%%%%%%%%%%%%%%%%%%%%%%%%%%%
\begin{thm}\label{th2}
Suppose that {\textbf{[A1]}}-{\textbf{[A5]}}, {\textbf{[A7]}} and {\textbf{[A8]}} hold.
Then under $H_0^\beta$, as $nh_n^2\lto0$,
\begin{align*}
T_{1,n}^\beta\dto \sup_{0\le s\le 1}| \boldsymbol B_1^0(s) |.
\end{align*}
\end{thm}
%%%%%%%%%%%%%%%%%%%%%%%%%%%%%%%%
%\begin{rmk}
$T_{1,n}^\beta$ is the simple test statistic, 
but the test may not be consistent, see Section 3 below. 
In this case, we consider another test statistic.
%\end{rmk}
%%%%%%%%%%
Let
\begin{align*}
\hat\zeta_i&=\partial_\beta b(\Xs,\hat\beta_n)^\TT A^{-1}(\Xs,\hat\alpha_n)
\left(\Xt-\Xs-h_nb(\Xs,\hat\beta_n)\right),\\
\mathcal I_n
&=\frac1n\sum_{i=1}^n
\partial_\beta b(\Xs,\hat\beta_n)^\TT A^{-1}(\Xs,\hat\alpha_n)
\partial_\beta b(\Xs,\hat\beta_n).
\end{align*}
Note that $\mathcal I_n$ is a consistent estimator of 
\begin{align*}
\mathcal I:=\int_{\mathbb R^d}
\partial_\beta b(x,\beta_0)^\TT A^{-1}(x,\alpha_0)
\partial_\beta b(x,\beta_0)\dd\mu_{\theta_0}(x).
\end{align*}
%%%
The test statistic for $\beta$ is as follows:
\begin{align*}
T_{2,n}^\beta=\frac{1}{\sqrt{nh_n}}\max_{1\le k\le n}
\left\|\mathcal I_n^{-1/2}\left(\sum_{i=1}^k\hat\zeta_i-\frac kn
\sum_{i=1}^n\hat\zeta_i\right)\right\|.
\end{align*}
We additionally assume the following condition:
\begin{enumerate}
\renewcommand{\labelenumi}{{\textbf{[A\arabic{enumi}]}}}
\setcounter{enumi}{8}
\item There exists an integer $m\ge 3$ such that
$nh_n^{m/(m-1)}\lto\infty$ and $b\in C^{4,m+1}_{\uparrow}(\mathbb R^d\times\Theta_B)$.
\end{enumerate}
\noindent
%{ 
%For $k=1,2, \ldots$, 
%let $\boldsymbol B_k^0$ be a $k$-dimensional standard Brownian bridge.
%}
%%
%%%%%%%%%%%%%%%%%%%%%%%%%%%%%%%%%
\begin{thm}\label{th03}
Suppose that {\textbf{[A1]}}-{\textbf{[A5]}}, {\textbf{[A7]}}, {\textbf{[A8]}} 
and {\textbf{[A9]}} hold.
Then under $H_0^\beta$, as $nh_n^2\lto0$,
\begin{align*}
T_{2,n}^\beta\dto \sup_{0\le s\le1}\| { \boldsymbol B_q^0(s)} \|.
\end{align*}
\end{thm}
%%%%%%%%%%%%%%%%%
%\begin{rmk}
%We think that the 7th line  
%\begin{align*}
%\sup_{u\in[0,1]}\sup_{\theta\in\Theta}
%\left\|
%\frac1{t_n^n}\ddot{\mathbb M}{}_n^A(u,\theta)
%-\frac1{t_n^n}\sum_{k:t_{k-1}^n\le ut_n^k}H^A(X_{t_{k-1}^n};\theta_0,\theta)
%\right\|
%=o_p(1)
%\end{align*}
%from the bottom of the 240 page of Negri and Nishiyama (2017) 
%is not always satisfied as $n\Delta_n^2\lto0$, that is, 
%Theorem 1 (i) is not always satisfied.
%\end{rmk}
%%%%%%%%%%%%%%%%%%%%%%%%%%%%%%%%%%%%%%%%%%%%%%%%%%%%%%%%%%%%%%%%%%%%%%%%%%%%%%%%%%%%%%%%
We assume the following condition instead of {\textbf{[A9]}}:
\begin{enumerate}
\renewcommand{\labelenumi}{{\textbf{[A\arabic{enumi}']}}}
\setcounter{enumi}{8}
\item 
There exists an integer $M \ge1$ such that 
$b\in C_{\uparrow}^{4,M+1}(\mathbb R^d\times\Theta_B)$ and
$\partial_{\beta^{\ell_{M+1}}}\cdots\partial_{\beta^{\ell_1}}b(x,\beta)=0$ for 
$1\le \ell_1,\ldots,\ell_{M+1}\le q$.
\end{enumerate}
%%%%%%%%%%%%%%%%%%%%%%%%%%%%%%%%%
\begin{cor}\label{th3}
Suppose that {\textbf{[A1]}}-{\textbf{[A5]}}, {\textbf{[A7]}}, {\textbf{[A8]}} 
and {\textbf{[A9']}} hold.
Then under $H_0^\beta$, as $nh_n^2\lto0$,
\begin{align*}
T_{2,n}^\beta\dto \sup_{0\le s\le1}\| \boldsymbol B_q^0(s) \|.
\end{align*}
\end{cor}
%%%%%%%%%%%%%%%%%%%%%%%%%%%
\begin{rmk}
Since $h_n \lto 0$, $nh_n \lto \infty$ and $nh_n^2 \lto 0$,
if $h_n=O(n^{-\alpha})$ for some $\alpha \in (\frac{1}{2}, 1)$,
then there exists an integer $m >\frac{1}{1 - \alpha}>2$ such that
$n h_n^{m/(m-1)} \lto \infty$.
Therefore, 
in  {\textbf{[A9]}} of Theorem \ref{th03}, 
we make the assumption of the smoothness of the drift coefficient 
$b$ with respect to $\beta$ up to order $m+1$ $(\geq 4)$
when the drift coefficient $b$ is general. 
In particular, if $b \in C_{\uparrow}^{4, \infty}(\mathbb R^d, \Theta_B)$, 
then {\textbf{[A9]}} is satisfied.
%there is a restriction on how to take $\{h_n\}$. 
In contrast, Corollary \ref{th3} indicates that 
when the drift coefficient $b(x,\beta)=-\beta x$,
we set $M=1$ in  {\textbf{[A9']}} and
there is no need to assume $m \geq 3$ in {\textbf{[A9]}}.
\end{rmk}
%%%%%%%%%%%%%%%%%%%%%%%%%%%%%%%%%%%%%%%
%\begin{rmk}
%In the above, we set $\{W_t\}$ as a $d$-dimensional Wiener process 
%and $a$ as $\mathbb R^d\otimes\mathbb R^d$-valued function on $\mathbb R^d\times\Theta_A$. 
%More generally, if $\{W_t\}$ is an $r$-dimensional 
%and $a$ is $\mathbb R^d\otimes\mathbb R^r$-valued function, the same argument can be made.
%However, it is necessary to assume 
%\begin{enumerate}
%\renewcommand{\labelenumi}{{\textbf{[A\arabic{enumi}']}}}
%\setcounter{enumi}{2}
%\item $\displaystyle\inf_{x,\alpha}\det \left(aa^\TT(x,\alpha)\right)>0$ and
%$a(x,\alpha)1_r\neq0$ for all $x\in\mathbb R^d,\alpha\in\Theta_A$, 
%where $r\ge d$,
%\end{enumerate}
%when constructing $\hat\xi_i$. 
%Then, we set 
%\begin{align*}
%\hat\xi_i
%&=
%\frac{1_r^\TT a^\TT A^{-1}(\Xs,\hat\alpha_n)(\Xt-\Xs-h_nb(\Xs,\hat\beta_n))}
%{\sqrt{1_r^\TT a^\TT A^{-1}a(\Xs,\hat\alpha_n)1_r^\TT}},\\
%T_{1,n}^\beta
%&=
%\frac{1}{\sqrt{nh_n}}\max_{1\le k\le n}
%\left|
%\sum_{i=1}^k\hat\xi_i-\frac kn\sum_{i=1}^n\hat\xi_i
%\right|.
%\end{align*}
%\end{rmk}
%%%%%%%%%%%%%%%%%%%%%%%%%%%%%%%%%%%%%%%%%%%%%%%%%%%%%%%%%%%%%%%%
\subsection{The powers of tests}
\

First, we consider the power of the test for the diffusion parameter $\alpha$,
that is, the following testing problem:
$H_0^\alpha: \alpha_0$ does not change over $0\le t\le nh_n$, 
$H_1^\alpha:$ There exists $0< t^*<1$ such that 
\begin{align*}
\alpha_0=
\begin{cases}
\alpha_1^*,&0\le t\le [nt^*]h_n,\\
\alpha_2^*,&[nt^*]h_n< t\le nh_n,
\end{cases}
\end{align*}
where $\alpha_1^*\neq\alpha_2^*$.
%%%%%%%%
Now, we assume the following conditions:
\renewcommand{\labelenumi}{{\textbf{[B\arabic{enumi}]}}}
\begin{enumerate}
\item Under $H_1^\alpha$, there exist
$\bar\alpha_*\in\Theta_A$ and 
{ an estimator} $\hat\alpha_n$
%of $\alpha$ 
such that  
\begin{align*}
\hat\alpha_n-\bar\alpha_*=o_p(1).
\end{align*}
\end{enumerate}
{ Set}
\begin{align*}
F(\alpha)=\int_{\mathbb R^d}
\tr\left(
A^{-1}(x,\bar\alpha_*)A(x,\alpha)
\right)\dd\mu_{\alpha}(x).
\end{align*}
%%%%%%%%%%
\begin{enumerate}
\setcounter{enumi}{1}
\item
$F(\alpha_1^*)\neq F(\alpha_2^*)$ under $H_1^\alpha$.
\end{enumerate}
%%%
%{ Let $\boldsymbol B^0$ be a $k$-dimensional standard Brownian bridge.}
Let $0<\epsilon<1$ and
$w_k(\epsilon)$ denote the upper-$\epsilon$ point  of 
${\displaystyle \sup_{0\le s\le 1}\|\boldsymbol B_k^0(s)\|}$, that is,
${\displaystyle
P(\sup_{0\le s\le 1}\|\boldsymbol B_k^0(s)\|>w_k(\epsilon))=\epsilon}$.
%%%%%
\begin{thm}\label{th4}
Assume {\textbf{[A1]}}-{\textbf{[A4]}}, {\textbf{[B1]}} and {\textbf{[B2]}}. 
Then, under $H_1^\alpha$,
\begin{align*}
P\left(T_n^\alpha>w_1(\epsilon)\right)\lto1.
\end{align*}
Hence, if {\textbf{[B1]}} and {\textbf{[B2]}} are satisfied, 
then the test $T_n^\alpha$ is consistent.
\end{thm}

%%%%%%%%%%%%%%%%%%%%%%%%%%%%%%%%%%
In the following, we assume {\textbf{[A7]}}.
We consider the power of the test for the drift parameter $\beta$, that is, 
the following testing problem:
%\begin{center}
$H_0^\beta: \beta_0$ does not change over $0\le t\le nh_n$, 
$H_1^\beta:$ There exists $0< t^*<1$ such that 
\begin{align*}
\beta_0=\begin{cases}
\beta_1^*,&0\le t\le [nt^*]h_n,\\
\beta_2^*,&[nt^*]h_n< t\le nh_n,
\end{cases}
\end{align*}
where $\beta_1^*\neq\beta_2^*$.
%%%%%%%%
Now, we assume the following conditions:
\begin{enumerate}
\setcounter{enumi}{2}
\item Under $H_1^\beta$, there exist $\bar\beta_*\in\Theta_B$ and 
estimators $\hat\alpha_n$ and $\hat\beta_n$
such that  
\begin{align*}
\sqrt n(\hat\alpha_n-\alpha_0)=O_p(1),\quad
\sqrt{nh_n} (\hat\beta_n-\bar\beta_*)=O_p(1).
\end{align*}
\end{enumerate}
%%%%%%%%%%
\begin{enumerate}
\renewcommand{\labelenumi}{{\textbf{[B\arabic{enumi}']}}}
\setcounter{enumi}{2}
\item Under $H_1^\beta$, there exist $\bar\beta_*\in\Theta_B$ and 
estimators $\hat\alpha_n$ and $\hat\beta_n$
such that  
\begin{align*}
\sqrt n(\hat\alpha_n-\alpha_0)=O_p(1),\quad
\hat\beta_n-\bar\beta_*=o_p(1).
\end{align*}
\end{enumerate}

%%%%%%%%%%

Let
\begin{align*}
G(\beta)&=\int_{\mathbb R^d}1_d^\TT a^{-1}(x,\alpha_0)
(b(x,\beta)-b(x,\bar\beta_*))\dd\mu_{(\alpha_0,\beta)}(x),\\
H(\beta)&=\int_{\mathbb R^d}\partial_\beta(x,\bar\beta_*)^\TT A^{-1}(x,\alpha_0)
(b(x,\beta)-b(x,\bar\beta_*))\dd\mu_{(\alpha_0,\beta)}(x).
\end{align*}
%%%%%
\begin{enumerate}
\setcounter{enumi}{3}
\item $G(\beta_1^*)\neq G(\beta_2^*)$ under $H_1^\beta$. 
\item $H(\beta_1^*)\neq H(\beta_2^*)$ under $H_1^\beta$. 
\end{enumerate}
%%%%%
\begin{thm}\label{th5}
Assume {\textbf{[A1]}}-{\textbf{[A4]}}, {\textbf{[A7]}}, 
{\textbf{[B3']}} and {\textbf{[B4]}}. 
Then, under $H_1^\beta$,
\begin{align*}
P\left(T_{1,n}^\beta>w_1(\epsilon)\right)\lto1.
\end{align*}
\end{thm}
%%%%%%%%%%
\begin{thm}\label{th6}
Assume {\textbf{[A1]}}-{\textbf{[A4]}}, {\textbf{[A7]}}, {\textbf{[A9]}}, 
{\textbf{[B3]}} and {\textbf{[B5]}}. 
Then, under $H_1^\beta$,
\begin{align*}
P\left(T_{2,n}^\beta>w_q(\epsilon)\right)\lto1.
\end{align*}
\end{thm}
%%%%%%%%%%%%%%%
\begin{cor}\label{th7}
Assume {\textbf{[A1]}}-{\textbf{[A4]}}, {\textbf{[A7]}}, {\textbf{[A9']}}, 
{\textbf{[B3']}} and {\textbf{[B5]}}. 
Then, under $H_1^\beta$,
\begin{align*}
P\left(T_{2,n}^\beta>w_q(\epsilon)\right)\lto1.
\end{align*}
\end{cor}
%%%%%%%%%%%%%%%
\begin{rmk}
Theorems \ref{th5}, \ref{th6} and Corollary \ref{th7} indicate that the tests 
$T_{1,n}^\beta$ and $T_{2,n}^\beta$ are consistent under 
some regularity conditions. 
Theorem \ref{th6} and Corollary \ref{th7} indicate that the test $T_{2,n}^\beta$
is consistent when the drift coefficient $b$ satisfies 
\textbf{[A9]} and \textbf{[A9']}, respectively. 
In particular,
for the drift coefficient $b$ satisfying \textbf{[A9]},
it is necessary to assume $\sqrt{nh_n}(\hat\beta_n-\beta_*)=O_p(1)$ 
in \textbf{[B3]}. 
When the drift coefficient $b$ satisfies \textbf{[A9']}, it is enough to assume 
$\hat\beta_n-\beta_*=o_p(1)$ in \textbf{[B3']}.
For the maximum likelihood type estimators of the misspecified diffusion processes,
see Uchida and Yoshida (2011). 
\end{rmk}
%%%%%%%%%%%%
\begin{rmk}
In Theorems \ref{th5}, \ref{th6} and Corollary \ref{th7}, 
%if {\textbf{[B3]}} is assumed, 
it is not necessary to assume $nh_n^2\lto0$, 
but it is required to construct 
estimators $\hat\alpha_n, \hat\beta_n$ 
that satisfies $\sqrt n(\hat\alpha_n-\alpha_0)=O_p(1)$ and
$\sqrt{nh_n}(\hat\beta_n-\beta_*)=O_p(1)$ in {\textbf{[B3]}},
or 
estimator $\hat\alpha_n$ 
that satisfies $\sqrt n(\hat\alpha_n-\alpha_0)=O_p(1)$ in {\textbf{[B3']}}.
\end{rmk}

\section{Examples and simulation results}
In this section, we evaluate the asymptotic performance of the test statistics 
$T_n^\alpha$, $T_{1,n}^\beta$ and $T_{2,n}^\beta$ through a simulation study.
In this study, we consider the $1$-dimensional Ornstein-Uhlenbeck process (OU process):
\begin{align}\label{OU}
\dd X_t=-\beta(X_t-\gamma)\dd t+\alpha\dd W_t,\quad X_0=1.
\qquad(\alpha, \beta>0, \gamma\in\mathbb R)
\end{align}
The data $\{\Xt\}_{i=0}^n$ are obtained discretely with sampling interval $h_n=n^{-2/3}$,
so that $nh_n=n^{1/3}\lto\infty$. 

Song and Lee (2009), Lee (2011) and Song (2020) 
conduct the simulations for the test statistic of the diffusion coefficient. 
On the other hand, we also perform the simulations for the test statistic of the drift coefficient 
as well as the diffusion coefficient.
In addition, Negri and Nishiyama (2017) perform the simulations for the test statistic of 
the diffusion and drift parameters at the same time. 
In our case, however, we first test the diffusion parameter, 
and if we determine that there is no change in the diffusion parameter, 
then we test the drift parameter.
They are the adaptive tests for the changes of both the drift and the diffusion parameters.
As far as we know, this is the first numerical simulations of the adaptive tests for the change point problem
of an ergodic diffusion process.
%the change of the diffusion parameter first, and then the drift parameter.
For details of the adaptive tests based on the Wald type test statistics and the Rao type ones
and the likelihood ratio type ones for ergodic diffusion processes, 
see Kitagawa and Uchida (2014).
%%%%%%%%%%%

Under the above setting with
 $a(x,\alpha)=\alpha$ and
$b(x,\boldsymbol\beta)=-\beta(x-\gamma)$ for $\boldsymbol\beta=(\beta,\gamma)$, 
we have
\begin{align*}
\hat\eta_i&=\frac{(\Xt-\Xs)^2}{h_n\hat\alpha_n^2},\\
T_n^{\alpha}&=\frac{1}{\sqrt{2n}}\max_{1\le k\le n}
\left|\sum_{i=1}^k\hat\eta_i-\frac{k}{n}\sum_{i=1}^n\hat\eta_i\right|,\\
%%%%
\hat\xi_i&=\frac{\Xt-\Xs+h_n\hat\beta_n(\Xs-\hat\gamma_n)}{\hat\alpha_n},\\
T_{1,n}^\beta&=\frac{1}{\sqrt{nh_n}}\max_{1\le k\le n}
\left|\sum_{i=1}^k\hat\xi_i-\frac{k}{n}\sum_{i=1}^n\hat\xi_i\right|,\\
\hat\zeta_i
&=
\left(\begin{array}{@{\,}c@{\,}}
	-(\Xs-\hat\gamma_n) \\
	\hat\beta_n
\end{array}\right)
\frac{\Xt-\Xs+h_n\hat\beta_{n}(\Xs-\hat\gamma_n)}{\hat\alpha_n^2},\\
\mathcal I_n
&=
\frac{1}{n\hat\alpha_n^2}\sum_{i=1}^n
\left(\begin{array}{@{\,}cc@{\,}}
	(\Xs-\hat\gamma_n)^2 & b(\Xs,\hat{\boldsymbol\beta}_n) \\
	b(\Xs,\hat{\boldsymbol\beta}_n)  & \hat\beta_n^2
\end{array}\right),\\
T_{2,n}^\beta&=\frac{1}{\sqrt{nh_n}}\max_{1\le k\le n}
\left\|\mathcal I_n^{-1/2}
\left(\sum_{i=1}^k\hat\zeta_i-\frac{k}{n}\sum_{i=1}^n\hat\zeta_i\right)
\right\|,
\end{align*}
%%%%%%%%%%
where $(\hat\alpha_n,\hat{\boldsymbol\beta}_n)$ is an estimator satisfying the assumption 
{\textbf{[A6]}} or {\textbf{[A8]}}, and we construct the estimator from 
\begin{align*}
U_n^{(1)}(\alpha)=-\frac12\sum_{i=1}^n\frac{(\Xt-\Xs)^2}{h_n\alpha^2}-n\log\alpha,\quad
U_n^{(2)}(\boldsymbol\beta|\alpha)
=-\frac12\sum_{i=1}^n\frac{(\Xt-\Xs+h_n\beta(\Xs-\gamma))^2}{h_n\alpha^2},
\end{align*}
\begin{align*}
\hat\alpha_n=\arg\sup_{\alpha}U_n^{(1)}(\alpha),\quad
\hat{\boldsymbol\beta}_n=\arg\sup_{\boldsymbol\beta}U_n^{(2)}(\boldsymbol\beta|\hat\alpha_n),
\end{align*}
see, for example, Kessler (1995, 1997), Uchida and Yoshida (2012) and Kaino and Uchida (2018).
We have, from Theorems \ref{th1}, \ref{th2} and Corollary \ref{th3},
\begin{align}\label{simBB}
T_n^\alpha\dto\sup_{0\le s\le 1}|{ \boldsymbol B_1^0(s)}|,\quad
T_{1,n}^\beta\dto\sup_{0\le s\le 1}|{ \boldsymbol B_1^0(s)}|,\quad
T_{2,n}^\beta\dto\sup_{0\le s\le 1}\| { \boldsymbol B_2^0(s)} \|.
\end{align}
%where $B^0,\boldsymbol B^0$ are 1-dimensional and 2-dimensional standard Brownian bridges,
%respectively.
%%%%%%%%%%%%%%%%%%%%%%%%%%%%%%%%%%%%%%%%%

%%%%%%%%%%%%%%%%%%%%%%%%%%%%%%%%%%%%%%%%%
In order to evaluate the empirical sizes and powers of $T_n^\alpha,T_{1,n}^\beta$ and $T_{2,n}^\beta$, 
we consider three cases as follows:
\begin{description}
\item[Case 1] Neither parameter changes.
\item[Case 2] The diffusion parameter $\alpha$ changes. 
\item[Case 3] The diffusion parameter $\alpha$ does not change, 
but the drift parameter $(\beta,\gamma)$ changes.
\end{description}
%%%%%%%%%%%%%%%%%%%%%%%%%
For each of the above cases, we change parameters at $nh_n/2$.
The empirical sizes and powers are calculated as the rejection number of the null 
hypothesis out of 1,000 repetitions.
All simulations are conducted at significance level $10\%$ and 
the corresponding critical values are  
obtained from the following: 
the standard Brownian bridge is generated by taking $10^4$ points 
on the interval $[0,1]$, 
and the maximum value of its norm is recorded. 
This is repeated $10^4$ times. As a result, we have
\begin{align*}
P\left(\sup_{0\le s\le 1}| { \boldsymbol B_1^0(s)} |> 1.223\right)=0.1,\quad
P\left(\sup_{0\le s\le 1}\| { \boldsymbol B_2^0(s)} \|> 1.444\right)=0.1,
\end{align*}
i.e., the corresponding critical values are 1.223 and 1.444 for 
the 
1-dimensional and 2-dimensional standard Brownian bridges, respectively.

In Case 1, we confirm that \eqref{simBB} holds with various parameters. 
Next, we consider the empirical sizes of $T_n^\alpha$, 
$T_{1,n}^\beta$ and $T_{2,n}^\beta$.
The empirical sizes are calculated for the OU process with 
$(\alpha,\beta,\gamma)\in\{(1,1,1), (0.5,1,0)$,\\ $(1.5,1.5,-1), (2,3,0.5)\}$.
%%%

In (i) of Case 2, we confirm that  $T_n^\alpha$ is consistent, that is, 
%we consider 
the power of $T_n^\alpha$ goes to $1$ under the following situation:
\begin{enumerate}
\renewcommand{\labelenumi}{(\roman{enumi})}
\item 
the only parameter $\alpha$ changes from $\alpha_1^*=1$ to $\alpha_2^*\in\{1.01, 1.05, 1.1, 1.5\}$.
\end{enumerate} 
Lee (2011) also  performed this simulation, 
but we consider the cases where the parameter change is small 
and the number of data is large.
%%%

In Case 3, after confirming that $T_n^\alpha\dto \sup_{0\le s\le 1}|\boldsymbol B_1^0(s)|$, 
it is verified whether $T_{1,n}^\beta$ and $T_{2,n}^\beta$ 
%have consistency.
are consistent.
We consider the empirical size of $T_n^\alpha$ and 
the powers of $T_{1,n}^\beta$ and $T_{2,n}^\beta$ under the following situations:
\renewcommand{\labelenumi}{(\roman{enumi})}
\begin{enumerate}
\item 
the parameter $\beta$ only changes from $\beta_1^*=1$ to $\beta_2^*\in\{1.1, 1.5, 3, 5\}$;
\item 
the parameter $\gamma$ only changes from $\gamma_1^*=1$ to $\gamma_2^*\in\{0.9, 0.5, 0, -1\}$;
\item 
the parameter $(\beta,\gamma)$ changes from $(\beta_1^*,\gamma_1^*)=(1,1)$ to 
$(\beta_2^*,\gamma_2^*)\in\{(3,0.5), (3,0), (5,0.5), (5,0)\}$.
\end{enumerate}
$\beta$ and $\gamma$ in OU process \eqref{OU} 
represent the speed of reversion and the mean value as $nh_n \rightarrow \infty$, respectively. 
In (i) and (ii), we verify the consistency of the test 
when each parameter changes, and 
in (iii), we also verify the consistency of the test 
when both $\beta$ and $\gamma$ change.

%%%%%%%%%%%%%%%%%%%%%%%%%%%%%%%%%%%%
First, let us consider the consistency of the test based on $T_n^\alpha$. 
The invariant measure of the solution in \eqref{OU} is 
$\mu_\theta\sim N(\gamma,\frac{\alpha^2}{2\beta})$, $\theta=(\alpha,\beta,\gamma)$, 
and we have
%%%%%%%%%%%%
\begin{align*}
F(\alpha)=\int_{\mathbb R}\frac{\alpha^2}{\bar\alpha_*^2}\dd\mu_\alpha(x)
=\frac{\alpha^2}{\bar\alpha_*^2}.
\end{align*}
Therefore, noting that $\alpha>0$, 
$F(\alpha_1^*)\neq F(\alpha_2^*)$ for $\alpha_1^*\neq\alpha_2^*$,
and the test $T_n^\alpha$ has consistency according to Theorem \ref{th4}.

%%%%%%%%%%%
Next, we investigate  the consistency of the test $T_{1,n}^\beta$. 
For the drift parameter $\boldsymbol\beta_1^*=(\beta_1^*,\gamma_1^*)$, we have
\begin{align*}
G(\boldsymbol\beta_1^*)
=\int_\mathbb R
\frac{1}{\alpha_0}(b(x,\boldsymbol\beta_1^*)-b(x,\bar{\boldsymbol\beta}_*))\dd\mu_{\theta_1^*}(x)
&=\int_\mathbb R
\frac{1}{\alpha_0}\left(
-(\beta_1^*-\bar\beta_*)x+(\beta_1^*\gamma_1^*-\bar\beta_*\bar\gamma_*)
\right)
\dd\mu_{\theta_1^*}(x)\\
&=\frac{1}{\alpha_0}(
-\gamma_1^*(\beta_1^*-\bar\beta_*)+(\beta_1^*\gamma_1^*-\bar\beta_*\bar\gamma_*)
)\\
&=\frac{\bar\beta_*}{\alpha_0}
(\gamma_1^*-\bar\gamma_*),
\end{align*}
where $\theta_1^*=(\alpha_0,\beta_1^*,\gamma_1^*)$.
%%%%%%
%%%%%%%%%%%%%%%%%%%%%%%%%%%%%%%%%%%%%%%%%%%%%%%%%%%%%%%%%%%%%%%%%%%%%%%%%%%%%%%%%%%%%%%%%%%

We construct the estimators $\hat\alpha_n$ { and} $\hat{\boldsymbol\beta}_n$ 
in the same way as above.
Since
\begin{align*}
\frac{1}{nh_n}
\left(U_n^{(2)}(\boldsymbol\beta|\hat\alpha_n)
-U_n^{(2)}(\boldsymbol\beta_0|\hat\alpha_n)\right)
\pto
-\frac12(t^*Q(\boldsymbol\beta,\boldsymbol\beta_1^*)
+(1-t^*)Q(\boldsymbol\beta,\boldsymbol\beta_2^*))
=:\bar Q(\boldsymbol\beta),
\end{align*}
%%%
where
\begin{align*}
Q(\boldsymbol\beta,\boldsymbol\beta_1^*)
=\int_{\mathbb R}\frac{(b(x,\boldsymbol\beta)-b(x,\boldsymbol\beta_1^*))^2}
{\alpha_0^2}\dd\mu_{\theta_1^*}(x)
=\int_{\mathbb R}\frac{((-\beta+\beta_1^*)x+(\beta\gamma-\beta_1^*\gamma_1^*))^2}
{\alpha_0^2}\dd\mu_{\theta_1^*}(x),
\end{align*}
%%%%%%%%%%%
if we set
$\bar{\boldsymbol\beta}_*=\arg\sup_{\boldsymbol\beta} \bar Q(\boldsymbol\beta)$,
then 
$\hat{\boldsymbol\beta}_n\pto\bar{\boldsymbol\beta}_*$ under $H_1^\beta$.
%%%%%%%%%%%%%%%%%%%%%%%%%%%%%%%%%%%%%%%%%%%%%%%%%%%%%%%%%%%%%%%%%%%%%%%%%%%%%%%

Since 
$\partial_{\boldsymbol\beta}\bar Q(\bar{\boldsymbol\beta}_*)=0$,
\begin{align*}
\partial_\gamma Q(\boldsymbol\beta,\boldsymbol\beta_1^*)
=\frac{2\beta}{\alpha_0^2}
\int_{\mathbb R}((-\beta+\beta_1^*)x+(\beta\gamma-\beta_1^*\gamma_1^*))
\dd\mu_{\theta_1^*}
=\frac{2\beta^2}{\alpha_0^2}(\gamma-\gamma_1^*),
\end{align*}
\begin{align*}
\partial_\gamma\bar Q(\boldsymbol\beta)
=t^*\partial_\gamma Q(\boldsymbol\beta,\boldsymbol\beta_1^*)
+(1-t^*)\partial_\gamma Q(\boldsymbol\beta,\boldsymbol\beta_2^*)
&=t^*\frac{2\beta^2}{\alpha_0^2}(\gamma-\gamma_1^*)
+(1-t^*)\frac{2\beta^2}{\alpha_0^2}(\gamma-\gamma_2^*)\\
&=\frac{2\beta^2}{\alpha_0^2}(\gamma-(t^*\gamma_1^*+(1-t^*)\gamma_2^*)),
\end{align*}
%%%%%%%%%%%%%%%%%%%%%%%%%%%%%%%%%%%%%%%%%%%%%%%%
\begin{align*}
\partial_\beta \bar Q(\boldsymbol\beta,\boldsymbol\beta_1^*)
=\frac{2}{\alpha_0^2}\int_{\mathbb R}(-x+\gamma)
((-\beta+\beta_1^*)x+(\beta\gamma-\beta_1^*\gamma_1^*))\dd\mu_{\theta_1^*}(x)
=\frac{2}{\alpha_0^2}
\left(
\beta\left(\frac{\alpha_0^2}{2\beta_1^*}+(\gamma-\gamma_1^*)^2\right)
-\frac{\alpha_0^2}{2}
\right),
\end{align*}
%%%%%%%%%%%%%%
and
\begin{align*}
\partial_\beta\bar Q(\boldsymbol\beta)
&=t^*\partial_\beta Q(\boldsymbol\beta,\boldsymbol\beta_1^*)
+(1-t^*)\partial_\beta Q(\boldsymbol\beta,\boldsymbol\beta_2^*)\\
&=\frac{2}{\alpha_0^2}
\left(
\beta\left(
t^*\left(\frac{\alpha_0^2}{2\beta_1^*}+(\gamma-\gamma_1^*)^2\right)
+(1-t^*)\left(\frac{\alpha_0^2}{2\beta_2^*}+(\gamma-\gamma_2^*)^2\right)
\right)
-\frac{\alpha_0^2}{2}
\right),
\end{align*}
%%%%%%%%%%%%%%%%%%%%%%%%%%%%%%%
we have
\begin{align*}
\bar\gamma_*&=t^*\gamma_1^*+(1-t^*)\gamma_2^*,\\
\bar\beta_*
&=\frac{\frac{\alpha_0^2}{2}}
{t^*\left(\frac{\alpha_0^2}{2\beta_1^*}+(\gamma-\gamma_1^*)^2\right)
+(1-t^*)\left(\frac{\alpha_0^2}{2\beta_2^*}+(\gamma-\gamma_2^*)^2\right)
}\\
&=\frac{1}
{t^*\frac{1}{\beta_1^*}+(1-t^*)\frac{1}{\beta_2^*}
+\frac{2}{\alpha_0^2}t^*(1-t^*)(\gamma_1^*-\gamma_2^*)^2}>0.
\end{align*}
%%%%%%%%%%%%%%%%%%%%%%%%%%%%%%%%%%%%%%%%%%%
If $\beta_1^*\neq\beta_2^*$ and $\gamma_1^*=\gamma_2^*$, then
\begin{align*}
G(\boldsymbol\beta_1^*)-G(\boldsymbol\beta_2^*)
=\frac{\bar\beta_*}{\alpha_0}((\gamma_1^*-\bar\gamma_*)-(\gamma_2^*-\bar\gamma_*))
=\frac{\bar\beta_*}{\alpha_0}(\gamma_1^*-\gamma_2^*)=0,
\end{align*}
and {\textbf{[B4]}} does not hold.
%%%%%%
If $\gamma_1^*\neq\gamma_2^*$, then
\begin{align*}
G(\boldsymbol\beta_1^*)-G(\boldsymbol\beta_2^*)
=\frac{\bar\beta_*}{\alpha_0}(\gamma_1^*-\gamma_2^*)\neq0,
\end{align*}
and the test $T_{1,n}^\beta$ has consistency according to Theorem \ref{th5}.

%%%%%%%%%%%%%%%%%
Furthermore, we study  the consistency of the test $T_{2, n}^\beta$. 
One has
\begin{align*}
H(\boldsymbol\beta_1^*)
&=\frac{1}{\alpha_0^2}\int_{\mathbb R}
\left(\begin{array}{@{\,}c@{\,}}
	-(x-\bar\gamma_*) \\
	\bar\beta_*
\end{array}\right)
((-\beta_1^*x+\beta_1^*\gamma_1^*)-(-\bar\beta_*x+\bar\beta_*\bar\gamma_*))
\dd\mu_{\theta_1^*}(x)\\
&=\frac{1}{\alpha_0^2}\int_{\mathbb R}
\left(\begin{array}{@{\,}c@{\,}}
	-(x-\bar\gamma_*) \\
	\bar\beta_*
\end{array}\right)
(-(\beta_1^*-\bar\beta_*)x+(\beta_1^*\gamma_1^*-\bar\beta_*\bar\gamma_*))
\dd\mu_{\theta_1^*}(x)\\
&=\frac{1}{\alpha_0^2}\int_{\mathbb R}
\left(\begin{array}{@{\,}c@{\,}}
	(\beta_1^*-\bar\beta_*)x^2
	-(\beta_1^*\gamma_1^*-\bar\beta_*\bar\gamma_*+\bar\gamma_*(\beta_1^*-\bar\beta_*))x 
	+\bar\gamma_*(\beta_1^*\gamma_1^*-\bar\beta_*\bar\gamma_*)\\
	\bar\beta_*(-(\beta_1^*-\bar\beta_*)x+(\beta_1^*\gamma_1^*-\bar\beta_*\bar\gamma_*))
\end{array}\right)
\dd\mu_{\theta_1^*}(x)\\
&=\frac{1}{\alpha_0^2}
\left(\begin{array}{@{\,}c@{\,}}
	(\beta_1^*-\bar\beta_*)\left((\gamma_1^*)^2+\frac{\alpha_0^2}{2\beta_1^*}\right)
	-(\beta_1^*\gamma_1^*-\bar\beta_*\bar\gamma_*+\bar\gamma_*(\beta_1^*-\bar\beta_*))
	\gamma_1^* 
	+\bar\gamma_*(\beta_1^*\gamma_1^*-\bar\beta_*\bar\gamma_*)\\
	\bar\beta_*(-(\beta_1^*-\bar\beta_*)\gamma_1^*
	+(\beta_1^*\gamma_1^*-\bar\beta_*\bar\gamma_*))
\end{array}\right)
\dd\mu_{\theta_1^*}(x)\\
&=\frac{1}{\alpha_0^2}
\left(\begin{array}{@{\,}c@{\,}}
	\frac{\alpha_0^2}{2}\left(1-\frac{\bar\beta_*}{\beta_1^*}\right)
-\bar\beta_*(\gamma_1^*-\bar\gamma_*)^2\\
	(\bar\beta_*)^2(\gamma_1^*-\bar\gamma_*)
\end{array}\right)
\end{align*}
%%%%%%%%%%%%%
and
\begin{align*}
H(\boldsymbol\beta_1^*)-H(\boldsymbol\beta_2^*)
&=\frac{\bar\beta_*}{\alpha_0^2}
\left(\begin{array}{@{\,}c@{\,}}
	-\frac{\alpha_0^2}{2}\left(\frac{1}{\beta_1^*}-\frac{1}{\beta_2^*}\right)
	-(\gamma_1^*-\gamma_2^*)(\gamma_1^*+\gamma_2^*-2\bar\gamma_*)\\
	\bar\beta_*(\gamma_1^*-\gamma_2^*)
\end{array}\right).
\end{align*}
%%%%%%%%%%%%%
If $\gamma_1^*\neq\gamma_2^*$, noting that $\bar\beta_*>0$, then
\begin{align*}
H(\beta_1^*,\beta_1^*\gamma_1^*)-H(\beta_2^*,\beta_2^*\gamma_2^*)\neq0.
\end{align*}
%%%%
If $\beta_1^*\neq\beta_2^*$ and $\gamma_1^*=\gamma_2^*$, 
%noting that $\bar\gamma_*=\gamma_1^*$, 
then
\begin{align*}
H(\beta_1^*,\beta_1^*\gamma_1^*)-H(\beta_2^*,\beta_2^*\gamma_2^*)
&=-\frac{\bar\beta_*}{2}\left(\frac{1}{\beta_1^*}-\frac{1}{\beta_2^*}\right)
\left(\begin{array}{@{\,}c@{\,}}
	1\\
	0
\end{array}\right)\neq0.
\end{align*}
Therefore, the test $T_{2,n}^\beta$ is consistent by Corollary \ref{th7}.

%%%%%%%%%%%%%%%%%%%%%%%%%%%%%%%%%%%%%%%%%%%
Figures $\ref{fig3}$-$\ref{fig8}$ show the sample paths
in Case 1, (i) of Case 2 and (i)-(iii) of Case 3, and 
Figures $\ref{fig9}$-$\ref{fig22}$ show the histograms 
of $T_n^\alpha, T_{1,n}^\beta$ and $T_{2,n}^\beta$ and 
the probability density function 
of $\sup_{0\le s\le 1}|\boldsymbol B_1^0(s)|$ and  
$\sup_{0\le s\le1}\|\boldsymbol B(s)\|$ obtained by the simulation in the above cases.
Figures $\ref{fig23}$-$\ref{fig35}$ show the empirical distribution function of 
$T_n^\alpha, T_{1,n}^\beta$ and $T_{2,n}^\beta$ and
the cumulative distribution function of 
$\sup_{0\le s\le 1}|\boldsymbol B_1^0(s)|$ and $\sup_{0\le s\le1}\|\boldsymbol B(s)\|$.
%%%

Table \ref{tab1} shows the empirical sizes of $T_n^\alpha, T_{1,n}^\beta$ and $T_{2,n}^\beta$ 
in Case 1 when $n=8\times 10^3, 1.25\times 10^5$ and $10^6$. 
According to Table \ref{tab1}, the empirical sizes of 
$T_n^\alpha, T_{1,n}^\beta$ and $T_{2,n}^\beta$ are around 0.1 
for each $n$ and 
$(\alpha,\beta,\gamma)$. 
Furthermore, from 
Figures $\ref{fig9}$-$\ref{fig11}$ and Figures $\ref{fig23}$-$\ref{fig25}$, 
we can see \eqref{simBB}. 
%%%

Table $\ref{tab2}$ shows the empirical powers of $T_n^\alpha$
in Case $2$. 
According to Table \ref{tab2}, 
if $n=10^6$, the changes as shown in Figure \ref{fig4} is detected.
We can confirm this from Figures $\ref{fig12}$ and $\ref{fig26}$.
%%%

Tables $\ref{tab4}$-$\ref{tab6}$ show the empirical sizes of $T_n^\alpha$ 
and empirical powers of $T_{1,n}^\beta$ and $T_{2,n}^\beta$ in Case 3.
First, in (i)-(iii) of Case 3, $T_n^\alpha$ are about 0.1, and we can confirm 
$T_n^\alpha\dto\sup_{0\le s\le 1}|\boldsymbol B_1^0(s)|$ from 
Figures $\ref{fig14}$, $\ref{fig17}$, $\ref{fig20}$, $\ref{fig27}$, $\ref{fig30}$ and $\ref{fig33}$.
%Looking at Table \ref{tab4}, 
It can be seen from Table \ref{tab4} that $T_{1,n}^\beta$ does not give good results
in (i) of Case 3.
%%%%
%From this and 
As seen above, 
since $T_{1,n}^\beta$ does not satisfy the sufficient condition \textbf{[B4]} for having consistency
in (i) of Case 3, 
it can be guessed that $T_{1,n}^\beta$ may not have consistency in (i) of Case 3. 
In contrast to $T_{1,n}^\beta$, $T_{2,n}^\beta$ gives good results in (i) of Case 3. 
We can confirm { them} from Figures $\ref{fig15}$, $\ref{fig16}$, $\ref{fig28}$ and $\ref{fig29}$. 
%%%
Similarly, when $\gamma$ changes in (ii) of Case 3,
we can see $T_n^\alpha\dto\sup_{0\le s\le 1}|\boldsymbol B_1^0(s)|$
from Table \ref{tab5}, Figures $\ref{fig17}$ and $\ref{fig30}$, 
and the consistency of $T_{1,n}^\beta$ and $T_{2,n}^\beta$  
from Figures $\ref{fig18}$, $\ref{fig19}$, $\ref{fig31}$ and $\ref{fig32}$.
%%%
Moreover, in (iii) of Case 3,
we can confirm 
$T_n^\alpha\dto\sup_{0\le s\le 1}|\boldsymbol B_1^0(s)|$ from Table \ref{tab6},  Figures $\ref{fig20}$ and $\ref{fig33}$
and the consistency of $T_{1,n}^\beta$ and $T_{2,n}^\beta$ 
from Figures $\ref{fig21}$, $\ref{fig22}$, $\ref{fig34}$ and $\ref{fig35}$.

%%%%%%
\begin{table}[h]
\begin{center}
\caption{Empirical sizes of { $T_n^\alpha$, $T_{1,n}^\beta$ and $T_{2,n}^\beta$} in Case 1}
\begin{tabular*}{.8\textwidth}{@{\extracolsep{\fill}}cccc|cccc}
	& & & & \multicolumn{4}{c}{$(\alpha,\beta,\gamma)$}   \rule[0mm]{0cm}{4mm}\\\hline
	$n$ & $nh_n$ & $h_n$ & & $(1,1,1)$ & $(0.5,1,0)$ & $(1.5,1.5,-1)$ & $(2,3,0.5)$  \rule[0mm]{0cm}{4mm}\\\hline\hline
%	$10^3$ & 10 & $10^{-2}$ & $T_n^\alpha$ & 0.089 & 0.093 & 0.095 & 0.094\, \rule[0mm]{0cm}{4mm}\\%\hline
% 	&  &  & $T_{1,n}^\beta$ & 0.061 & 0.061 & 0.066 & 0.065	\rule[0mm]{0cm}{4mm}\\%\hline
%	&  &  & $T_{2,n}^\beta$ & 0.101 & 0.084 & 0.091 & 0.088	\rule[0mm]{0cm}{4mm}\\\hline
	$8.0\times10^3$ & 20 & $2.5\times10^{-3}$ &$T_n^\alpha$ & 0.106 & 0.107 & 0.107 & 0.107	\rule[0mm]{0cm}{4mm}\\%\hline
 	&  &  &$T_{1,n}^\beta$& 0.095 & 0.078 & 0.091 & 0.097	\rule[0mm]{0cm}{4mm}\\%\hline
 	&  &  &$T_{2,n}^\beta$ &0.091 & 0.083 & 0.087 & 0.090	\rule[0mm]{0cm}{4mm}\\\hline
	$1.25\times10^5$ & 50 & $4.0\times10^{-4}$&$T_n^\alpha$ & 0.100 & 0.099 & 0.099 & 0.100	\rule[0mm]{0cm}{4mm}\\%\hline
 	&  &  &$T_{1,n}^\beta$& 0.093 & 0.102 & 0.109 & 0.115	\rule[0mm]{0cm}{4mm}\\%\hline
 	&  &  &$T_{2,n}^\beta$ &0.098 & 0.095 & 0.091 & 0.106	\rule[0mm]{0cm}{4mm}\\\hline
	$10^6$ & 100 & $10^{-4}$ &$T_n^\alpha$& 0.119 & 0.118 & 0.119 & 0.118 	\rule[0mm]{0cm}{4mm}\\%\hline
 	&  &  &$T_{1,n}^\beta$& 0.106 & 0.104 & 0.107 & 0.102	\rule[0mm]{0cm}{4mm}\\%\hline
 	&  &  &$T_{2,n}^\beta$& 0.097 & 0.096 & 0.095 & 0.095	\rule[0mm]{0cm}{4mm}
\end{tabular*}
\label{tab1}
\end{center}
\end{table}
%%%
%%%%%%%%%%%%%%%%%%%%%%%%%%%%
\begin{table}[h]
\begin{center}
\caption{Empirical powers of $T_n^\alpha$ in (i) of Case 2}
\begin{tabular*}{.8\textwidth}{@{\extracolsep{\fill}}ccc|cccc}
	& & &  \multicolumn{4}{c}{$\alpha_1^*=1\lto\alpha_2^*$}   \rule[0mm]{0cm}{4mm}\\\hline
	$n$ & $nh_n$ & $h_n$ & 1.01 & 1.05 & 1.1 & 1.5 	\rule[0mm]{0cm}{4mm}\\\hline\hline
%	$10^3$ & 10 & $10^{-2}$ & 0.098 & 0.220 & 0.531 & 1.000\,  	\rule[0mm]{0cm}{4mm}\\\hline
	$8.0\times10^3$ & 20 & $2.5\times10^{-3}$ & 0.144 & 0.864 & 1.000 & 1.000 	\rule[0mm]{0cm}{4mm}\\\hline
	$1.25\times10^5$ & 50 & $4.0\times10^{-4}$ & 0.741 & 1.000 & 1.000 & 1.000 	\rule[0mm]{0cm}{4mm}\\\hline
	$10^6$ & 100 & $10^{-4}$ & 1.000 & 1.000 & 1.000 &  1.000	\rule[0mm]{0cm}{4mm}
\end{tabular*}
\label{tab2}
\end{center}
\end{table}
%%%%%%%%%%%%%%%%%%%%%%%%%%%%%%%%%%%%%%%%%%%%%%%%%%%%%%%%%%%%%%%%%%
\begin{table}[h]
\begin{center}
\caption{Empirical sizes of $T_n^\alpha$ and 
empirical powers of $T_{1,n}^\beta$ and $T_{2,n}^\beta$ in (i) of Case 3}
\begin{tabular*}{.8\textwidth}{@{\extracolsep{\fill}}cccc|cccc}
	& & &&  \multicolumn{4}{c}{$\beta_1^*=1\lto\beta_2^*$}	   \rule[0mm]{0cm}{4mm}\\\hline
	$n$ & $nh_n$ & $h_n$ && 1.1 & 1.5 & 3 & 5  	\rule[0mm]{0cm}{4mm}\\\hline\hline
%	$10^3$ & 10 & $10^{-2}$ &$T_n^\alpha$& 0.089 & 0.095 & 0.100 & 0.097\, \rule[0mm]{0cm}{4mm}\\%\hline
% 	&  &  &$T_{1,n}^\beta$& 0.070 & 0.082 & 0.140 & 0.217		\rule[0mm]{0cm}{4mm}\\%\hline
%	&  &  &$T_{2,n}^\beta$& 0.107 & 0.145 & 0.305 & 0.574 \rule[0mm]{0cm}{4mm}\\\hline
	$8.0\times10^3$ & 20 & $2.5\times10^{-3}$ &$T_n^\alpha$& 0.106 & 0.107 & 0.102 & 0.102	\rule[0mm]{0cm}{4mm}\\%\hline
 	&  &  & $T_{1,n}^\beta$& 0.095 & 0.112 & 0.208 & 0.296	\rule[0mm]{0cm}{4mm}\\%\hline
	&  & & $T_{2,n}^\beta$& 0.100 & 0.152 & 0.532 & 0.918\rule[0mm]{0cm}{4mm}\\\hline
	$1.25\times10^5$ & 50 & $4.0\times10^{-4}$ &$T_n^\alpha$& 0.099 & 0.101 & 0.101 & 0.097	\rule[0mm]{0cm}{4mm}\\%\hline
 	&  &  &	$T_{1,n}^\beta$& 0.101 & 0.119 & 0.204 & 0.296\rule[0mm]{0cm}{4mm}\\%\hline
	&  & &$T_{2,n}^\beta$ & 0.119 & 0.247 & 0.906 & 0.999\rule[0mm]{0cm}{4mm}\\\hline
	$10^6$ & 100 & $10^{-4}$ &$T_n^\alpha$ & 0.119 & 0.118 & 0.118 & 0.116	\rule[0mm]{0cm}{4mm}\\%\hline
 	&  &  &	$T_{1,n}^\beta$& 0.108 & 0.129 & 0.220 & 0.302	\rule[0mm]{0cm}{4mm}\\%\hline
	& & & $T_{2,n}^\beta$&0.126 & 0.351 & 0.998 & 1.000 \rule[0mm]{0cm}{4mm}
\end{tabular*}
\label{tab4}
\end{center}
\end{table}
%%%%%%%%%%%%%%%%%%%%%%%%%%%%%%%%%%%%%%%%%%%%%%%%%%%%%%%%%%%%%%%%%%%%%%%%%%%%%
%%%%%%%%%%%%%%%%%%%%%%%%%%%%%%%%%%%%%%%%%%%%%%%%%%%%%%%%%%%%%%%%%%%%%%%%%%%%
\begin{table}[h]
\begin{center}
\caption{Empirical sizes of $T_n^\alpha$ and 
empirical powers of $T_{1,n}^\beta$ and $T_{2,n}^\beta$ in (ii) of Case 3}
\begin{tabular*}{.8\textwidth}{@{\extracolsep{\fill}}cccc|cccc}
	& & &&  \multicolumn{4}{c}{$\gamma_1^*=1\lto\gamma_2^*$}   \rule[0mm]{0cm}{4mm}\\\hline
	$n$ & $nh_n$ & $h_n$ && 0.9 & 0.5 & 0 & $-1$  \rule[0mm]{0cm}{4mm}\\\hline\hline
%	$10^3$ & 10 & $10^{-2}$ &$T_n^\alpha$& 0.089 & 0.089 & 0.089 & 0.092\, \rule[0mm]{0cm}{4mm}\\%\hline
% 	&  &  &$T_{1,n}^\beta$& 0.061 & 0.075 & 0.094 & 0.066	\rule[0mm]{0cm}{4mm}\\%\hline
%	&&&$T_{2,n}^\beta$& 0.105 & 0.099 & 0.106 & 0.078 \rule[0mm]{0cm}{4mm}\\\hline
	$8.0\times10^3$ & 20 & $2.5\times10^{-3}$ &$T_n^\alpha$& 0.106 & 0.106 & 0.104 & 0.102	\rule[0mm]{0cm}{4mm}\\%\hline
 	&  &  &	$T_{1,n}^\beta$& 0.103 & 0.170 & 0.287 & 0.210\rule[0mm]{0cm}{4mm}\\%\hline
	&&&$T_{2,n}^\beta$& 0.093 & 0.118 & 0.158 & 0.075\rule[0mm]{0cm}{4mm}\\\hline
	$1.25\times10^5$ & 50 & $4.0\times10^{-4}$ &$T_n^\alpha$&0.100 & 0.100 & 0.100 & 0.100	\rule[0mm]{0cm}{4mm}\\%\hline
 	&  &  & $T_{1,n}^\beta$& 0.108 & 0.378 & 0.852 & 0.990\rule[0mm]{0cm}{4mm}\\%\hline
	&&&$T_{2,n}^\beta$& 0.102 & 0.277 & 0.621 & 0.641\rule[0mm]{0cm}{4mm}\\\hline
	$10^6$ & 100 & $10^{-4}$ &$T_n^\alpha$ &0.119 & 0.119 & 0.119 & 0.119	\rule[0mm]{0cm}{4mm}\\%\hline
 	&  &  & $T_{1,n}^\beta$& 0.121 & 0.680 & 0.996 & 1.000\rule[0mm]{0cm}{4mm}\\%\hline
	&&&$T_{2,n}^\beta$& 0.115 & 0.574 & 0.980 & 1.000 \rule[0mm]{0cm}{4mm}
\end{tabular*}
\label{tab5}
\end{center}
\end{table}
%%%%%%%%%%%%%%%%%%%%%%%%%%%%%%%%%%%%%%%%%%%%%%%%%%%%%%%%%%%%%%%%
%%%%%%%%%%%%%%%%%%%%%%%%%%%%%%%%%%%%%%%%%%%%%%%%%%%%%%%%%%%%%%%
\begin{table}[h]
\begin{center}
\caption{Empirical sizes of $T_n^\alpha$ and 
empirical powers of $T_{1,n}^\beta$ and $T_{2,n}^\beta$ in (iii) of Case 3}
\begin{tabular*}{.8\textwidth}{@{\extracolsep{\fill}}cccc|ccccc}
	& & &&  \multicolumn{4}{c}{$(\beta_1^*,\gamma_1^*)=(1,1)\lto(\beta_2^*,\gamma_2^*)$}   \rule[0mm]{0cm}{4mm}\\\hline
	$n$ & $nh_n$ & $h_n$ && $(3,0.5)$ & $(3,0)$ & $(5,0.5)$ & $(5,0)$  \rule[0mm]{0cm}{4mm}\\\hline\hline
%	$10^3$ & 10 & $10^{-2}$ &$T_n^\alpha$& 0.095 & 0.091 & 0.098 & 0.101\, \rule[0mm]{0cm}{4mm}\\%\hline
%	&  &  &$T_{1,n}^\beta$& 0.231 & 0.228 & 0.367 & 0.411 \rule[0mm]{0cm}{4mm}\\%\hline
%	&&&$T_{2,n}^\beta$& 0.331 & 0.302 & 0.617 & 0.497 \rule[0mm]{0cm}{4mm}\\\hline
	$8.0\times10^3$ & 20 & $2.5\times10^{-3}$ &$T_n^\alpha$& 0.102 & 0.102 & 0.100 & 0.101 \rule[0mm]{0cm}{4mm}\\%\hline
 	&  &  &	$T_{1,n}^\beta$& 0.458 & 0.758 & 0.585 & 0.860 \rule[0mm]{0cm}{4mm}\\%\hline
	&&&$T_{2,n}^\beta$& 0.676 & 0.735 & 0.961 & 0.962 \rule[0mm]{0cm}{4mm}\\\hline
	$1.25\times10^5$ & 50 & $4.0\times10^{-4}$ &$T_n^\alpha$&0.100 & 0.100 & 0.096 & 0.096 \rule[0mm]{0cm}{4mm}\\%\hline
 	&  &  & $T_{1,n}^\beta$& 0.720 & 0.992 & 0.808 & 0.996 \rule[0mm]{0cm}{4mm}\\%\hline
	&&&$T_{2,n}^\beta$& 0.989 & 1.000 & 0.999 & 1.000 \rule[0mm]{0cm}{4mm}\\\hline
	$10^6$ & 100 & $10^{-4}$ &$T_n^\alpha$ & 0.118 & 0.119 & 0.116 & 0.116 \rule[0mm]{0cm}{4mm}\\
 	&  &  &	$T_{1,n}^\beta$& 0.931 & 1.000 & 0.956 & 1.000 \rule[0mm]{0cm}{4mm}\\
	&&&$T_{2,n}^\beta$& 1.000 & 1.000 & 1.000 & 1.000 \rule[0mm]{0cm}{4mm}
\end{tabular*}
\label{tab6}
\end{center}
\end{table}
%%%%%%%%%%%%%%%%%%%%%%%%%%%%%%%%%%%%%%%%%%%%%%%%%%%%%%%%%%%%%%%%%%%%%%%%%%%%%%%%%%%%%%%%
%%%%%%
\begin{figure}[h]
\begin{center}
\caption{Sample paths of the 1-dimensional OU process \eqref{OU} with 
$(\alpha,\beta,\gamma)=(1,1,1), (0.5,1,0), (1.5,1.5,-1)$, and $(2,3,0.5)$,
{ in the order of upper left, upper right, lower left and lower right}  (Case 1).}
\includegraphics[height=10cm, width=15cm]{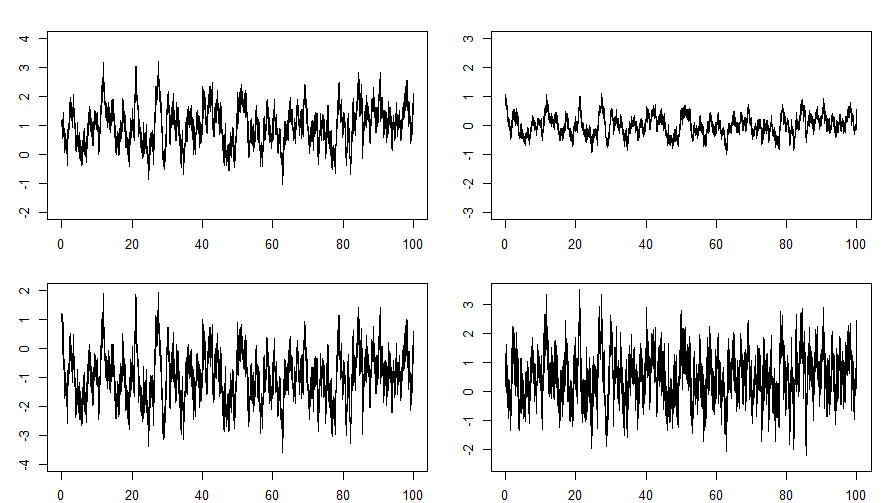}
\label{fig3}
%\end{center}
%\end{figure}
%%%%%%
%\begin{figure}[h]
%\begin{center}
\caption{
Sample paths of the 1-dimensional OU process \eqref{OU}
whose parameter changes
%where the parameter 
from  $\alpha=1$ to 
$1.01, 1.05, 1.1$ and $1.5$,
in the order of upper left, upper right, lower left and lower right} ((i) of Case 2).
\includegraphics[height=10cm, width=15cm]{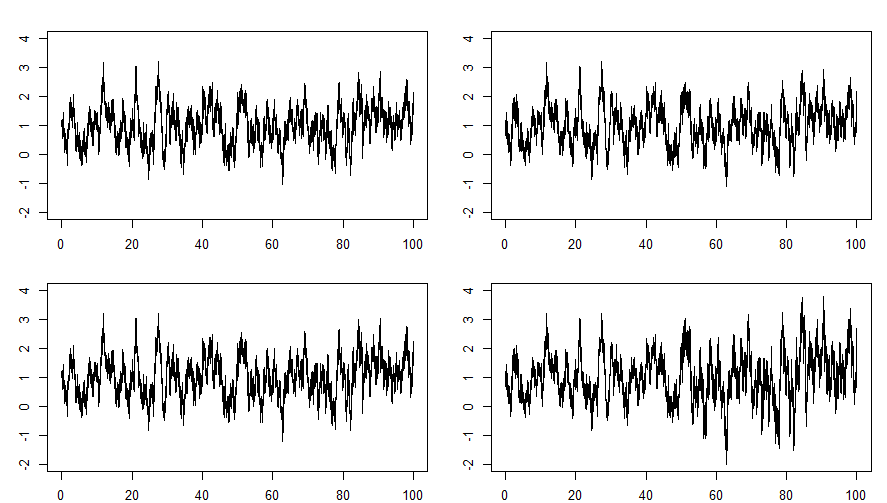}
\label{fig4}
\end{center}
\end{figure}
%%%%%%%%%%%%%%%%%%%%%%%%%%
%%%%%%
\begin{figure}[h]
\begin{center}
\caption{
Sample paths of the 1-dimensional OU process \eqref{OU}
where the parameter changes from  $\beta=1$ to 
$1.1, 1.5, 3$ and $5$,
whose parameter in the order of upper left, upper right, lower left and lower right
 ((i) of Case 3).}
\includegraphics[height=10cm, width=15cm]{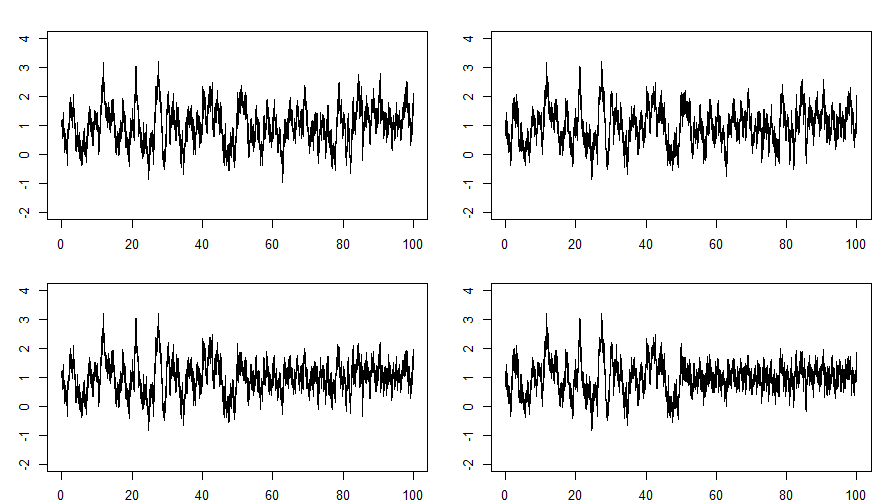}
\label{fig6}
%\end{center}
%\end{figure}
%%%%%%%%%%%%%%%%%%%%%%%%%%%%
%\begin{figure}[h]
%\begin{center}
\caption{
Sample paths of the 1-dimensional OU process \eqref{OU}
whose parameter changes from  $\gamma=1$ to 
$0.9, 0.5, 0$ and $-1$,
whose parameter in the order of upper left, upper right, lower left and lower right ((ii) of Case 3).}
\includegraphics[height=10cm, width=15cm]{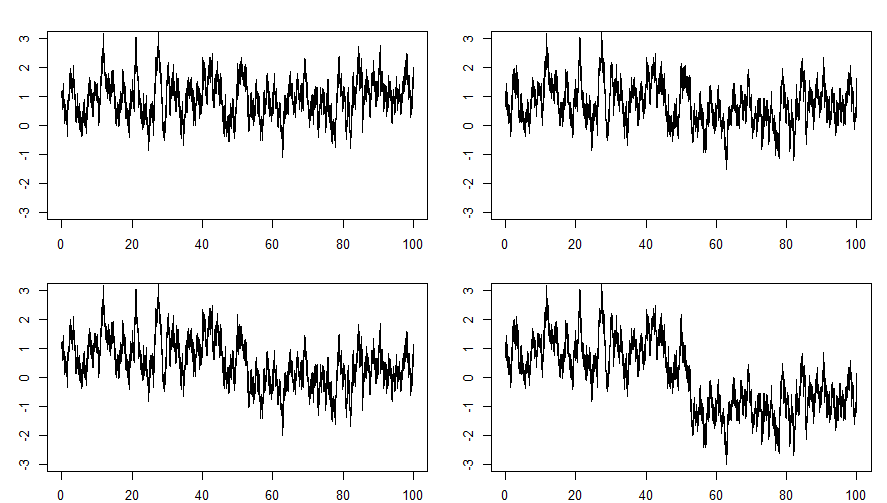}
\label{fig7}
\end{center}
\end{figure}
%%%%%%
\begin{figure}[h]
\begin{center}
\caption{
Sample paths of the 1-dimensional OU process \eqref{OU}
whose parameter changes from  $(\beta,\gamma)=(1,1)$ to 
$(3,0.5), (3,0), (5,0.5)$ and $(5,0)$,
whose parameter in the order of upper left, upper right, lower left and lower right ((iii) of Case 3).}
\includegraphics[height=10cm, width=15cm]{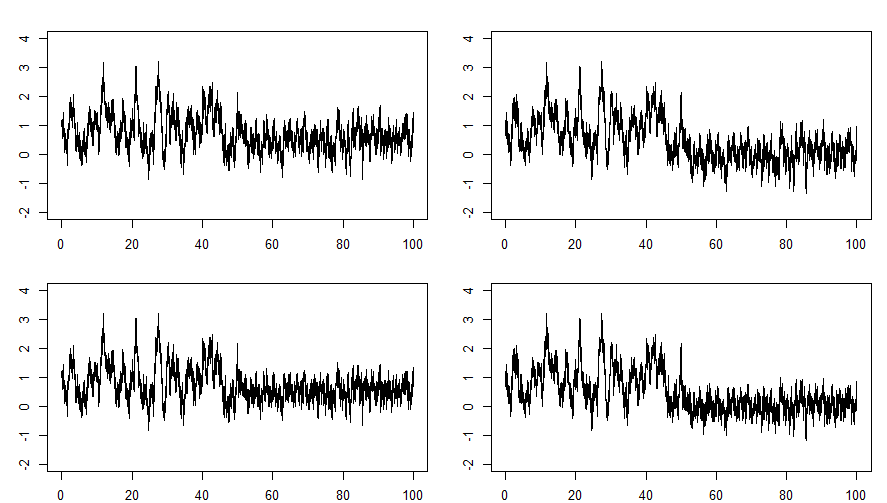}
\label{fig8}
%\end{center}
%\end{figure}
%%%%%%
%%%%%%%%%%%%%%%%%%%%%%%%%%%%%%%%%%%%%%%%%%%%%%%%%%%%%%%%%%%%%%%%%%%%%%%%%%%%%%%%%%%%%%%%
%\begin{figure}[h]
%\begin{center}
\caption{
Histograms of $T_n^\alpha$ when $n=8\times10^3, 1.25\times10^5$ 
and $10^6$ in order from the first row 
and $(\alpha,\beta,\gamma)=(1,1,1), (0.5,1,0), (1.5,1.5,-1)$ and $(2,3,0.5)$ in oder from the first column (Case 1). 
The red line is the probability density function of $\sup_{0\le s\le 1}|\boldsymbol B_1^0(s)|$.
}
\includegraphics[height=10cm, width=15cm]{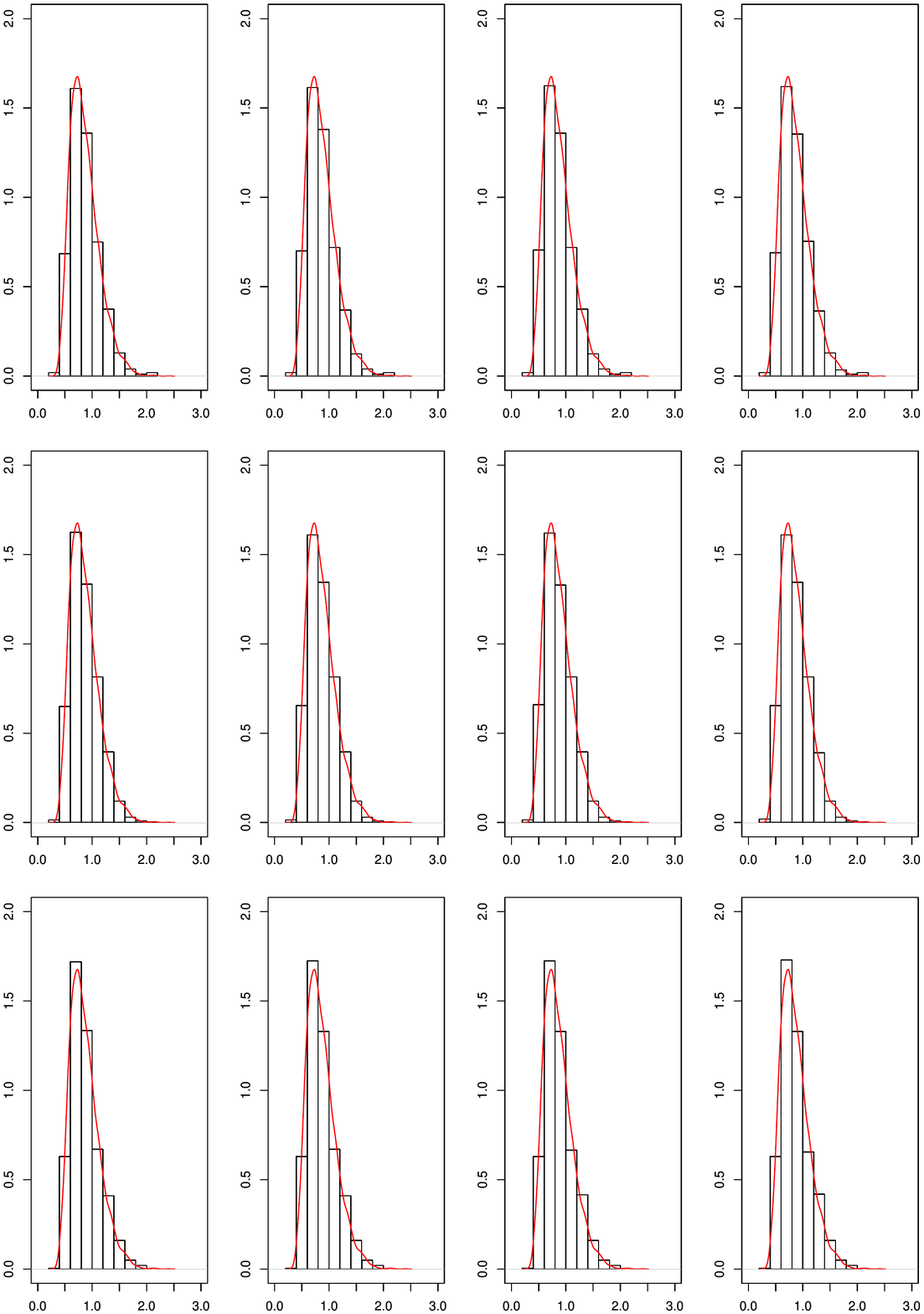}
\label{fig9}
\end{center}
\end{figure}
%%%
\begin{figure}[h]
\begin{center}
\caption{
Histograms of $T_{1,n}^\beta$ when { $n=8\times10^3, 1.25\times10^5$} and $10^6$ in order from the first row 
and $(\alpha,\beta,\gamma)=(1,1,1), (0.5,1,0), (1.5,1.5,-1)$ and $(2,3,0.5)$ in oder from the first column (Case 1). 
The red line is the probability density function of $\sup_{0\le s\le 1}|\boldsymbol B_1^0(s)|$.
}
\includegraphics[height=10cm, width=15cm]{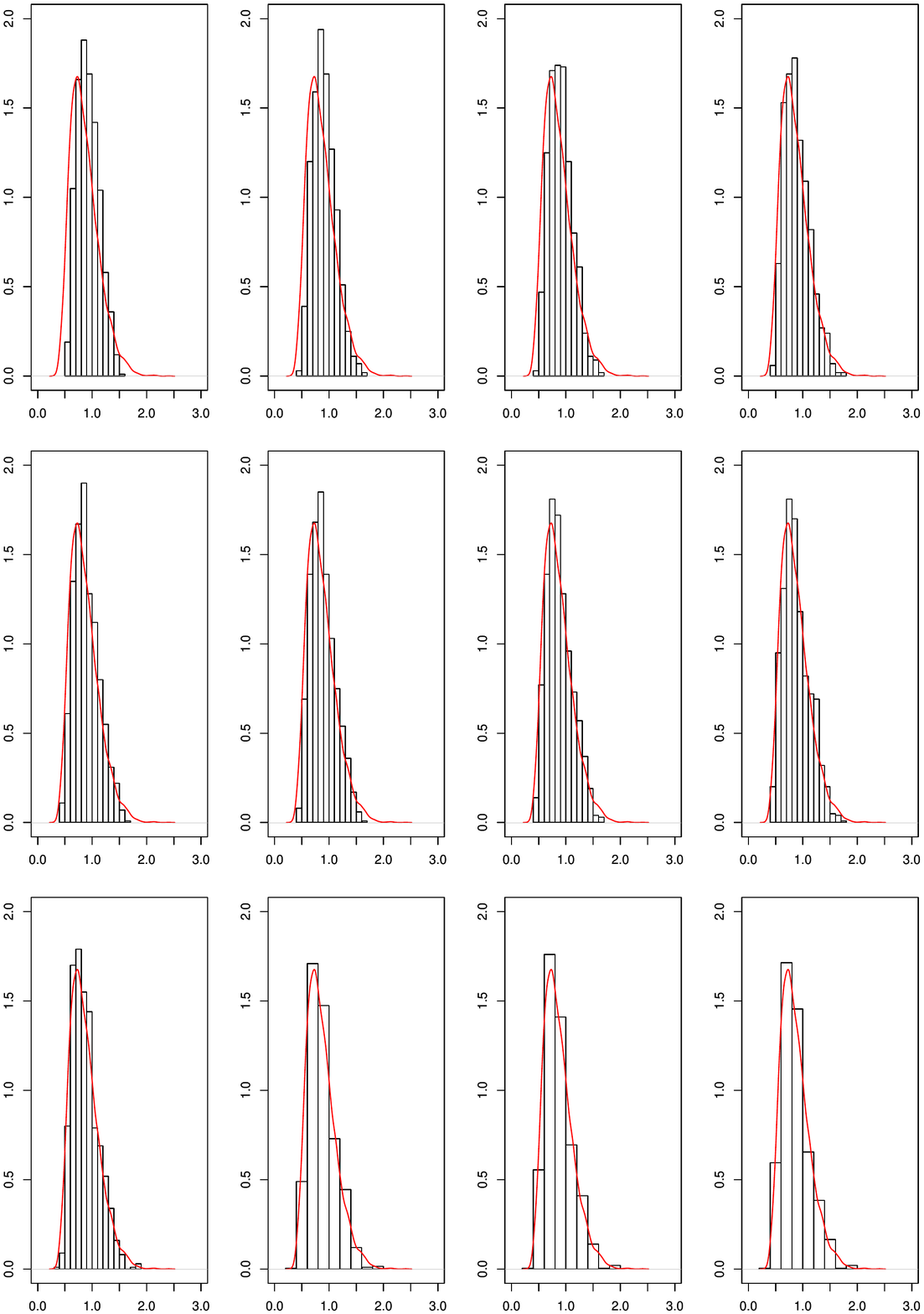}
\label{fig10}
%\end{center}
%\end{figure}
%%%%%%%%%%%%%%%%%%%%%%%
%\begin{figure}[h]
%\begin{center}
\caption{
Histograms of $T_{2,n}^\beta$ when { $n=8\times10^3, 1.25\times10^5$} and $10^6$ in order from the first row 
and $(\alpha,\beta,\gamma)=(1,1,1), (0.5,1,0), (1.5,1.5,-1)$ and $(2,3,0.5)$ in oder from the first column (Case 1). 
The red line is the probability density function of $\sup_{0\le s\le 1}\|\boldsymbol B_2^0(s)\|$.
}
\includegraphics[height=10cm, width=15cm]{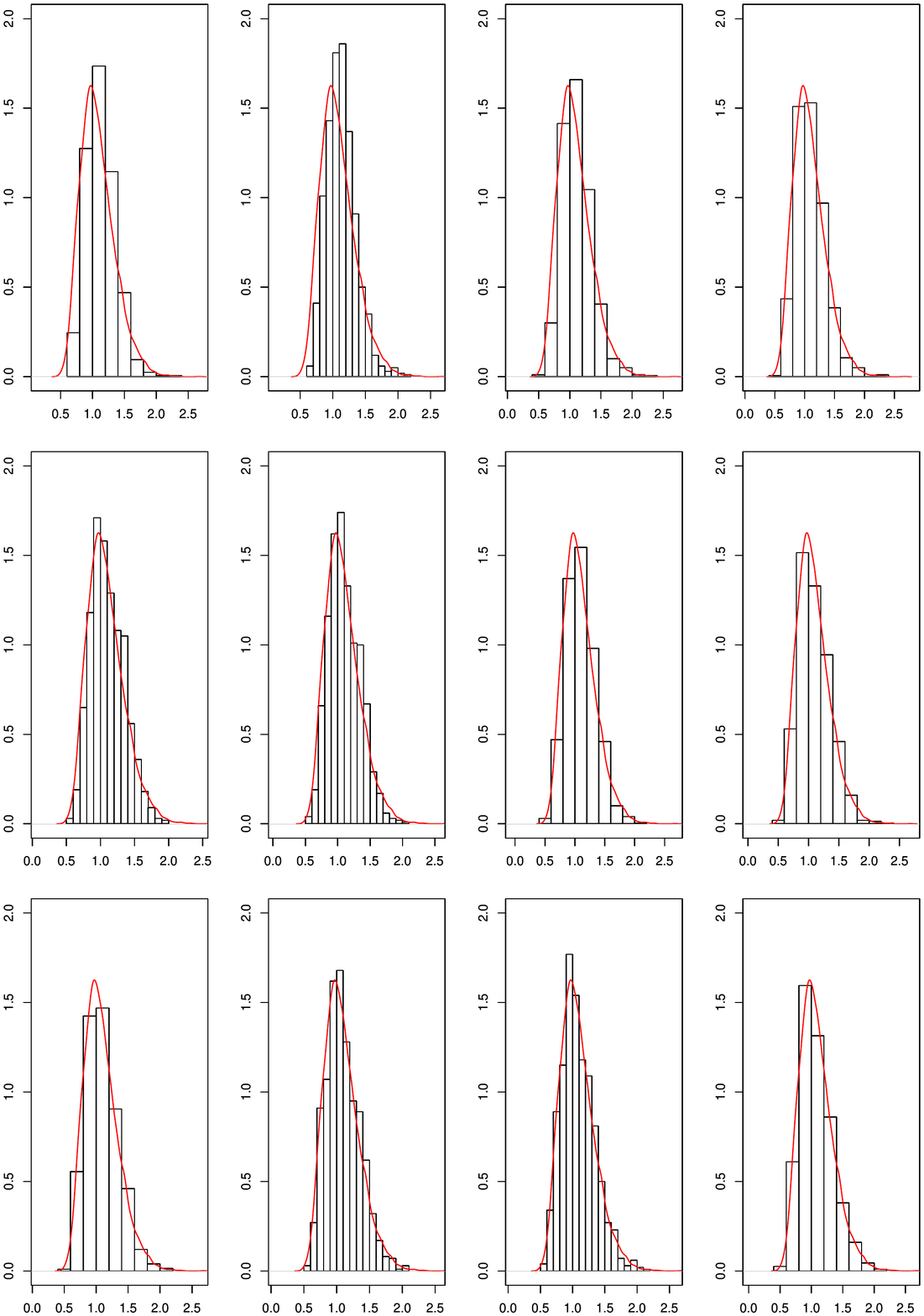}
\label{fig11}
\end{center}
\end{figure}
%%%
\begin{figure}[h]
\begin{center}
\caption{
Histograms of $T_n^\alpha$ when { $n=8\times10^3, 1.25\times10^5$} and $10^6$ in order from the first row 
and $\alpha$ changes from $1$ to $1.01, 1.05, 1.1$ and $1.5$ in oder from the first column 
((i) of Case 2). 
The red line is the probability density function of $\sup_{0\le s\le 1}|\boldsymbol B_1^0(s)|$.
}
\includegraphics[height=10cm, width=15cm]{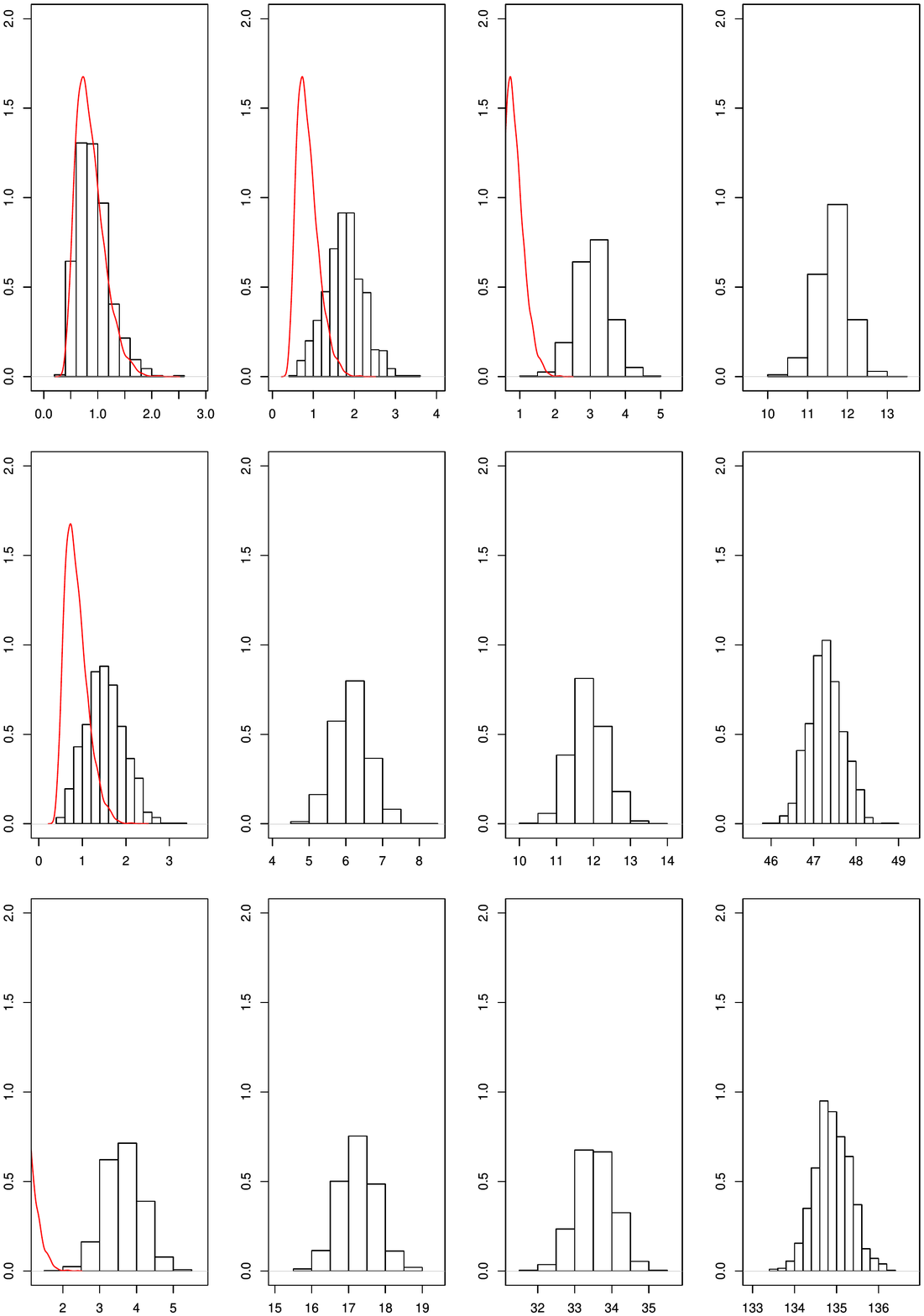}
\label{fig12}
%\end{center}
%\end{figure}
%%%%%%%%%%%%%%%%%%%%%%%
%\begin{figure}[h]
%\begin{center}
\caption{
Histograms of $T_n^\alpha$ when { $n=8\times10^3, 1.25\times10^5$} and $10^6$ in order from the first row 
and $\beta$ changes from $1$ to $1.1, 1.5, 3$ and $5$ in oder from the first column ((i) of Case 3).
The red line is the probability density function of $\sup_{0\le s\le 1}|\boldsymbol B_1^0(s)|$.
}
\includegraphics[height=10cm, width=15cm]{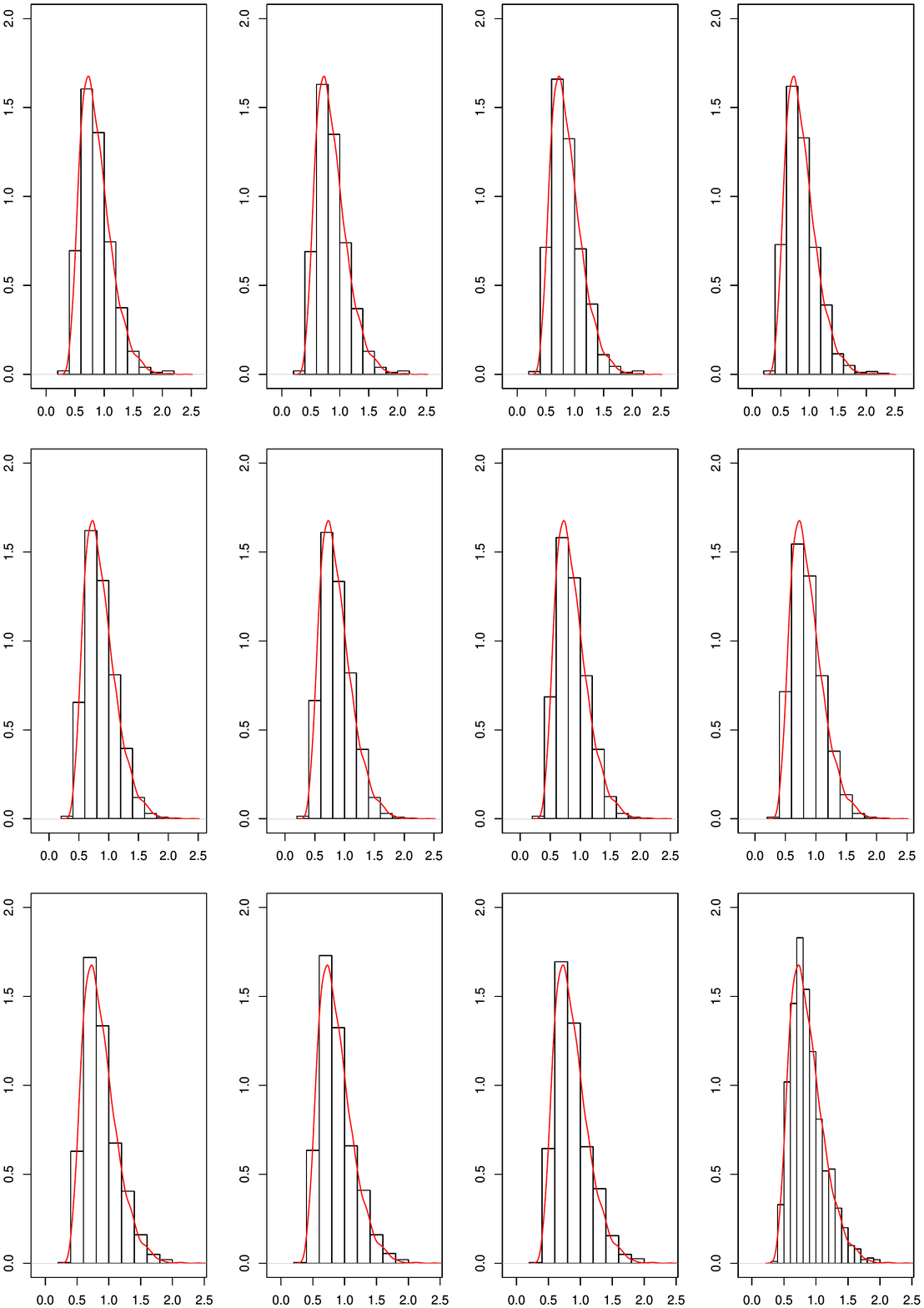}
\label{fig14}
\end{center}
\end{figure}
%%%%%%%%%%%%%%%%%%%%%%%%%
\begin{figure}[h]
\begin{center}
\caption{
Histograms of $T_{1,n}^\beta$ when { $n=8\times10^3, 1.25\times10^5$} and $10^6$ in order from the first row 
and $\beta$ changes from $1$ to $1.1, 1.5, 3$ and $5$ in oder from the first column ((ii) of Case 3). 
The red line is the probability density function of $\sup_{0\le s\le 1}|\boldsymbol B_1^0(s)|$.
}
\includegraphics[height=10cm, width=15cm]{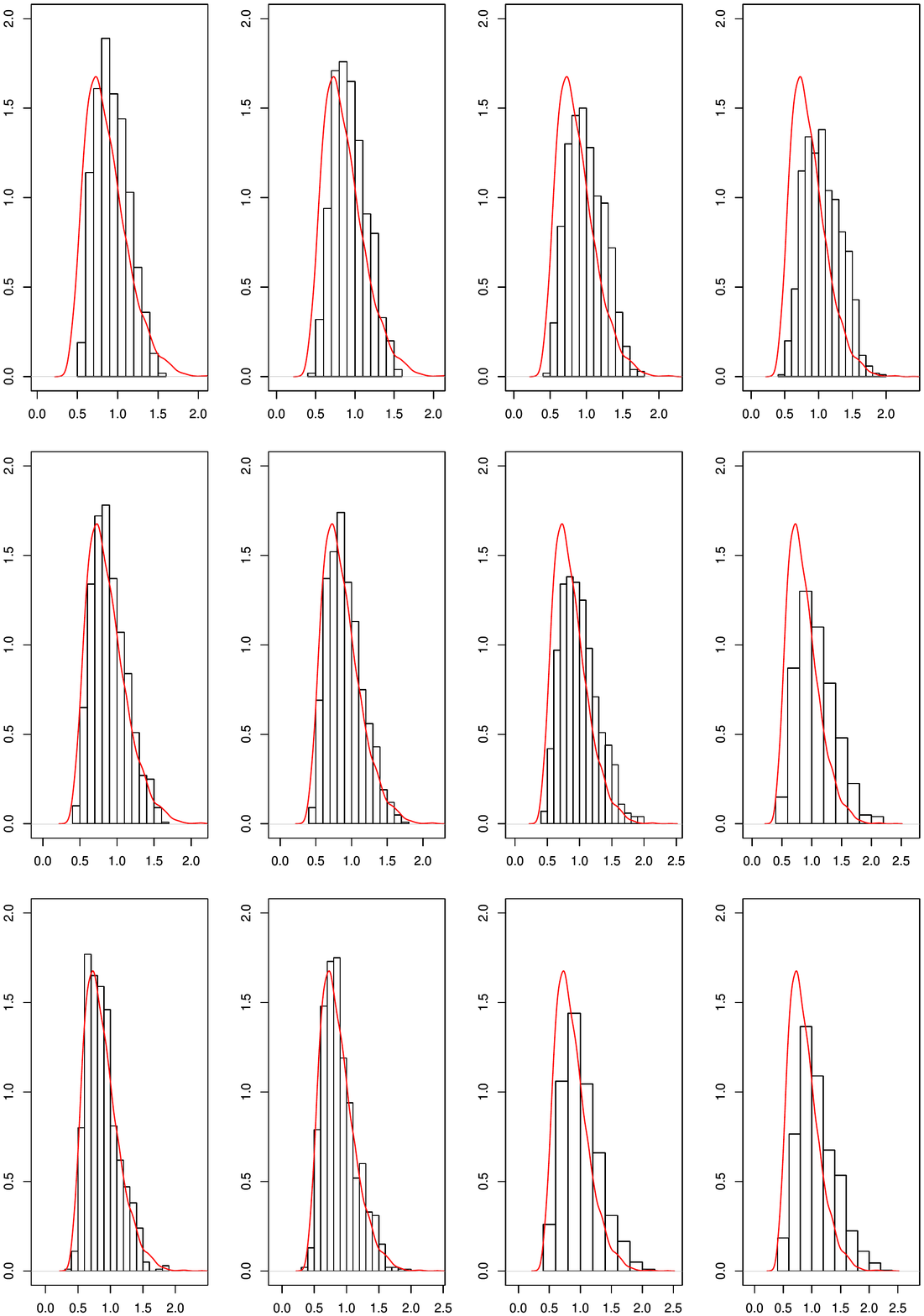}
\label{fig15}
%\end{center}
%\end{figure}
%%%
%\begin{figure}[h]
%\begin{center}
\caption{
Histograms of $T_{2,n}^\beta$ when { $n=8\times10^3, 1.25\times10^5$} and $10^6$ in order from the first row 
and $\beta$ changes from $1$ to $1.1, 1.5, 3$ and $5$ in oder from the first column ((ii) of Case 3). 
The red line is the probability density function of $\sup_{0\le s\le 1}\|\boldsymbol B_2^0(s)\|$.
}
\includegraphics[height=10cm, width=15cm]{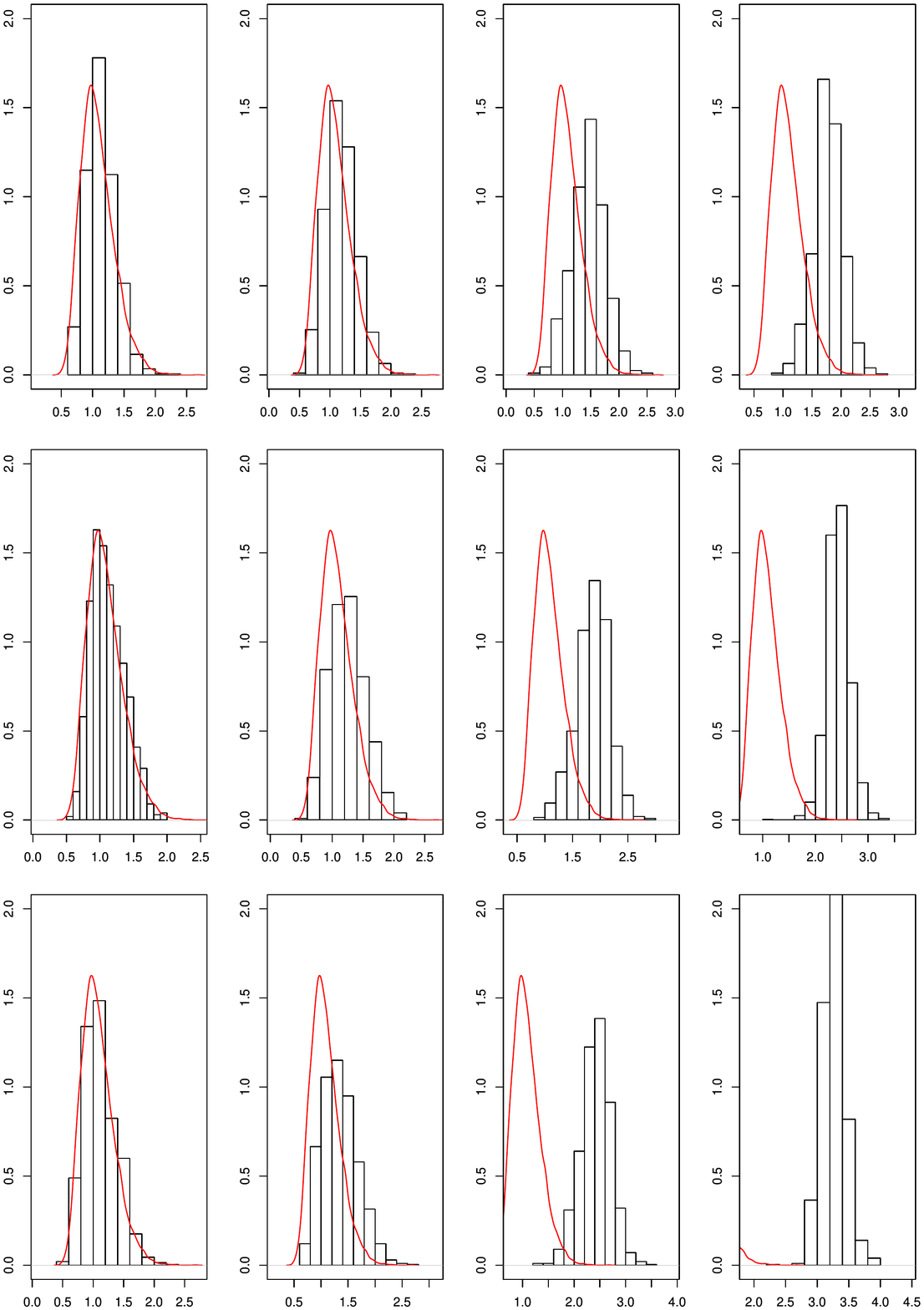}
\label{fig16}
\end{center}
\end{figure}
%%%%%%%%%%%%%%%%%%%%%%%%%
\begin{figure}[h]
\begin{center}
\caption{
Histograms of $T_n^\alpha$ when { $n=8\times10^3, 1.25\times10^5$} and $10^6$ in order from the first row 
and $\gamma$ changes from $1$ to $0.9, 0.5, 0$ and $-1$ in oder from the first column ((i) of Case 3). 
The red line is the probability density function of $\sup_{0\le s\le 1}|\boldsymbol B_1^0(s)|$.
}
\includegraphics[height=10cm, width=15cm]{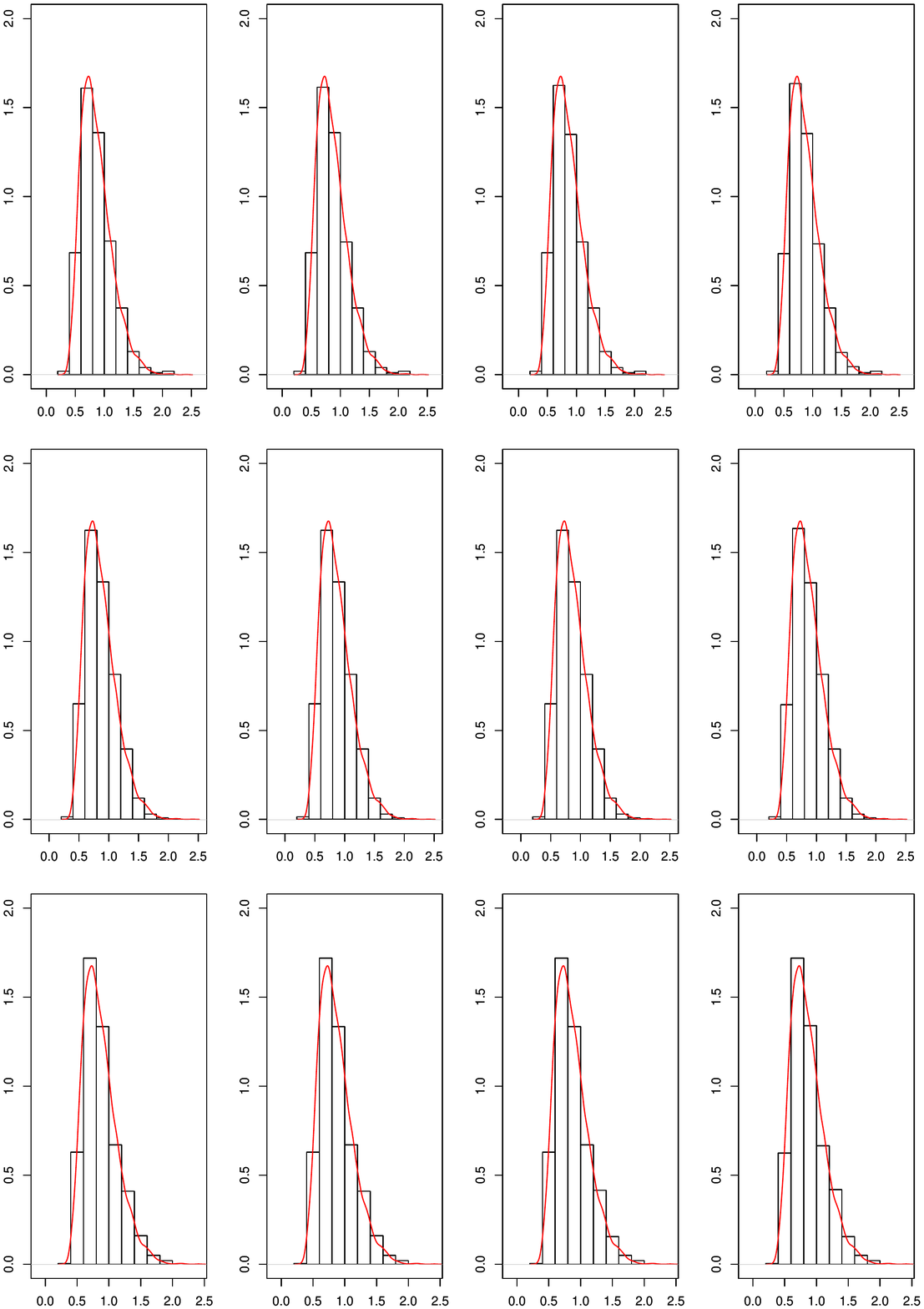}
\label{fig17}
%\end{center}
%\end{figure}
%%%
%\begin{figure}[h]
%\begin{center}
\caption{
Histograms of $T_{1,n}^\beta$ when { $n=8\times10^3, 1.25\times10^5$} and $10^6$ in order from the first row 
and $\gamma$ changes from $1$ to $0.9, 0.5, 0$ and $-1$ in oder from the first column ((ii) of Case 3). 
The red line is the probability density function of $\sup_{0\le s\le 1}|\boldsymbol B_1^0(s)|$.
}
\includegraphics[height=10cm, width=15cm]{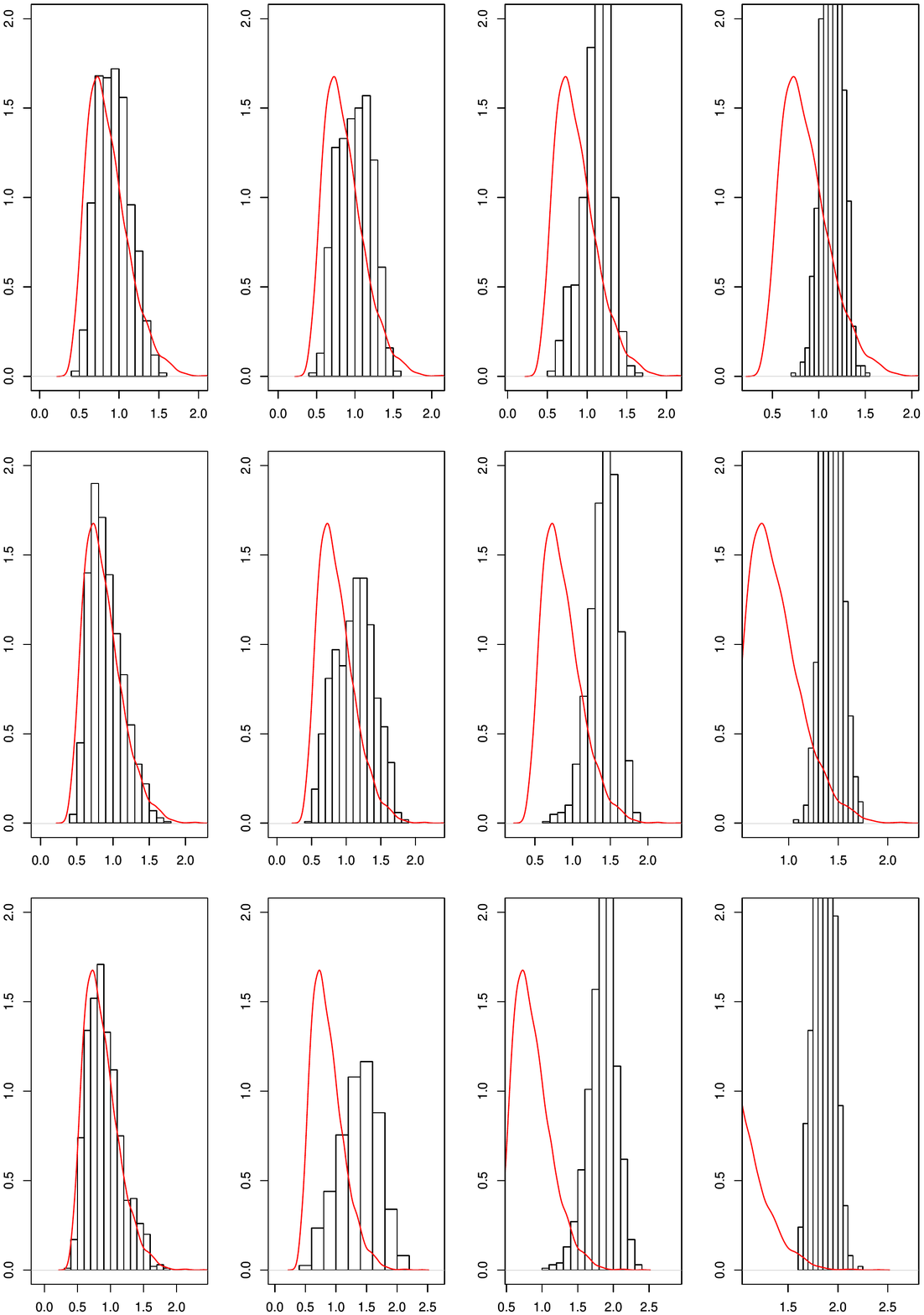}
\label{fig18}
\end{center}
\end{figure}
%%%%%%%%%%%%%%%%%%%%%%%%
\begin{figure}[h]
\begin{center}
\caption{
Histograms of $T_{2,n}^\beta$ when { $n=8\times10^3, 1.25\times10^5$} and $10^6$ in order from the first row 
and $\gamma$ changes from $1$ to $0.9, 0.5, 0$ and $-1$ in oder from the first column 
((iii) of Case 3). 
The red line is the probability density function of $\sup_{0\le s\le 1}\|\boldsymbol B_2^0(s)\|$.
}
\includegraphics[height=10cm, width=15cm]{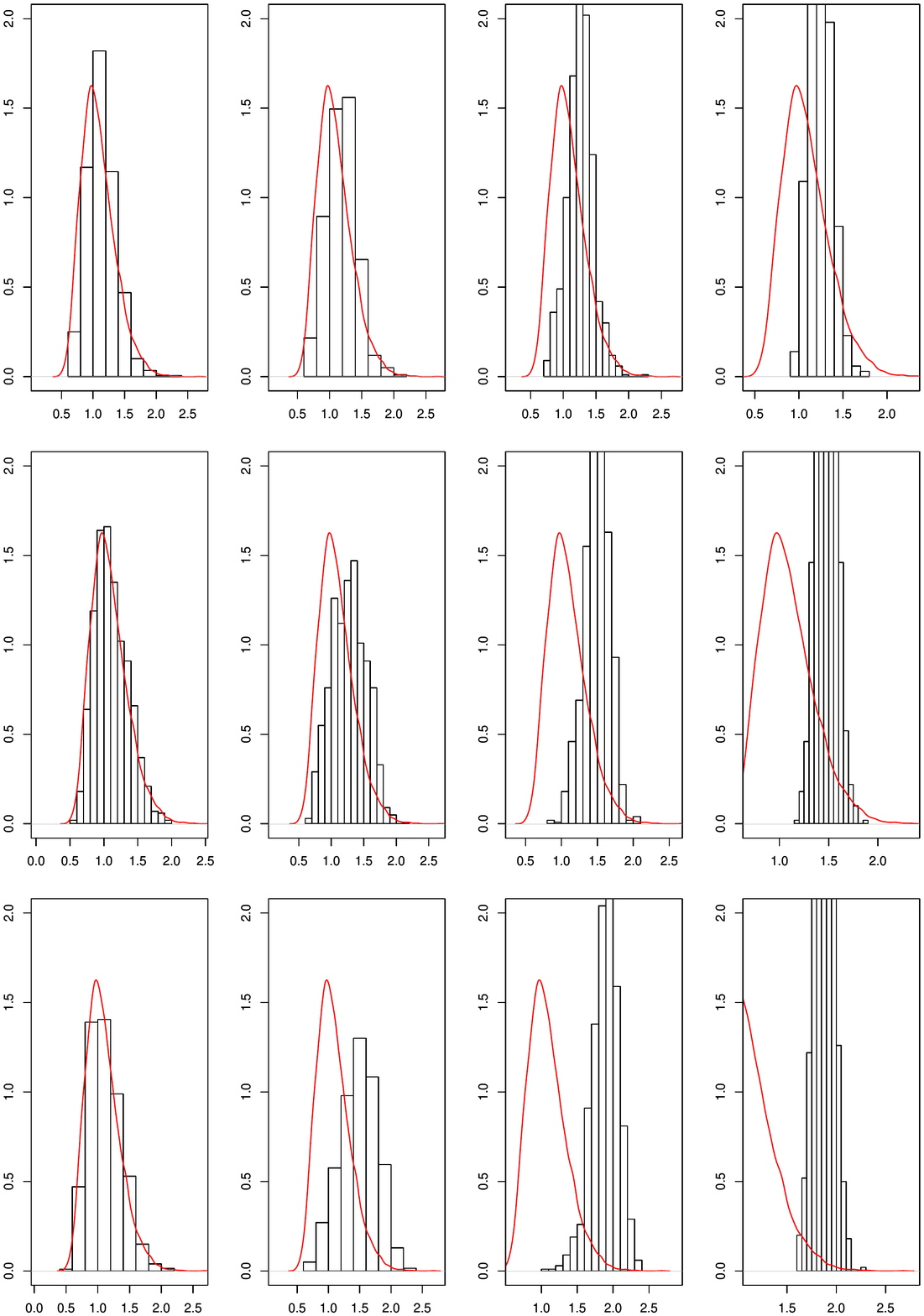}
\label{fig19}
%\end{center}
%\end{figure}
%%%%%%%%
%\begin{figure}[h]
%\begin{center}
\caption{
Histograms of $T_n^\alpha$ when { $n=8\times10^3, 1.25\times10^5$} and $10^6$ in order from the first row 
and $(\beta,\gamma)$ changes from $(1,1)$ to $(3,0.5), (3,0), (5,0.5)$ and $(5,0)$ in oder from the first column ((i) of Case 3). 
The red line is the probability density function of $\sup_{0\le s\le 1}|\boldsymbol B_1^0(s)|$.
}
\includegraphics[height=10cm, width=15cm]{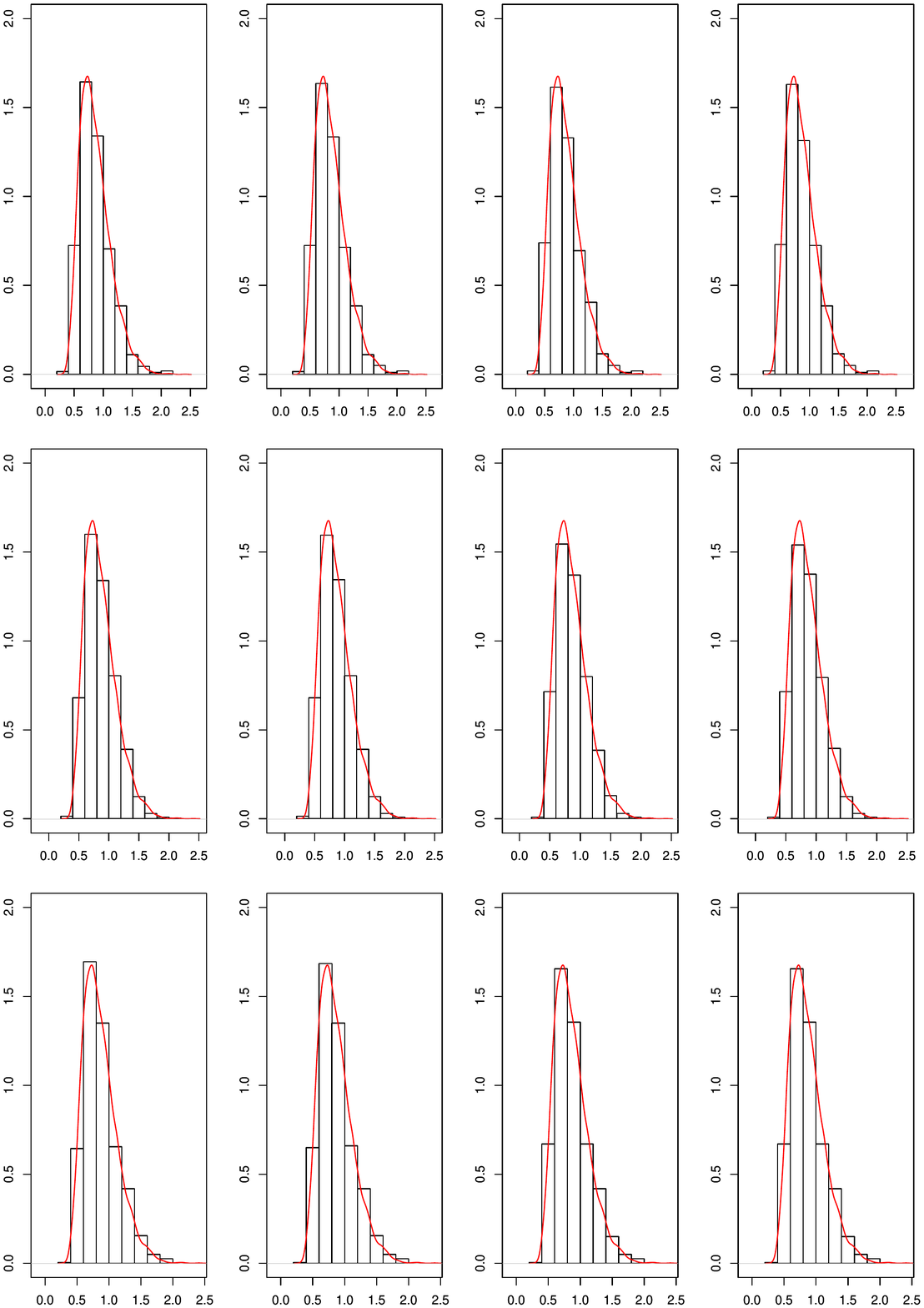}
\label{fig20}
\end{center}
\end{figure}
%%%%%%%%%%%%%%%%%%%%%%%%
\begin{figure}[h]
\begin{center}
\caption{
Histograms of $T_{1,n}^\beta$ when { $n=8\times10^3, 1.25\times10^5$} and $10^6$ in order from the first row 
and $(\beta,\gamma)$ changes from $(1,1)$ to $(3,0.5), (3,0), (5,0.5)$ and $(5,0)$ in oder from the first column ((ii) of Case 3). 
The red line is the probability density function of $\sup_{0\le s\le 1}|\boldsymbol B_1^0(s)|$.
}
\includegraphics[height=10cm, width=15cm]{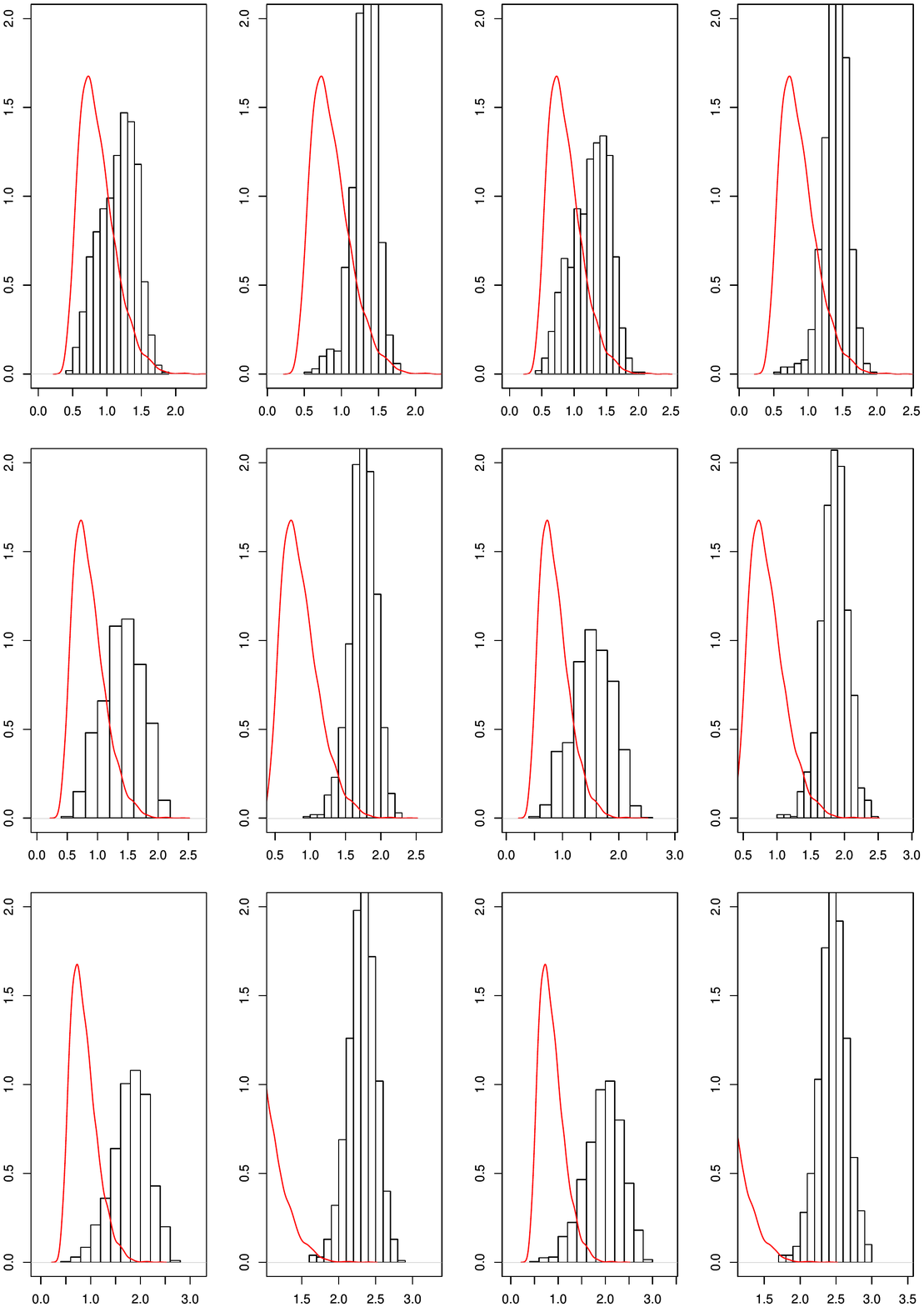}
\label{fig21}
%\end{center}
%\end{figure}
%%%%%%%%%%%%
%\begin{figure}[h]
%\begin{center}
\caption{
Histograms of $T_{2,n}^\beta$ when { $n=8\times10^3, 1.25\times10^5$} and $10^6$ in order from the first row 
and $(\beta,\gamma)$ changes from $(1,1)$ to $(3,0.5), (3,0), (5,0.5)$ and $(5,0)$ in oder from the first column ((iii) of Case 3). 
The red line is the probability density function of $\sup_{0\le s\le 1}\|\boldsymbol B_2^0(s)\|$.
}
\includegraphics[height=10cm, width=15cm]{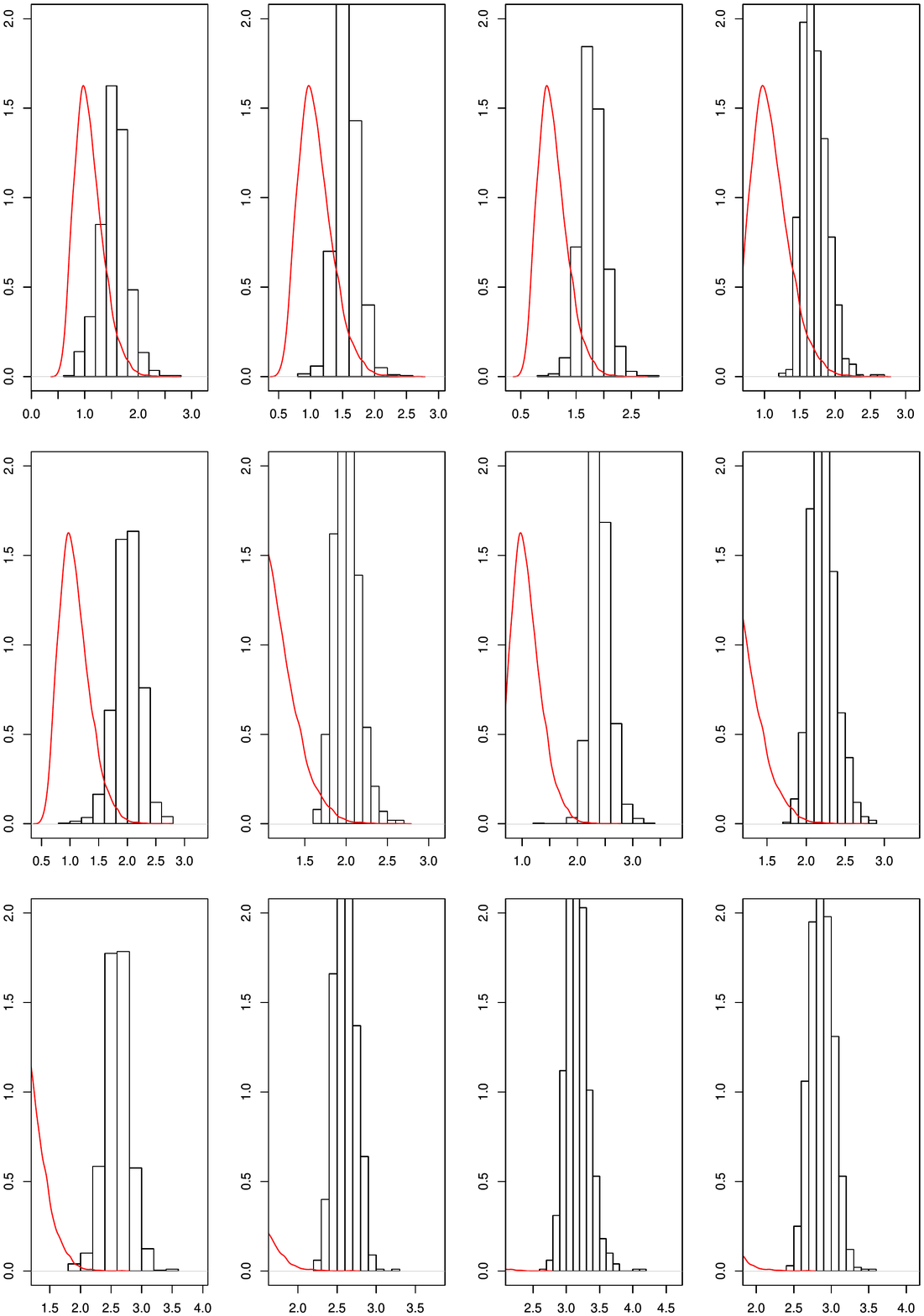}
\label{fig22}
\end{center}
\end{figure}
%%%%%%%%%%%%%%%%%%%%%%%%%%%%%%%%%%%%%%%%%%%%%%%%%%%%%%%%%%%%%%%%%%%%%%%%%%%%%%%%%%%%%%%%
\begin{figure}[h]
\begin{center}
\caption{
Empirical distribution functions of $T_n^\alpha$ when { $n=8\times10^3, 1.25\times10^5$} and $10^6$ in order from the first row 
and $(\alpha,\beta,\gamma)=(1,1,1), (0.5,1,0), (1.5,1.5,-1)$ and $(2,3,0.5)$ in oder from the first column (Case 1). 
The red line is the cumulative distribution function of $\sup_{0\le s\le 1}|\boldsymbol B_1^0(s)|$.
}
\includegraphics[height=10cm, width=15cm]{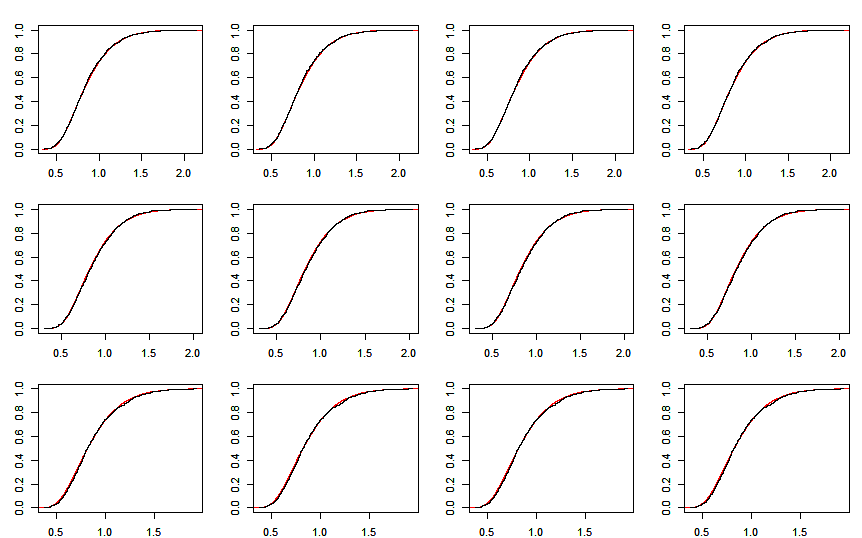}
\label{fig23}
%\end{center}
%\end{figure}
%%%
%\begin{figure}[h]
%\begin{center}
\caption{
Empirical distribution functions of $T_{1,n}^\beta$ when { $n=8\times10^3, 1.25\times10^5$} and $10^6$ in order from the first row 
and $(\alpha,\beta,\gamma)=(1,1,1), (0.5,1,0), (1.5,1.5,-1)$ and $(2,3,0.5)$ in oder from the first column (Case 1). 
The red line is the cumulative distribution function of $\sup_{0\le s\le 1}|\boldsymbol B_1^0(s)|$.
}
\includegraphics[height=10cm, width=15cm]{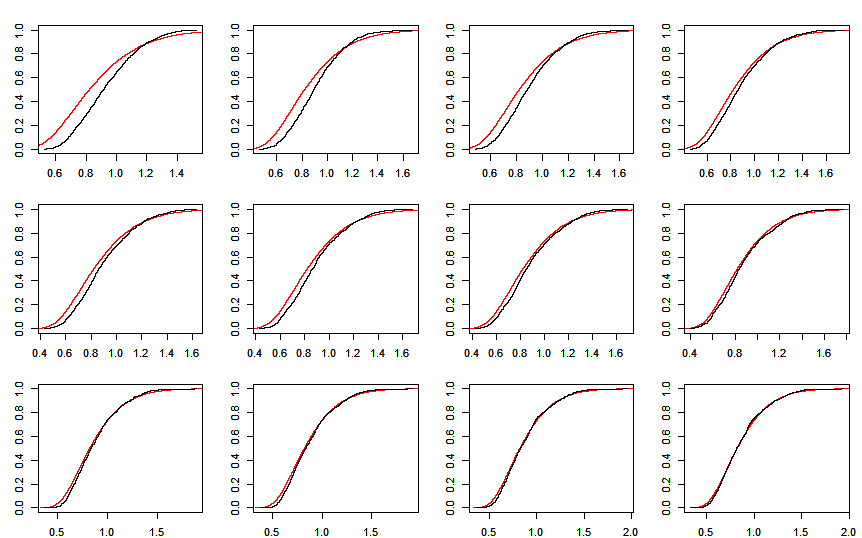}
\label{fig24}
\end{center}
\end{figure}
%%%%%%%%%%%%%%%%%%%%%%%
\begin{figure}[h]
\begin{center}
\caption{
Empirical distribution functions of $T_{2,n}^\beta$ when { $n=8\times10^3, 1.25\times10^5$} and $10^6$ in order from the first row 
and $(\alpha,\beta,\gamma)=(1,1,1), (0.5,1,0), (1.5,1.5,-1)$ and $(2,3,0.5)$ in oder from the first column (Case 1). 
The red line is the cumulative distribution function of 
$\sup_{0\le s\le 1}\|\boldsymbol B_2^0(s)\|$.
}
\includegraphics[height=10cm, width=15cm]{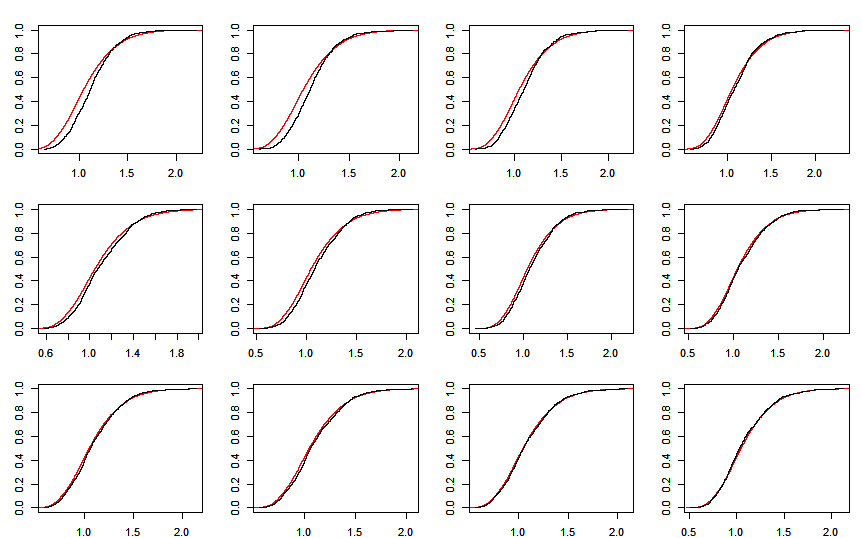}
\label{fig25}
%\end{center}
%\end{figure}
%%%
%\begin{figure}[h]
%\begin{center}
\caption{
Empirical distribution functions of $T_n^\alpha$ when { $n=8\times10^3, 1.25\times10^5$} and $10^6$ in order from the first row 
and $\alpha$ changes from $1$ to $1.01, 1.05, 1.1$ and $1.5$ in oder from the first column ((i) of Case 2). 
The red line is the cumulative distribution function of $\sup_{0\le s\le 1}|\boldsymbol B_1^0(s)|$.
}
\includegraphics[height=10cm, width=15cm]{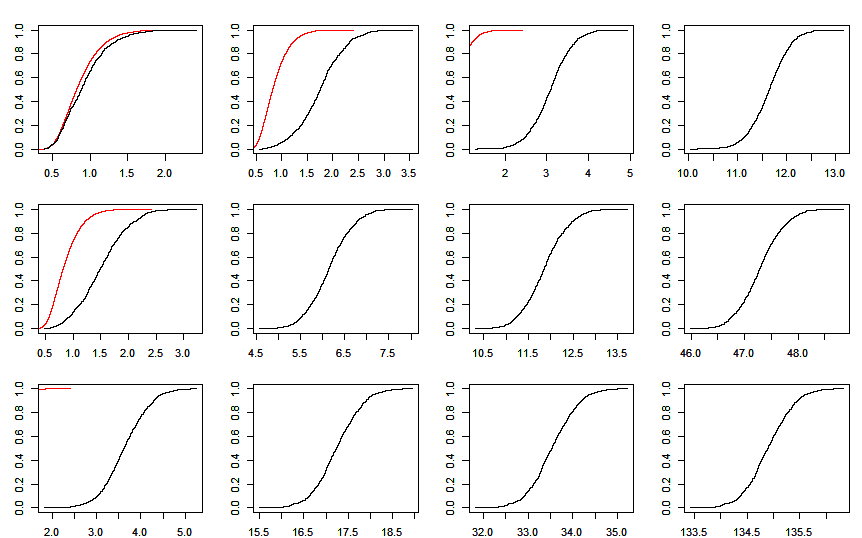}
\label{fig26}
\end{center}
\end{figure}
%%%%%%%%%%%%%%%%%%%%%%%
\begin{figure}[h]
\begin{center}
\caption{
Empirical distribution functions of $T_n^\alpha$ when { $n=8\times10^3, 1.25\times10^5$} and $10^6$ in order from the first row 
and $\beta$ changes from $1$ to $1.1, 1.5, 3$ and $5$ in oder from the first column ((i) of Case 3). 
The red line is the cumulative distribution function of $\sup_{0\le s\le 1}|\boldsymbol B_1^0(s)|$.
}
\includegraphics[height=10cm, width=15cm]{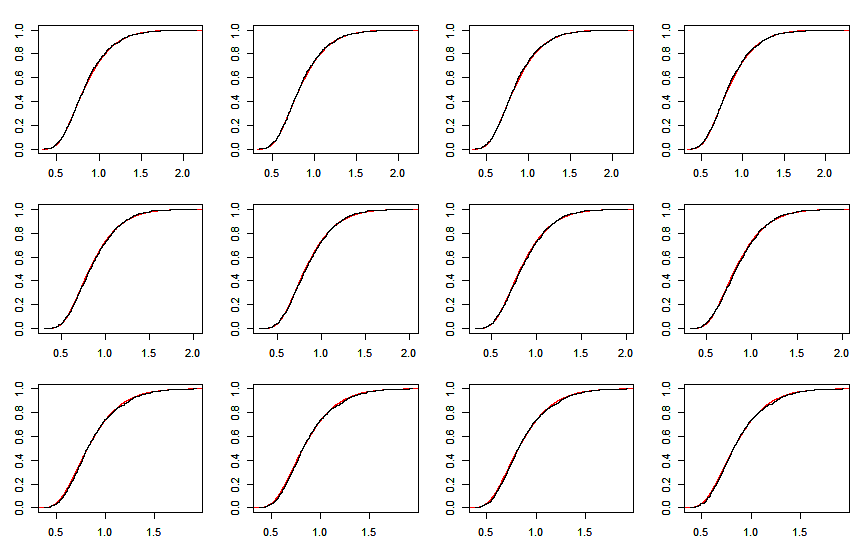}
\label{fig27}
%\end{center}
%\end{figure}
%\clearpage
%%%%%%%%%%%%%%%%%%%%%%%%%
%\begin{figure}[h]
%\begin{center}
\caption{
Empirical distribution functions of $T_{1,n}^\beta$ when { $n=8\times10^3, 1.25\times10^5$} and $10^6$ in order from the first row 
and $\beta$ changes from $1$ to $1.1, 1.5, 3$ and $5$ in oder from the first column ((ii) of Case 3). 
The red line is the cumulative distribution function of $\sup_{0\le s\le 1}|\boldsymbol B_1^0(s)|$.
}
\includegraphics[height=10cm, width=15cm]{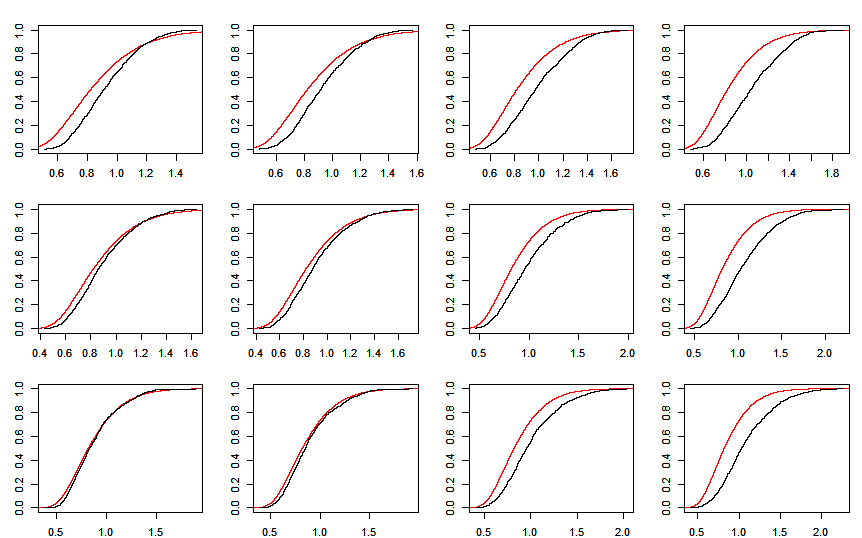}
\label{fig28}
\end{center}
\end{figure}
%%%%%%%%%%%%%
\begin{figure}[h]
\begin{center}
\caption{
Empirical distribution functions of $T_{2,n}^\beta$ when { $n=8\times10^3, 1.25\times10^5$} and $10^6$ in order from the first row 
and $\beta$ changes from $1$ to $1.1, 1.5, 3$ and $5$ in oder from the first column ((iii) of Case 3). 
The red line is the cumulative distribution function of 
$\sup_{0\le s\le 1}\|\boldsymbol B_2^0(s)\|$.
}
\includegraphics[height=10cm, width=15cm]{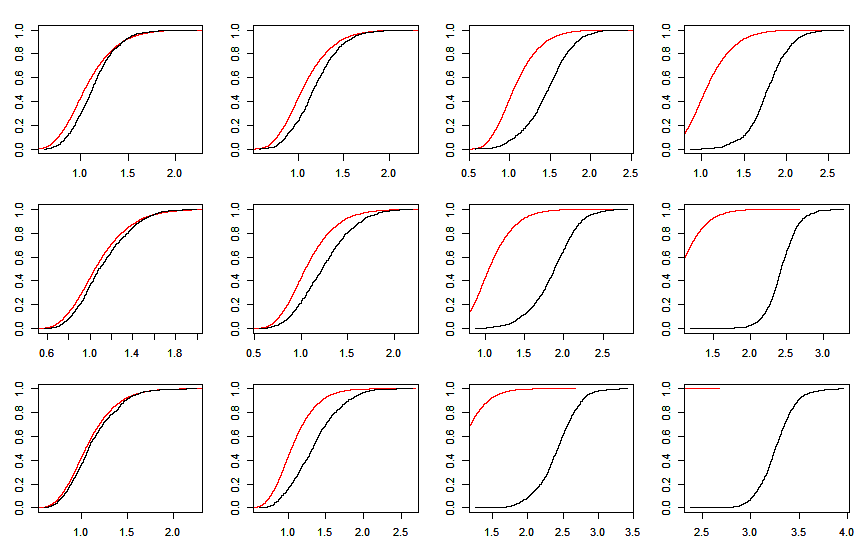}
\label{fig29}
%\end{center}
%\end{figure}
%%%%%%%%%%%%%%%%%%%%%%%%%
%\begin{figure}[h]
%\begin{center}
\caption{
Empirical distribution functions of $T_n^\alpha$ when { $n=8\times10^3, 1.25\times10^5$} and $10^6$ in order from the first row 
and $\gamma$ changes from $1$ to $0.9, 0.5, 0$ and $-1$ in oder from the first column ((i) of Case 3). 
The red line is the cumulative distribution function of $\sup_{0\le s\le 1}|\boldsymbol B_1^0(s)|$.
}
\includegraphics[height=10cm, width=15cm]{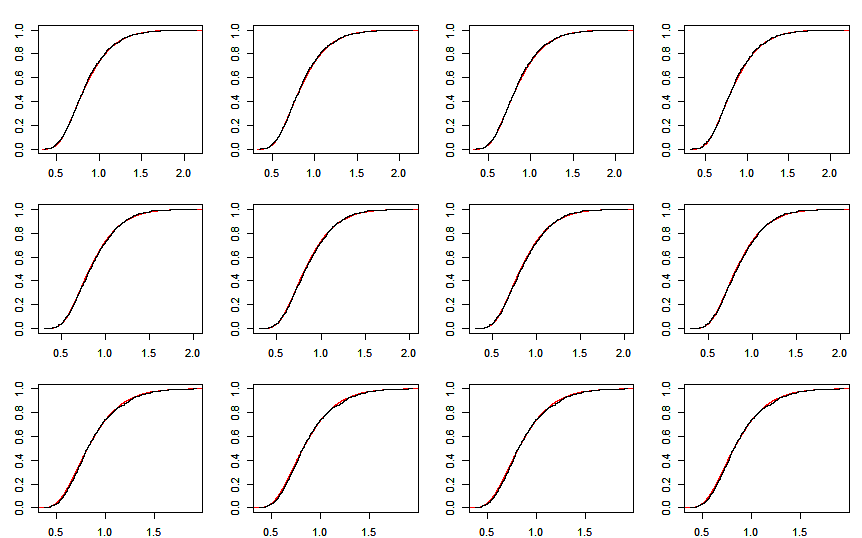}
\label{fig30}
\end{center}
\end{figure}
\clearpage
%%%%%%%%%%%%%%%%%%
\begin{figure}[h]
\begin{center}
\caption{
Empirical distribution functions of $T_{1,n}^\beta$ when { $n=8\times10^3, 1.25\times10^5$} and $10^6$ in order from the first row 
and $\gamma$ changes from $1$ to $0.9, 0.5, 0$ and $-1$ in oder from the first column ((ii) of Case 3). 
The red line is the cumulative distribution function of $\sup_{0\le s\le 1}|\boldsymbol B_1^0(s)|$.
}
\includegraphics[height=10cm, width=15cm]{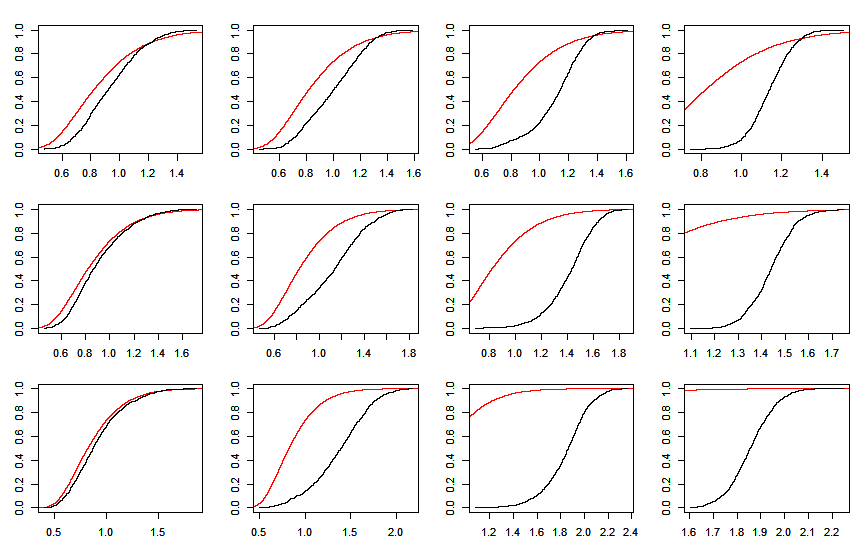}
\label{fig31}
%\end{center}
%\end{figure}
%%%%%%%%%%%%%%%%%%%%%%%%
%\begin{figure}[h]
%\begin{center}
\caption{
Empirical distribution functions of $T_{2,n}^\beta$ when { $n=8\times10^3, 1.25\times10^5$} and $10^6$ in order from the first row 
and $\gamma$ changes from $1$ to $0.9, 0.5, 0$ and $-1$ in oder from the first column ((iii) of Case 3). 
The red line is the probability density function of $\sup_{0\le s\le 1}\|\boldsymbol B_2^0(s)\|$ obtained by simulation.
}
\includegraphics[height=10cm, width=15cm]{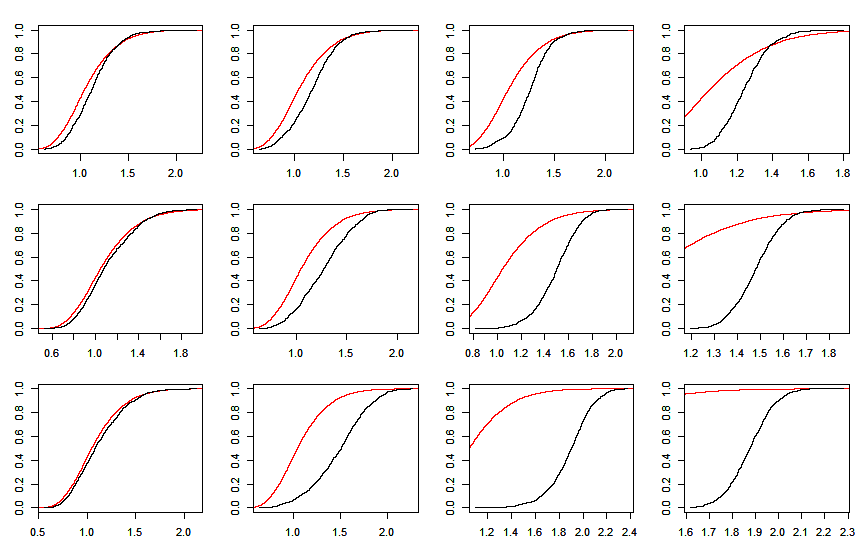}
\label{fig32}
\end{center}
\end{figure}
%%%%%%%%%%%%%%%%%%%%%%%%%%%%%%%%%%
\begin{figure}[h]
\begin{center}
\caption{
Empirical distribution functions of $T_n^\alpha$ when { $n=8\times10^3, 1.25\times10^5$} and $10^6$ in order from the first row 
and $(\beta,\gamma)$ changes from $(1,1)$ to $(3,0.5), (3,0), (5,0.5)$ and $(5,0)$ in oder from the first column ((i) of Case 3). 
The red line is the cumulative distribution function of $\sup_{0\le s\le 1}|\boldsymbol B_1^0(s)|$.
}
\includegraphics[height=10cm, width=15cm]{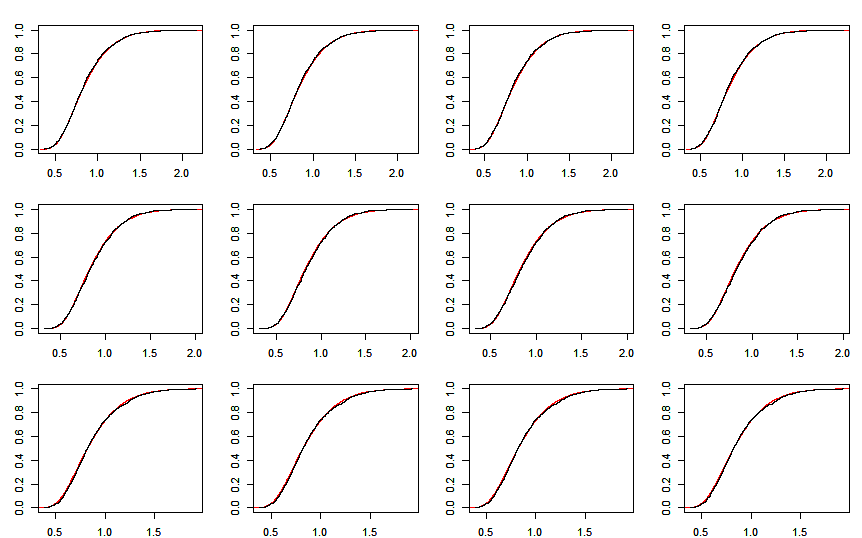}
\label{fig33}
%\end{center}
%\end{figure}
%%%%%%%%%%%%%%%%%%%%%%%%
%\begin{figure}[h]
%\begin{center}
\caption{
Empirical distribution functions of $T_{1,n}^\beta$ when { $n=8\times10^3, 1.25\times10^5$} and $10^6$ in order from the first row 
and $(\beta,\gamma)$ changes from $(1,1)$ to $(3,0.5), (3,0), (5,0.5)$ and $(5,0)$ in oder from the first column ((ii) of Case 3). 
The red line is the cumulative distribution function of $\sup_{0\le s\le 1}|\boldsymbol B_1^0(s)|$.
}
\includegraphics[height=10cm, width=15cm]{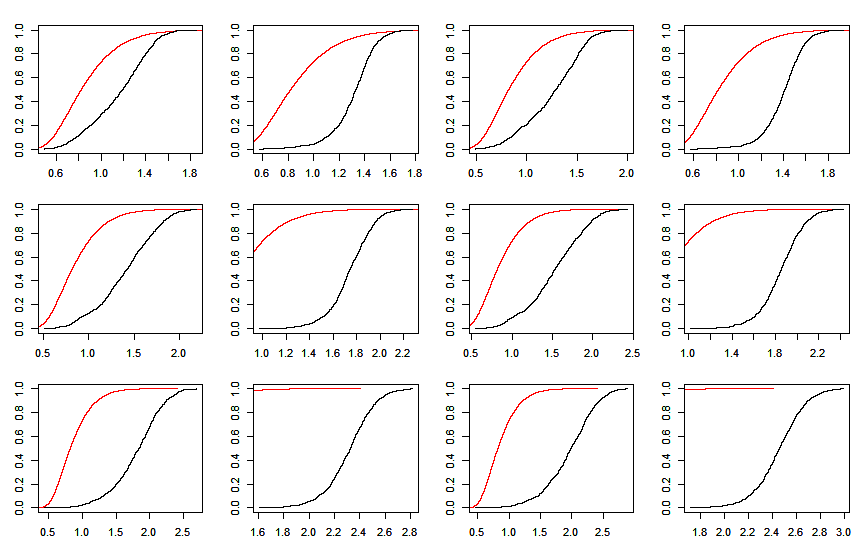}
\label{fig34}
\end{center}
\end{figure}
%%%%%%%%%%%%%%%%%%%%%%
\clearpage
\begin{figure}[h]
\begin{center}
\caption{
Empirical distribution functions of $T_{2,n}^\beta$ when { $n=8\times10^3, 1.25\times10^5$} and $10^6$ in order from the first row 
and $(\beta,\gamma)$ changes from $(1,1)$ to $(3,0.5), (3,0), (5,0.5)$ and $(5,0)$ in oder from the first column ((iii) of Case 3). 
The red line is the cumulative distribution function of 
$\sup_{0\le s\le 1}\|\boldsymbol B_2^0(s)\|$.
}
\includegraphics[height=10cm, width=15cm]{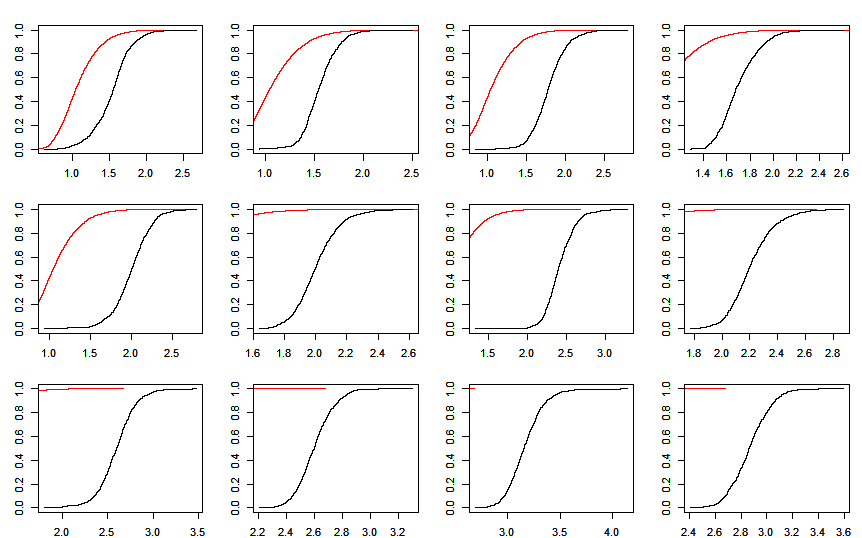}
\label{fig35}
\end{center}
\end{figure}
%%%%%%%%%%%%%%%%%%%%%%%%%%%%%%%%%%%%%%%%%%%%%%%%%%%%%%%%%%

%\clearpage
%\input{change_point2-5}
\section{Proofs}
Let $\GG=\sigma\left[\{W_s\}_{s\le t_i^n}\right]$, and 
$C$, $C'>0$ denote universal constants.
If 
$f$ is a function on $\mathbb R^d\times \Theta$,
%we denote by 
{ $f_{i-1}(\theta)$
denotes 
the value $f(\Xs,\theta)$}.
If $\{u_n\}$ is a positive sequence, $R$ denotes a function 
on $\mathbb R^d\times\mathbb R_{+}\times\Theta$
for which there exists a constant $C>0$ such that 
\begin{align*}
\sup_{\theta\in\Theta}\| R(x,u_n,\theta)\|\le u_n C(1+\|x\|)^C.
\end{align*}
Let $R_{i-1}(u_n,\theta)=R(\Xs,u_n,\theta)$.
%%%%%%%%%%%%%%%%%%%%%%%%%%%%%%%%%%%%%%%%%%%%%%%%%%%%%%%%%%%%%%%%%%%%%%%%%%%
%%%%%%%%%%%%%%%%%%%%%%%%%%%%%%%%%%%%%%%%%%%%%%%%%%%%%%%%%%%%%%%%%%%%%%%%%%%
\begin{lem}[Kessler,1997]\label{ke7}
Suppose that {\textbf{[A1]}}-{\textbf{[A4]}} hold. 
Then
for $\ell, \ell_1, \ell_2, \ell_3, \ell_4 =1, \ldots, d$, 
\begin{align}
&\Ep\left[(\DeX)^\ell\middle|\GG\right]
=h_nb_{i-1}^\ell(\beta)+\Rd,\label{ke7-1}\\
%%%
&\Ep\left[(\DeX)^{\ell_1}(\DeX)^{\ell_2}\middle|\GG\right]
=h_nA^{\ell_1,\ell_2}_{i-1}(\alpha)+\Rd,\label{ke7-2}\\
&\Ep\left[
(\DeX)^{\ell_1}(\DeX)^{\ell_2}(\DeX)^{\ell_3}(\DeX)^{\ell_4}\middle|\GG
\right]\nonumber\\
&=h_n^2\left(A^{\ell_1,\ell_2}_{i-1}A^{\ell_3,\ell_4}_{i-1}(\alpha)
+A^{\ell_1,\ell_3}_{i-1}A^{\ell_2,\ell_4}_{i-1}(\alpha)
+A^{\ell_1,\ell_4}_{i-1}A^{\ell_2,\ell_3}_{i-1}(\alpha)\right)
+\Rt.\label{ke7-3}
\end{align}
\end{lem}
{ 
Let 
\begin{align*}
\eta_i&=\tr\left(A^{-1}_{i-1}(\alpha_0)
\frac{(\Delta X_i)^{\otimes2}}{h_n}\right),\quad
%%%
\kappa(x,\alpha)=
1_d^\TT a^{-1}(x,\alpha)
,\quad
%%%
\xi_i=\kappa_{i-1}(\alpha_0)(\DeX-h_nb_{i-1}(\beta_0)),\\
%%%
\zeta_i&=\partial_\beta b_{i-1}(\beta_0)^\TT A_{i-1}^{-1}(\alpha_0)
\left(\DeX-h_nb_{i-1}(\beta_0)\right).
\end{align*}

\begin{lem} 
Suppose that {\textbf{[A1]}}-{\textbf{[A4]}} hold. 
Then,
% we have
%%%
\begin{align}
\EEt\left[\eta_i\middle|\GG\right]&=d+\Ri,\label{keq1}\\
%%%
\EEt\left[\eta_i^2\middle|\GG\right]
&=d^2+2d+\Ri,\\
%%%
\EEt\left[\xi_i\middle|\GG\right]&=\Rd,\\
%%%
\EEt\left[\xi_i^2\middle|\GG\right]
&=dh_n+\Rd,\\
%%%
\EEt\left[\zeta_i\middle|\GG\right]
&=\Rd,\\
%%%
\EEt\left[\zeta_i\zeta_i^\TT\middle|\GG\right]
&=h_n\partial_\beta b_{i-1}(\beta_0)^\TT A_{i-1}^{-1}(\alpha_0)
\partial_\beta b_{i-1}(\beta_0)+\Rd.\label{keq4}
\end{align}
\end{lem}
}
%%%%%%%%%%%%%%%%%%%%%%%%%%%%%%%%%%%%%%%%%%%%%%%%%%%%%%%%%%%%%%%%%%%%%
\begin{proof}
%%%%%%%%%%%%%
Let us prove \eqref{keq1} to \eqref{keq4}.
We have, from \eqref{ke7-1} to \eqref{ke7-3}, 
\begin{align*}
\EEt[\eta_i|\GG]
&=\EEt\left[
\tr\left(A^{-1}_{i-1}(\alpha_0)\frac{(\DeX)^{\otimes2}}{h_n}\right)
\middle|\GG\right]
=\tr\,\EEt\left[
A^{-1}_{i-1}(\alpha_0)\frac{(\DeX)^{\otimes2}}{h_n}
\middle|\GG\right]\\
&=\tr\left(A^{-1}_{i-1}(\alpha_0)\EEt\left[
\frac{(\DeX)^{\otimes2}}{h_n}
\middle|\GG\right]\right)
=\tr\left(A^{-1}_{i-1}(\alpha_0)(A_{i-1}(\alpha_0)+\Ri)\right)\\
&=d+\Ri,
\end{align*}
%%%%%%%%
\begin{align*}
&\EEt\left[\eta_i^2\middle|\GG\right]\\
&=\EEt\left[
\sum_{\ell_1,\ell_2=1}^d
\sum_{\ell_3,\ell_4=1}^d
(A_{i-1}^{-1})^{\ell_1,\ell_2}(A_{i-1}^{-1})^{\ell_3,\ell_4}(\alpha_0)
\frac{(\DeX)^{\ell_1}(\DeX)^{\ell_2}(\DeX)^{\ell_3}(\DeX)^{\ell_4}}{h_n^2}
\middle|\GG\right]\\
%%%
&=
\sum_{\ell_1,\ell_2=1}^d
\sum_{\ell_3,\ell_4=1}^d
(A_{i-1}^{-1})^{\ell_1,\ell_2}(A_{i-1}^{-1})^{\ell_3,\ell_4}(\alpha_0)
\EEt\left[
\frac{(\DeX)^{\ell_1}(\DeX)^{\ell_2}(\DeX)^{\ell_3}(\DeX)^{\ell_4}}{h_n^2}
\middle|\GG\right]\\
%%%
&=
\sum_{\ell_1,\ell_2=1}^d
\sum_{\ell_3,\ell_4=1}^d
(A_{i-1}^{-1})^{\ell_1,\ell_2}(A_{i-1}^{-1})^{\ell_3,\ell_4}
\left(A^{\ell_1,\ell_2}_{i-1}A^{\ell_3,\ell_4}_{i-1}
+A^{\ell_1,\ell_3}_{i-1}A^{\ell_2,\ell_4}_{i-1}
+A^{\ell_1,\ell_4}_{i-1}A^{\ell_2,\ell_3}_{i-1}\right)(\alpha_0)
+\Ri\\
%%%
&=\left(\tr\left(A^{-1}A\right)\right)^2
+2\tr(A^{-1}AA^{-1}A)+\Ri\\
%%%
&=d^2+2d+\Ri,
\end{align*}
%%%
\begin{align*}
\EEt\left[\xi_i\middle|\GG\right]
&=\kappa_i(\alpha_0)\EEt[\DeX
-h_nb_{i-1}(\beta_0)|\GG]
=\Rd,\\
%%%%
\EEt\left[\xi_i^2\middle|\GG\right]
&=\kappa_{i-1}(\alpha_0)
\EEt\left[
\left(\DeX-h_nb_{i-1}(\beta_0)\right)
\left(\DeX-h_nb_{i-1}(\beta_0)\right)^\TT
\middle|\GG\right]\kappa_{i-1}^\TT(\alpha_0)\\
%%%
&=
\kappa_{i-1}(\alpha_0)\left(
\EEt\left[
(\DeX)^{\otimes2}
\middle|\GG\right]
-h_n\EEt\left[\DeX\middle|\GG\right]b_{i-1}^\TT(\beta_0)\right.\\
&\left.\qquad-h_nb_{i-1}(\beta_0)\EEt\left[\DeX\middle|\GG\right]^\TT
+h_n^2b_{i-1}(\beta_0)b_{i-1}^\TT(\beta_0)\right)\kappa_{i-1}^\TT(\alpha_0)\\
%%%
&=\kappa_{i-1}(\alpha_0)(h_nA_{i-1}(\alpha_0)+\Rd)\kappa_{i-1}^\TT(\alpha_0)
\\
%%%
&=h_n\kappa A\kappa^\TT+\Rd
\\
&=dh_n+\Rd,\\
%\end{align*}
%%%
%\begin{align*}
\EEt\left[\zeta_i\middle|\GG\right]
&=\partial_\beta b_{i-1}(\beta_0)^\TT A_{i-1}^{-1}(\alpha_0)
\EEt\left[\DeX-h_nb_{i-1}(\beta_0)\middle|\GG\right]
=\Rd,\\
%%%
\EEt\left[\zeta_i\zeta_i^\TT\middle|\GG\right]
&=\partial_\beta b_{i-1}(\beta_0)^\TT A_{i-1}^{-1}(\alpha_0)
\EEt\left[(\DeX-h_nb_{i-1}(\beta_0))^{\otimes2}\middle|\GG\right]
A_{i-1}^{-1}(\alpha_0)\partial_\beta b_{i-1}(\beta_0)\\
&=
\partial_\beta b_{i-1}(\beta_0)^\TT A_{i-1}^{-1}(\alpha_0)
(h_nA_{i-1}^{-1}(\alpha_0)+\Rd)
A_{i-1}^{-1}(\alpha_0)\partial_\beta b_{i-1}(\beta_0)\\
&=
h_n\partial_\beta b_{i-1}(\beta_0)^\TT A_{i-1}^{-1}(\alpha_0)
\partial_\beta b_{i-1}(\beta_0)+\Rd.
\end{align*}
\end{proof}
%%%%%%%%%%%%%%%%%%%%%%%%%%%%%%%%%%%%%%%%%%%%%%%%%%%%%%%%%%%%%%%%%%%%%%%%%%%%%%%%%%%
%%%%%%%%%%%%%%%%%%%%%%%%%%%%%%%%%%%%%%%%%%%%%%%%%%%%%%%%%%%%%%%%%%%%%%%%%%%%%%%%%%%
\begin{lem}[Song and Lee, 2009]\label{s-l4.3}
Suppose that {\textbf{[A1]}}, {\textbf{[A2]}}, {\textbf{[A5]}} hold and 
a function $f$ on $\mathbb R^d\times\Theta$ 
satisfies
\begin{enumerate}
\renewcommand{\labelenumi}{(\roman{enumi})}
\item $f$ is continuous in $\theta\in\Theta$ for all $x\in\mathbb R^d$;
\item $\partial_xf$ exists and $f,\partial_xf$ are of polynomial growth 
in $x\in\mathbb R^d$ uniformly $\theta\in\Theta$.
\end{enumerate}
{ Moreover}, if $nh_n^r\lto\infty$ for some $1< r<2$, 
{ then}, under $H_0^\alpha$ (or {\textbf{[A7]}} and $H_0^\beta$), 
as $nh_n^2\lto0$,
\begin{align*}
\max_{[n^{1/r}]\le k\le n}\sup_{\theta\in\Theta}
\left\|\frac1k\sum_{i=1}^kf(X_{t_{i-1}^n},\theta)
-\int_{\mathbb R^d} f(x,\theta)\dd\mu_{\theta_0}(x)\right\|
\ato0.
\end{align*}
\end{lem}
%%%%%%%%%%%%%%%%%%%%%%%%%%%%%%%%%%%%%%%%%%%%%%%%%%%%%%%%%%%%%%%%%%%%%%%%%%%%%%%%%%%%%%
%%%%%%%%%%%%%%%%%%%%%%%%%%%%%%%%%%%%%%%%%%%%%%%%%%%%%%%%%%%%%%%%%%%%%%%%%%%%%%%%%%%%%%
\begin{lem}\label{lem1}
Suppose that 
{\textbf{[A1]}}, {\textbf{[A2]}}, {\textbf{[A5]}} hold and
$f$ satisfies the conditions (i), (ii) in Lemma \ref{s-l4.3}. 
Then, 
under $H_0^\alpha$ (or {\textbf{[A7]}} and $H_0^\beta$), 
as $nh_n^2\lto0$,
\begin{align*}
\frac{1}{n}
\max_{1\le k\le n}\sup_{\theta\in\Theta}\left\|
\sum_{i=1}^kf(X_{t_{i-1}^n},\theta)-\frac kn\sum_{i=1}^n
f(X_{t_{i-1}^n},\theta)\right\|=o_p(1).
\end{align*}
\end{lem}
%%%%%%%%%%%%%%%
\begin{proof}
If $nh_n\lto\infty,\ nh_n^2\lto0$, there exists $1< r<2$ such that 
$nh_n^r\lto\infty$. 
Then, { it follows from Lemma \ref{s-l4.3} that}
\begin{align*}
&\max_{1\le k\le [n^{1/r}]}\sup_{\theta\in\Theta}
\frac1n\left\|\sum_{i=1}^kf_{i-1}(\theta)
-\frac kn\sum_{i=1}^nf_{i-1}(\theta)
\right\|\\
&\le\frac1n\left(
\sum_{i=1}^{[n^{1/r}]}\sup_{\theta\in\Theta}\|f_{i-1}(\theta)\|
+\frac{[n^{1/r}]}{n}\sum_{i=1}^{n}\sup_{\theta\in\Theta}\|f_{i-1}(\theta)\|
\right)\\
&=\frac{[n^{1/r}]}n\left(
\frac{1}{[n^{1/r}]}\sum_{i=1}^{[n^{1/r}]}\sup_{\theta\in\Theta}\|f_{i-1}(\theta)\|
+\frac{1}{n}\sum_{i=1}^{n}\sup_{\theta\in\Theta}\|f_{i-1}(\theta)\|
\right)
=o_p(1)
\end{align*}
and
\begin{align*}
&\max_{[n^{1/r}]\le k\le n}\sup_{\theta\in\Theta}
\frac1n\left\|\sum_{i=1}^kf_{i-1}(\theta)
-\frac kn\sum_{i=1}^nf_{i-1}(\theta)
\right\|\\
&=\max_{[n^{1/r}]\le k\le n}\sup_{\theta\in\Theta}
\frac kn\left\|\frac1k\sum_{i=1}^kf_{i-1}(\theta)
-\frac 1n\sum_{i=1}^nf_{i-1}(\theta)
\right\|\\
&\le\max_{[n^{1/r}]\le k\le n}\sup_{\theta\in\Theta}
\left\|\frac1k\sum_{i=1}^kf_{i-1}(\theta)
-\int_{\mathbb R^d} f(x,\theta)\dd\mu_{\theta_0}(x)
-\frac 1n\sum_{i=1}^nf_{i-1}(\theta)
+\int_{\mathbb R^d} f(x,\theta)\dd\mu_{\theta_0}(x)
\right\|\\
&\le2\max_{[n^{1/r}]\le k\le n}\sup_{\theta\in\Theta}
\left\|\frac1k\sum_{i=1}^kf_{i-1}(\theta)
-\int_{\mathbb R^d} f(x,\theta)\dd\mu_{\theta_0}(x)
\right\|
\ato0.
\end{align*}
Therefore, we obtain 
\begin{align*}
&\frac{1}{n}
\max_{1\le k\le n}\sup_{\theta\in\Theta}\left\|
\sum_{i=1}^kf_{i-1}(\theta)-\frac kn\sum_{i=1}^n
f_{i-1}(\theta)\right\|\\
&\le\frac{1}{n}
\max_{1\le k\le [n^{1/r}]}\sup_{\theta\in\Theta}\left\|
\sum_{i=1}^kf_{i-1}(\theta)-\frac kn\sum_{i=1}^n
f_{i-1}(\theta)\right\|
+\frac{1}{n}
\max_{[n^{1/r}]\le k\le n}\sup_{\theta\in\Theta}\left\|
\sum_{i=1}^kf_{i-1}(\theta)-\frac kn\sum_{i=1}^n
f_{i-1}(\theta)\right\|\\
&=o_p(1).
\end{align*}
\end{proof}
%%%%%%%%%%%%%%%%%%%%%%%%%%%%%%%%%%%%%%%%%%%%%%%%%%%%%%%%%%%%%%%%%%%%%%%%%%%%%%%%%%%%%%
%%%%%%%%%%%%%%%%%%%%%%%%%%%%%%%%%%%%%%%%%%%%%%%%%%%%%%%%%%%%%%%%%%%%%%%%%%%%%%%%%%%%%%
\begin{lem}\label{lem2}
Suppose that 
{\textbf{[A1]}}-{\textbf{[A5]}} hold and
$f$ satisfies the conditions (i), (ii) in Lemma \ref{s-l4.3}. 
Then, 
under $H_0^\alpha$ (or {\textbf{[A7]}} and $H_0^\beta$), as $nh_n^2\lto0$,
\begin{align}
&\frac{1}{nh_n}
\max_{1\le k\le n}\left\|
\sum_{i=1}^kf(X_{t_{i-1}^n},\theta_0)(\DeX)^\ell
-\frac kn\sum_{i=1}^nf(X_{t_{i-1}^n},\theta_0)(\DeX)^\ell
\right\|=o_p(1),\label{cx1}\\
&\frac{1}{nh_n}
\max_{1\le k\le n}\left\|
\sum_{i=1}^kf(X_{t_{i-1}^n},\theta_0)(\DeX)^{\ell_1}(\DeX)^{\ell_2}
-\frac kn\sum_{i=1}^nf(X_{t_{i-1}^n},\theta_0)(\DeX)^{\ell_1}(\DeX)^{\ell_2}
\right\|=o_p(1).\label{cx2}
\end{align}
\end{lem}
%%%%%%%%%%%%%%%%%%%%%%%%%%%%%%%%%%%%%%(4.10)%%%%%%%%%%%%%%%%%%%%%%%%%%%%%%%%%%%%%%%%%
\begin{proof}
If we prove
\begin{align}
&\frac1n\max_{1\le k\le n}\left\|
\sum_{i=1}^kf_{i-1}(\theta_0)\frac{(\DeX)^{\ell}}{h_n}
-\sum_{i=1}^k
\EEt\left[f_{i-1}(\theta_0)\frac{(\DeX)^{\ell}}{h_n}\middle|\GG\right]\right\|
=o_p(1),
\label{po1}\\
%%%
&\frac1n
\max_{1\le k\le n}\left\|
\sum_{i=1}^kf_{i-1}(\theta_0)\EEt\left[\frac{(\DeX)^\ell}{h_n}\middle|\GG\right]
-k\int_{\mathbb R^d}f(x,\theta_0)b^{\ell}(x,\beta_0)\dd\mu_{\theta_0}(x)
\right\|=o_p(1),
\label{po2}
\end{align}
%%%%%%%%%
then we have
\begin{align*}
&\frac{1}{nh_n}
\max_{1\le k\le n}\left\|
\sum_{i=1}^kf_{i-1}(\theta_0)(\DeX)^\ell
-\frac kn\sum_{i=1}^nf_{i-1}(\theta_0)(\DeX)^\ell
\right\|\\
%%%%
&\le\frac1n
\max_{1\le k\le n}\left\|
\sum_{i=1}^kf_{i-1}(\theta_0)\frac{(\DeX)^\ell}{h_n}
-k\int_{\mathbb R^d}f(x,\theta_0)b^\ell(x,\beta_0)\dd\mu_{\theta_0}(x)
\right\|\\
&\qquad+
\frac1n
\max_{1\le k\le n}\left\|
k\int_{\mathbb R^d}f(x,\theta_0)b^\ell(x,\beta_0)\dd\mu_{\theta_0}(x)
-\frac kn\sum_{i=1}^nf_{i-1}(\theta_0)\frac{(\DeX)^\ell}{h_n}
\right\|\\
%%%%
&\le\frac1n
\max_{1\le k\le n}\left\|
\sum_{i=1}^kf_{i-1}(\theta_0)\frac{(\DeX)^\ell}{h_n}
-k\int_{\mathbb R^d}f(x,\theta_0)b^\ell(x,\beta_0)\dd\mu_{\theta_0}(x)
\right\|\\
&\qquad+
\frac1n
\left\|
n\int_{\mathbb R^d}f(x,\theta_0)b^\ell(x,\beta_0)\dd\mu_{\theta_0}(x)
-\sum_{i=1}^nf_{i-1}(\theta_0)\frac{(\DeX)^\ell}{h_n}
\right\|\\
%%%%
&\le\frac2n
\max_{1\le k\le n}\left\|
\sum_{i=1}^kf_{i-1}(\theta_0)\frac{(\DeX)^\ell}{h_n}
-k\int_{\mathbb R^d}f(x,\theta_0)b^\ell(x,\beta_0)\dd\mu_{\theta_0}(x)
\right\|
=o_p(1).
\end{align*}
Hence, we obtain \eqref{cx1}.
Let us prove \eqref{po1} and \eqref{po2}.

%%%%%%%%%%%%%%%%%%%%%%%%%%%%%%
\textit{Proof of \eqref{po1}.}
Set
\begin{align*}
\mathcal T_i=f_{i-1}(\theta_0)\left(
\frac{(\DeX)^{\ell}}{h_n}
-\EEt\left[
\frac{(\DeX)^{\ell}}{h_n}
\middle|\GG\right]\right),\quad 
\mathcal M_k=\sum_{i=1}^k\mathcal T_i.
\end{align*}
Note that, from Lemma \ref{ke7}, 
$\EEt[\|\mathcal T_i\|^2|\GG]=R_{i-1}(h_n^{-1},\theta)$.
It follows from
Theorem 2.11 of Hall and Heyde (1980) that
\begin{align*}
\frac{1}{n^2}\EEt\left[
\max_{1\le k\le n}\|\mathcal M_k\|^2
\right]
\le\frac{C}{n^2}
\left(
\EEt\left[\sum_{i=1}^n\EEt[\|\mathcal T_i\|^2|\GG]\right]
+\EEt\left[\max_{1\le k\le n}\|\mathcal T_k\|^2\right]
\right)
\le\frac{C'}{nh_n}\lto0.
\end{align*}
Hence, we obtain
\begin{align*}
&\frac1n\max_{1\le k\le n}\left\|
\sum_{i=1}^kf_{i-1}(\theta_0)\frac{(\DeX)^{\ell}}{h_n}
-\sum_{i=1}^k
\EEt\left[f_{i-1}(\theta_0)\frac{(\DeX)^{\ell}}{h_n}\middle|\GG\right]\right\|\\
&=
\frac1n\max_{1\le k\le n}\|\mathcal M_k\|
\pto0.
\end{align*} 
%%%%%%%%%%%%%%%%%%%%

\textit{Proof of \eqref{po2}.} 
We have, from Lemmas $\ref{ke7}$-$\ref{lem1}$,
\begin{align*}
&\frac1n
\max_{1\le k\le n}\left\|
\sum_{i=1}^kf_{i-1}(\theta_0)\EEt\left[\frac{(\DeX)^\ell}{h_n}\middle|\GG\right]
-k\int_{\mathbb R^d}f(x,\theta_0)b^{\ell}(x,\beta_0)\dd\mu_{\theta_0}(x)
\right\|\\
%%%%
&=
\frac1n
\max_{1\le k\le n}\left\|
\sum_{i=1}^kf_{i-1}(\theta_0)(b_{i-1}^\ell(\beta_0)+\Ri)
-k\int_{\mathbb R^d}f(x,\theta_0)b^\ell(x,\beta_0)\dd\mu_{\theta_0}(x)
\right\|\\
%%%
&\le
\frac1n
\max_{1\le k\le n}\left\|
\sum_{i=1}^kf_{i-1}(\theta_0)b_{i-1}^\ell(\beta_0)
-k\int_{\mathbb R^d}f(x,\theta_0)b^\ell(x,\beta_0)\dd\mu_{\theta_0}(x)
\right\|\\
&\qquad+
\frac1n
\max_{1\le k\le n}\left\|
\sum_{i=1}^k\Ri
\right\|\\
%%%%
&\le
\frac1n
\max_{1\le k\le n}\left\|
\sum_{i=1}^kf_{i-1}(\theta_0)b_{i-1}^\ell(\beta_0)
-\frac{k}{n}\sum_{i=1}^nf_{i-1}(\theta_0)b_{i-1}^\ell(\beta_0)
\right\|\\
&\qquad+\frac1n
\max_{1\le k\le n}\left\|
\frac{k}{n}\sum_{i=1}^nf_{i-1}(\theta_0)b_{i-1}^\ell(\beta_0)
-k\int_{\mathbb R^d}f(x,\theta_0)b^\ell(x,\beta_0)\dd\mu_{\theta_0}(x)
\right\|
+o_p(1)\\
%%%%
&\le
\frac1n
\max_{1\le k\le n}\left\|
\sum_{i=1}^kf_{i-1}(\theta_0)b_{i-1}^\ell(\beta_0)
-\frac{k}{n}\sum_{i=1}^nf_{i-1}(\theta_0)b_{i-1}^\ell(\beta_0)
\right\|\\
&\qquad+
\left\|
\frac{1}{n}\sum_{i=1}^nf_{i-1}(\theta_0)b_{i-1}^\ell(\beta_0)
-\int_{\mathbb R^d}f(x,\theta_0)b^\ell(x,\beta_0)\dd\mu_{\theta_0}(x)
\right\|
+o_p(1)\\
%%%
&=o_p(1).
\end{align*}
This completes the proof of \eqref{cx1}.
In the same way, we have \eqref{cx2}.
\end{proof}

%\input{change_point3}

%%%%%%%%%%%%%%%%%%%%%%%%%%%%%%%%%%%%%%%%%%%%%%%%%%%%%%%%%%%%%%%%%%%%%%%%%%%%%%%%%%%%%%%%
\begin{proof}[\bf{Proof of Theorem \ref{th1}}]
Let
\begin{align*}
\eta_i=\tr\left(A^{-1}_{i-1}(\alpha_0)
\frac{(\Delta X_i)^{\otimes2}}{h_n}\right)
=\sum_{\ell_1,\ell_2=1}^d\left(A_{i-1}^{-1}(\alpha_0)\right)^{\ell_1,\ell_2}
\frac{(\DeX)^{\ell_1}(\DeX)^{\ell_2}}{h_n}.
\end{align*}
By Taylor expansion, we have
\begin{align*}
(A_{i-1}^{-1}(\hat\alpha_n))^{\ell_1,\ell_2}
=
(A_{i-1}^{-1}(\alpha_0))^{\ell_1,\ell_2}
+\partial_\alpha(A_{i-1}^{-1}(\alpha_0))^{\ell_1,\ell_2}(\hat\alpha_n-\alpha_0)
+(\hat\alpha_n-\alpha_0)^\TT\Aa(\hat\alpha_n-\alpha_0),
\end{align*}
where 
\begin{align*}
\Aa=\int_0^1(1-u)\partial_\alpha^2
\left(A_{i-1}^{-1}(\alpha_0+u(\hat\alpha_n-\alpha_0)\right)^{\ell_1,\ell_2}\dd u.
\end{align*}
%%%
Then, we can express 
\begin{align*}
\hat\eta_i&=\sum_{\ell_1,\ell_2=1}^d
\left(A^{-1}_{i-1}(\hat\alpha_n)\right)^{\ell_1,\ell_2}
\frac{(\DeX)^{\ell_1}(\DeX)^{\ell_2}}{h_n}\\
&=
\sum_{\ell_1,\ell_2=1}^d
\left(
\left(A^{-1}_{i-1}(\alpha_0)\right)^{\ell_1,\ell_2}
+\partial_\alpha\left(A^{-1}_{i-1}(\alpha_0)\right)^{\ell_1,\ell_2}
(\hat\alpha_n-\alpha_0)
+(\hat\alpha_n-\alpha_0)^\TT \Aa(\hat\alpha_n-\alpha_0)\right)
\frac{(\DeX)^{\ell_1}(\DeX)^{\ell_2}}{h_n}\\
%%%
&=\eta_i
+\left(\frac1{\sqrt n}\sum_{\ell_1,\ell_2=1}^d
\partial_\alpha\left(A^{-1}_{i-1}(\alpha_0)\right)^{\ell_1,\ell_2}
\frac{(\DeX)^{\ell_1}(\DeX)^{\ell_2}}{h_n}\right)
\sqrt{n}(\hat\alpha_n-\alpha_0)\\
&\qquad+\left(\sqrt n(\hat\alpha_n-\alpha_0)\right)^\TT
\left(
\frac1n\sum_{\ell_1,\ell_2=1}^d
\Aa\frac{(\DeX)^{\ell_1}(\DeX)^{\ell_2}}{h_n}
\right)
\sqrt n(\hat\alpha_n-\alpha_0)\\
%%%
&=:\eta_i+\frac{1}{\sqrt n}\Hi\left(\sqrt{n}(\hat\alpha_n-\alpha_0)\right)
+
\frac1n\left(\sqrt n(\hat\alpha_n-\alpha_0)\right)^\TT
\Hd\left(\sqrt n(\hat\alpha_n-\alpha_0)\right).
\end{align*}
%%%%%%%%%%%%%%%%%%%%%%%%
Therefore, it is enough to show
\begin{align}
\frac{1}{\sqrt{2dn}}\max_{1\le k\le n}
\left|\sum_{i=1}^k\eta_i-\frac{k}{n}\sum_{i=1}^n\eta_i\right|&\dto
\sup_{0\le s\le 1}| { \boldsymbol B_1^0(s)} | ,
\label{ap1}\\
%%%
\frac1{n}\max_{1\le k\le n}
\left\|
\sum_{i=1}^k\Hi
-\frac kn\sum_{i=1}^n\Hi
\right\|&=o_p(1),
\label{ap2}\\
%%%
\frac1{n^{3/2}}\max_{1\le k\le n}
\left\|\sum_{i=1}^k\Hd
-\frac kn\sum_{i=1}^n\Hd
\right\|&=o_p(1).
\label{ap3}
\end{align}
%%%%%%%%%%%%%%%%%%%%%%%%%%%%%%%%%%%%%(4.9)%%%%%%%%%%%%%%%%%%%%%%%%%%%%%%%%%%%%%%%%

\textit{Proof of \eqref{ap1}.} 
Let $\boldsymbol B_d(t)$ be a $d$-dimensional standard Brownian motion.
$X_n(\cdot)\wto X(\cdot)$ in $\mathbb D[0,1]$ denotes that 
$X_n(\cdot)$ weakly converges to $X(\cdot)$ in the Skorohod space on $[0,1]$.
If we prove
\begin{align}\label{bm1}
\mathcal U_n(s):=\frac{1}{\sqrt{2dn}}\sum_{i=1}^{[ns]}(\eta_i-d)
\wto { \boldsymbol B_1(s)} \quad\text{in }\mathbb D[0,1],
\end{align}
%we obtain, applying 
then, it follows from the continuous mapping theorem that
\begin{align*}
\frac{1}{\sqrt{2dn}}\max_{1\le k\le n}
\left|\sum_{i=1}^k\eta_i-\frac{k}{n}\sum_{i=1}^n\eta_i\right|
&=
\frac{1}{\sqrt{2dn}}\max_{1\le k\le n}
\left|\sum_{i=1}^k(\eta_i-d)-\frac{k}{n}\sum_{i=1}^n(\eta_i-d)\right|\\
&=
\max_{1\le k\le n}\left|\frac{1}{\sqrt{2dn}}\sum_{i=1}^k(\eta_i-d)
-\frac{k}{n}\frac{1}{\sqrt{2dn}}\sum_{i=1}^n(\eta_i-d)\right|\\
&=
\sup_{0\le s\le 1}\left|\frac{1}{\sqrt{2dn}}\sum_{i=1}^{[ns]}(\eta_i-d)
-\frac{[ns]}{n}\frac{1}{\sqrt{2dn}}\sum_{i=1}^n(\eta_i-d)\right|\\
&=
\sup_{0\le s\le 1}\left|\mathcal U_n(s)
-\frac{[ns]}{n}\mathcal U_n(1)\right|\\
&\dto
\sup_{0\le s\le 1} | { \boldsymbol{B}_1(s) - s \boldsymbol{B}_1(1) } |   %|B(s)-sB(1)|
=
\sup_{0\le s\le 1}| { \boldsymbol B_1^0(s)} |.
\end{align*}
%%%%%%%%%%%%%%%%%%%%%%%%%%%%%%%%%
%where 
%$B$ is a $1$-dimensional standard Brownian motion and
%$X_n(\cdot)\wto X(\cdot)$ in $\mathbb D[0,1]$ denotes that 
%$X_n(\cdot)$ weakly converges to $X(\cdot)$ in Skorohod space on $[0,1]$.
%note the following:
%Let $\mathbb D[0,1]$ be the Skorohod space  { i.e.,}  
%the space of c\`adl\`ag functions on $[0,1]$, endowed with the Skorohod topology and 
%$X_n(\cdot),X(\cdot)$ be $\mathbb D[0,1]$-valued functions 
%defined on probability spaces $(\Omega_n,\FF_n,P_n), (\Omega,\FF,P)$ respectively.
%We define $X_n(\cdot)\wto X(\cdot)\ \text{in }\mathbb D[0,1]$ 
%if, for any continuous bounded function $f$ on $\mathbb D[0,1]$, 
%we have 
%\begin{align*}
%\int_{\Omega_n}f(X_n(\cdot,\omega))\dd P_n(\omega)\lto
%\int_{\Omega}f(X(\cdot,\omega))\dd P(\omega).
%\end{align*}

%%%%%%%%%%%%%%%%%%%%%%%%%%%%%%%%%%%

Let us prove \eqref{bm1}.
It is enough to show
\begin{align}
\frac{1}{\sqrt{2dn}}\sum_{i=1}^{[ns]}
\left(
\eta_i-d-\EEt\left[
\eta_i-d\middle|\GG\right]
\right)
&\wto  { \boldsymbol B_1(s)}  \quad{\text{in }}\mathbb D[0,1],\label{nm1}\\
\frac{1}{\sqrt{n}}\sum_{i=1}^{n}
\EEt\left[\eta_i-d
\middle|\GG\right]&=o_p(1).\label{nm2}
\end{align}
%%%%
From Lemma \ref{ke7},
\begin{align*}
\frac{1}{\sqrt{n}}\sum_{i=1}^{n}
\EEt\left[\eta_i-d
\middle|\GG\right]
=\frac{1}{\sqrt{n}}\sum_{i=1}^{n}\Ri
=\sqrt{nh_n^2}\cdot\frac{1}{n}\sum_{i=1}^{n}\Ro
=o_p(1).
\end{align*}
Hence we have \eqref{nm2}.
%%%%%%%%%%%%%%%

According to Corollary 3.8 of McLeish (1974),  
we obtain \eqref{nm1} if we prove
\begin{align}
\frac{1}{\sqrt n}\sum_{i=1}^{[ns]}
\left|\EEt\left[
(\eta_i-d)-\EEt\left[
\eta_i-d
\middle|\GG\right]
\middle|\GG\right]\right|&=o_p(1),
\label{gh1}\\
%%%%
\frac{1}{2dn}\sum_{i=1}^{[ns]}
\EEt\left[
\left(
(\eta_i-d)-\EEt\left[
\eta_i-d\middle|\GG\right]
\right)^2\middle|\GG\right]&\pto s,
\label{gh2}\\
%%%%
\frac{1}{n^2}\sum_{i=1}^{[ns]}
\EEt\left[
\left(
(\eta_i-d)-\EEt\left[
\eta_i-d\middle|\GG\right]
\right)^4\middle|\GG\right]&=o_p(1)
\label{gh3}
\end{align}
for all $s\in[0,1]$.

%%%%
\eqref{gh1} is obvious.
%%%%
From Lemma \ref{ke7}, we have
\begin{align*}
\EEt\left[
\left(
(\eta_i-d)-\EE_{\theta_0}\left[
\eta_i-d\middle|\GG\right]
\right)^2\middle|\GG\right]
%%%
&=\EEt\left[
\left(
\eta_i-d+\Ri
\right)^2\middle|\GG\right]\\
%%%%
&=
\EEt\left[
\eta_i^2-2d\eta_i+d^2+\Ri
\middle|\GG\right]\\
%%%%
&=
\EEt\left[\eta_i^2\middle|\GG\right]
-2d\EEt\left[\eta_i\middle|\GG\right]
+d^2+\Ri\\
%%%%
%&=(d^2+2d+\Ri)-2d(d+\Ri)+d^2+\Ri\\
%%%%
&=2d+\Ri
\end{align*}
%%%%
and
\begin{align*}
\frac{1}{2dn}\sum_{i=1}^{[ns]}\EEt\left[\left(
(\eta_i-d)
-\EEt\left[
\eta_i-d\middle|\GG\right]
\right)^2\middle|\GG\right]
=\frac{[ns]}{n}\frac{1}{2d[ns]}\sum_{i=1}^{[ns]}(2d+\Ri)
\pto s.
\end{align*}
Hence \eqref{gh2} holds.
%%%
Since $\EEt[\eta_i^4|\GG]=\Ro$,
\begin{align*}
\frac{1}{n^2}\sum_{i=1}^{[ns]}
\EEt\left[
\left(
\eta_i-d
-\EEt\left[
\eta_i-d\middle|\GG\right]
\right)^4\middle|\GG\right]
&\le\frac{C}{n^2}\sum_{i=1}^{[ns]}
\EEt\left[
\eta_i^4+d^4+\Rf
\middle|\GG\right]\\
&=\frac1n\cdot\frac1n\sum_{i=1}^{[ns]}\Ro=o_p(1).
\end{align*}
Therefore we obtain \eqref{gh3}. 
This completes the proof of \eqref{ap1}.
%%%%%%%%%%%%%%%%%%%%%%%%%%%%%%%%%%%%%%%%%%%%%%%%%%%%%%%%%%%%%%%%%%%%%%%%%%%%%%%%%%%%%%%

%%%%%%%%%%%%%%%%%%%%%%%%%%%%%%%%%%%%%(4.10)%%%%%%%%%%%%%%%%%%%%%%%%%%%%%%%%%%%%%%%%

\textit{Proof of \eqref{ap2}.}
Noting that 
\begin{align*}
\Hi=
\sum_{\ell_1,\ell_2=1}^d
\partial_\alpha\left(A^{-1}_{i-1}(\alpha_0)\right)^{\ell_1,\ell_2}
\frac{(\DeX)^{\ell_1}(\DeX)^{\ell_2}}{h_n},
\end{align*}
we have, by using Lemma \ref{lem2},
\begin{align*}
&\frac1{n}\max_{1\le k\le n}
\left\|
\sum_{i=1}^k\Hi
-\frac kn\sum_{i=1}^n\Hi
\right\|\\
&\le
\sum_{\ell_1,\ell_2=1}^d
\frac1{nh_n}\max_{1\le k\le n}
\left\|
\sum_{i=1}^k
\partial_\alpha\left(A^{-1}_{i-1}(\alpha_0)\right)^{\ell_1,\ell_2}
(\DeX)^{\ell_1}(\DeX)^{\ell_2}
-\frac kn\sum_{i=1}^n
\partial_\alpha\left(A^{-1}_{i-1}(\alpha_0)\right)^{\ell_1,\ell_2}
(\DeX)^{\ell_1}(\DeX)^{\ell_2}
\right\|\\
&=o_p(1).
\end{align*}
%%%%%%%%%%%%%%%%%%%%%%%%%%%%%%%%%%%%%%%%%%%%%%%%%%%%%%%%%%%%%%%%%%%%%%%%%%%%%%%%%%%

%%%%%%%%%%%%%%%%%%%%%%%%%%%%%%%%%%%%(4.11)%%%%%%%%%%%%%%%%%%%%%%%%%%%%%%%%%%%%%%%%%%%
\textit{Proof of \eqref{ap3}.} 
Since $\alpha_0\in\mathrm{Int\,}\Theta_A$, 
there exists an open neighborhood $\mathcal O_{\alpha_0}$ of $\alpha_0$ such that 
$\mathcal O_{\alpha_0}\subset\Theta_A$. 
If $\hat\alpha_n\in\mathcal O_{\alpha_0}$, then
\begin{align*}
\|\Hd\|
=\left\|
\sum_{\ell_1,\ell_2=1}^d\Aa\frac{(\DeX)^{\ell_1}(\DeX)^{\ell_2}}{h_n}
\right\|
\le
\sum_{\ell_1,\ell_2=1}^d
\sup_{\alpha\in\Theta_A}\|\partial_\alpha^2(A_{i-1}^{-1}(\alpha))^{\ell_1,\ell_2}\|
\left|\frac{(\DeX)^{\ell_1}(\DeX)^{\ell_2}}{h_n}\right|
\end{align*}
and
%%%%%%%%
\begin{align*}
&\EEt\left[
\frac{1}{n^{3/2}}\sum_{i=1}^n\sum_{\ell_1,\ell_2=1}^d
\sup_{\alpha\in\Theta_A}\left\|
\partial_\alpha^2(A^{-1}_{i-1}(\alpha))^{\ell_1,\ell_2}
\right\|
\left|\frac{(\DeX)^{\ell_1}(\DeX)^{\ell_2}}{h_n}\right|
\right]\\
%%%
&\le
\frac{1}{n^{3/2}}\sum_{i=1}^n\sum_{\ell_1,\ell_2=1}^d\EEt\left[
\sup_{\alpha\in\Theta_A}\left\|
\partial_\alpha^2(A^{-1}_{i-1}(\alpha))^{\ell_1,\ell_2}
\right\|^2\right]^{1/2}
\EEt\left[
\left|\frac{(\DeX)^{\ell_1}(\DeX)^{\ell_2}}{h_n}\right|^2
\right]^{1/2}\\
&\le 
\frac{C}{n^{1/2}}\lto0.
\end{align*}
%%%%%%%%%%%%%%%%
Hence, from {\textbf{[A6]}}, we have 
$\Pt(\hat\alpha_n\in\mathcal O_{\alpha_0})\lto1$ and, for all $\epsilon>0$, 
\begin{align*}
&\Pt\left(
\frac1{n^{3/2}}\sum_{i=1}^n
\|\Hd\|
\ge\epsilon
\right)\\
%%%
&\le
\Pt\left(
\left\{\frac1{n^{3/2}}\sum_{i=1}^n
\|\Hd\|
\ge\epsilon
\right\}
\cap \left\{\hat\alpha_n\in\mathcal O_{\alpha_0}\right\}
\right)
+
\Pt\left(
\left\{
\frac1{n^{3/2}}\sum_{i=1}^n\|\Hd\|
\ge\epsilon
\right\}
\cap \left\{\hat\alpha_n\notin\mathcal O_{\alpha_0}\right\}
\right)\\
%%%%
&\le
\Pt\left(
\left\{
\frac{1}{n^{3/2}}\sum_{i=1}^n\sum_{\ell_1,\ell_2=1}^d
\sup_{\alpha\in\Theta_A}\left\|
\partial_\alpha^2(A^{-1}_{i-1}(\alpha))^{\ell_1,\ell_2}
\right\|
\left|\frac{(\DeX)^{\ell_1}(\DeX)^{\ell_2}}{h_n}\right|
\ge\epsilon
\right\}
\cap \left\{\hat\alpha_n\in\mathcal O_{\alpha_0}\right\}
\right)\\
&\qquad+
\Pt\left(
\hat\alpha_n\notin\mathcal O_{\alpha_0}
\right)\\
%%%%
&\le
\Pt\left(
\frac{1}{n^{3/2}}\sum_{i=1}^n\sum_{\ell_1,\ell_2=1}^d
\sup_{\alpha\in\Theta_A}\left\|
\partial_\alpha^2(A^{-1}_{i-1}(\alpha))^{\ell_1,\ell_2}
\right\|
\left|\frac{(\DeX)^{\ell_1}(\DeX)^{\ell_2}}{h_n}\right|
\ge\epsilon
\right)
+
\Pt\left(
\hat\alpha_n\notin\mathcal O_{\alpha_0}
\right)\\
%%%%
&\le\frac{1}{\epsilon}
\frac{C}{n^{1/2}}+\Pt\left(\hat\alpha_n\notin\mathcal O_{\alpha_0}\right)
\lto0.
\end{align*}
Therefore we obtain 
\begin{align*}
\frac1{n^{3/2}}\sum_{i=1}^n\|\Hd\|=o_p(1)
\end{align*}
and
\begin{align*}
\frac1{n^{3/2}}\max_{1\le k\le n}
\left\|\sum_{i=1}^k\Hd
-\frac kn\sum_{i=1}^n\Hd
\right\|
&\le
\frac1{n^{3/2}}\max_{1\le k\le n}
\left(\sum_{i=1}^k\|\Hd\|+\frac kn\sum_{i=1}^n\|\Hd\|\right)\\
&\le\frac{2}{n^{3/2}}\sum_{i=1}^n\|\Hd\|=o_p(1).
\end{align*}
%%%%%%%%%%%%%%%%%%%%%%%%%%%%%%%%%%%%%%%%%%%%%%%%%%%%%%%%%%%%%%%%%%%%%%%%%%%%%%%%%%%%%%%%
\end{proof}

%\input{change_point4}
%%%%%%%%%%%%%%%%%%%%%%%%%%%%%%%%%%%%%%%%%%%%%%%%%%%%%%%%%%%%%%%%%%%%%%%%%%%%%%%
\begin{proof}[\bf{Proof of Theorem \ref{th2}}]
Let
\begin{align*}
&\kappa(x,\alpha)=
1_d^\TT a^{-1}(x,\alpha)
,\quad
\xi_i=\kappa_{i-1}(\alpha_0)(\DeX-h_nb_{i-1}(\beta_0)).
\end{align*} 
%%%%%%%%%%%
By the Taylor expansion, we have
\begin{align*}
\kappa_{i-1}^\ell(\hat\alpha_n)
=
\kappa_{i-1}^\ell(\alpha_0)
+\partial_\alpha\kappa_{i-1}^\ell(\alpha_0)(\hat\alpha_n-\alpha_0)
+(\hat\alpha_n-\alpha_0)^\TT\KK(\hat\alpha_n-\alpha_0),
\end{align*}
where 
\begin{align*}
\KK=\int_0^1(1-u)\partial_\alpha^2
\kappa_{i-1}^\ell(\alpha_0+u(\hat\alpha_n-\alpha_0))\dd u.
\end{align*}

%%%%%%%%%%%%
Then, we can express
\begin{align*}
\hat\xi_i&=
\kappa_{i-1}(\hat\alpha_n)\left(\DeX-h_nb_{i-1}(\hat\beta_n)\right)\\
&=\sum_{\ell=1}^d
\kappa_{i-1}^{\ell}(\hat\alpha_n)\left(\DeX-h_nb_{i-1}(\hat\beta_n)\right)^\ell\\
&=\sum_{\ell=1}^d
\left(
\kappa_{i-1}^{\ell}(\alpha_0)
+\partial_{\alpha}\kappa_{i-1}^\ell(\alpha_0)(\hat\alpha_n-\alpha_0)
+(\hat\alpha_n-\alpha_0)^\TT 
\KK(\hat\alpha_n-\alpha_0)
\right)
\left(\DeX-h_nb_{i-1}(\hat\beta_n)\right)^\ell\\
&=\sum_{\ell=1}^d\kappa_{i-1}^\ell(\alpha_0)
\left(\DeX-h_nb_{i-1}(\beta_0)-h_n(b_{i-1}(\hat\beta_n)-b_{i-1}(\beta_0))\right)^\ell\\
&\qquad+\left(
\frac{1}{\sqrt n}\sum_{\ell=1}^d
\partial_{\alpha}\kappa_{i-1}^\ell(\alpha_0)
\left(\DeX-h_nb_{i-1}(\hat\beta_n)\right)^\ell\right)
\sqrt n (\hat\alpha_n-\alpha_0)\\
&\qquad+\left(\sqrt n (\hat\alpha_n-\alpha_0)\right)^\TT
\left(
\frac1n\sum_{\ell=1}^d
\KK
\left(\DeX-h_nb_{i-1}(\hat\beta_n)\right)^\ell
\right)\sqrt n (\hat\alpha_n-\alpha_0)\\
&=\xi_i
-\left(\sqrt{\frac{h_n}{n}}\sum_{\ell=1}^d\kappa_{i-1}^\ell(\alpha_0)
\partial_\beta b_{i-1}^\ell(\beta_0)\right)
\sqrt{nh_n}(\hat\beta_n-\beta_0)\\
&\qquad-\left(\sqrt{nh_n}(\hat\beta_n-\beta_0)\right)^\TT
\left(\frac{1}{n}\sum_{\ell=1}^d\kappa_{i-1}^\ell(\alpha_0)
\int_0^1(1-u)\partial_\beta^2b_{i-1}^\ell(\beta_0+u(\hat\beta_n-\beta_0))\dd u\right)
\sqrt{nh_n}(\hat\beta_n-\beta_0)
\\
&\qquad+\left(
\frac{1}{\sqrt n}\sum_{\ell=1}^d
\partial_{\alpha}\kappa_{i-1}^\ell(\alpha_0)(\DeX)^\ell
\right)\sqrt n (\hat\alpha_n-\alpha_0)
-\left(
\frac{h_n}{\sqrt{n}}\sum_{\ell=1}^d
\partial_{\alpha}\kappa_{i-1}^\ell(\alpha_0)b_{i-1}^\ell(\hat\beta_n)
\right)\sqrt n (\hat\alpha_n-\alpha_0)\\
&\qquad+\left(\sqrt n (\hat\alpha_n-\alpha_0)\right)^\TT
\left(
\frac1n\sum_{\ell=1}^d
\KK\left(\DeX-h_nb_{i-1}(\hat\beta_n)\right)^\ell 
\right)\sqrt n (\hat\alpha_n-\alpha_0)\\
%%%%
&=:
\xi_i+\sqrt{\frac{h_n}{n}}\Qi\left(\sqrt{nh_n}(\hat\beta_n-\beta_0)\right)
+
\frac1n\left(\sqrt{nh_n}(\hat\beta_n-\beta_0)\right)^\TT\Qd
\left(\sqrt{nh_n}(\hat\beta_n-\beta_0)\right)\\
&\qquad+
\frac{1}{\sqrt n}\Qt
\left(\sqrt n (\hat\alpha_n-\alpha_0)\right)
+\frac{h_n}{\sqrt n}
\Qf
\left(\sqrt n (\hat\alpha_n-\alpha_0)\right)\\
&\qquad
+\frac1n\left(\sqrt n (\hat\alpha_n-\alpha_0)\right)^\TT
\Qg
\left(\sqrt n (\hat\alpha_n-\alpha_0)\right).
\end{align*}
Therefore, it is enough to show
%%%%%%%%
\begin{align}
\frac{1}{\sqrt{dnh_n}}\max_{1\le k\le n}\left|\sum_{i=1}^k\xi_i
-\frac{k}{n}\sum_{i=1}^n\xi_i\right|&\dto \sup_{0\le s\le 1}| { \boldsymbol B_1^0(s)} |,
\label{dp0}\\%%%%
\frac{1}{n}\max_{1\le k\le n}
\left\|
\sum_{i=1}^k\Qi
-\frac{k}{n}\sum_{i=1}^n\Qi
\right\|&=o_p(1),
\label{dp1}\\%%%%
\frac{1}{n\sqrt{nh_n}}\max_{1\le k\le n}
\left\|
\sum_{i=1}^k\Qd
-\frac{k}{n}\sum_{i=1}^n\Qd
\right\|&=o_p(1),
\label{dp1-2}\\%%%%
\frac1{n\sqrt{h_n}}\max_{1\le k\le n}\left\|
\sum_{i=1}^n\Qt-\frac kn\sum_{i=1}^n\Qt
\right\|&=o_p(1),
\label{dp2}\\%%%%
\frac{\sqrt{h_n}}{n}\max_{1\le k\le n}\left\|
\sum_{i=1}^k\Qf-\frac kn\sum_{i=1}^n\Qf
\right\|&=o_p(1),
\label{dp2-1}\\%%%%%%%
\frac1{n\sqrt{nh_n}}\max_{1\le k\le n}\left\|
\sum_{i=1}^k\Qg-\frac kn\sum_{i=1}^n\Qg
\right\|&=o_p(1).
\label{dp3}%%%%%%
\end{align}

%%%%%%%%%%%%%%%%%%%%%%%%%%%%%%%%(4.22)%%%%%%%%%%%%%%%%%%%%%%%%%%%%%%%%%%%%%%%%

\textit{Proof of \eqref{dp0}.} 
If we prove
\begin{align}\label{bm2}
\mathcal V_n(s):=\frac{1}{\sqrt{dnh_n}}\sum_{i=1}^{[ns]}\xi_i
\wto  { \boldsymbol B_1(s)} \quad\text{in }\mathbb D[0,1],
\end{align}
{ we see from the continuous mapping theorem that}
\begin{align*}
\frac{1}{\sqrt{dnh_n}}\max_{1\le k\le n}
\left|\sum_{i=1}^k\xi_i-\frac{k}{n}\sum_{i=1}^n\xi_i\right|
&=
\max_{1\le k\le n}\left|\frac{1}{\sqrt{dnh_n}}\sum_{i=1}^k\xi_i
-\frac{k}{n}\frac{1}{\sqrt{dnh_n}}\sum_{i=1}^n\xi_i\right|\\
&=
\sup_{0\le s\le 1}\left|\frac{1}{\sqrt{dnh_n}}\sum_{i=1}^{[ns]}\xi_i
-\frac{[ns]}{n}\frac{1}{\sqrt{dnh_n}}\sum_{i=1}^n\xi_i\right|\\
&=
\sup_{0\le s\le 1}\left|\mathcal V_n(s)
-\frac{[ns]}{n}\mathcal V_n(1)\right|\\
&\dto
\sup_{0\le s\le 1}| { \boldsymbol B_1(s)} -s { \boldsymbol B_1(1)} |
=
\sup_{0\le s\le 1}| { \boldsymbol B_1^0(s)} |.
\end{align*}
%%%%%%%%%%%%%%%%%%%%%%%%%%%%%%%%%%%%%%%%%%%%%%%%%%%%%%%%%%%%

Let us prove \eqref{bm2}.
It is enough to prove  %the following.
\begin{align}
\frac1{\sqrt{dnh_n}}\sum_{i=1}^{[ns]}
\left(\xi_i
-\EEt\left[\xi_i\middle|\GG\right]\right)
&\wto  { \boldsymbol B_1(s)} \quad{\text{in }}\mathbb D[0,1]{ ,} 
\label{as1}\\
%%%
\frac{1}{\sqrt{nh_n}}\sum_{i=1}^{n}\EEt
\left[\xi_i\middle|\GG\right]&=o_p(1).
\label{as2}
\end{align}
%%%%
From Lemma \ref{ke7},
\begin{align*}
\frac{1}{\sqrt{nh_n}}\sum_{i=1}^{n}\EEt
\left[\xi_i\middle|\GG\right]
=
\frac{1}{\sqrt{nh_n}}\sum_{i=1}^{n}\Rd
=
\sqrt{nh_n^3}\cdot\frac1n\sum_{i=1}^n\Ro
=o_p(1).
\end{align*}
Hence we have \eqref{as2}.
%%%%%%%%%%%%%%%%%%%%%%%%%%%

{ It follows from Corollary 3.8 of McLeish (1974) that  one has \eqref{as1}} if we prove
\begin{align}
\frac1{\sqrt{nh_n}}\sum_{i=1}^{[ns]}\left|
\EEt\left[
\xi_i-\EEt\left[\xi_i\middle|\GG\right]
\middle|\GG\right]\right|&=o_p(1),
\label{zx0}\\
%%%%
\frac1{dnh_n}\sum_{i=1}^{[ns]}
\EEt\left[
\left(\xi_i-\EEt\left[\xi_i\middle|\GG\right]\right)^2
\middle|\GG\right]&\pto s,
\label{zx1}\\
%%%%
\frac1{(nh_n)^2}\sum_{i=1}^{[ns]}
\EEt\left[
\left(\xi_i-\EEt\left[\xi_i\middle|\GG\right]\right)^4
\middle|\GG\right]&=o_p(1),
\label{zx2}
\end{align}
for all $s\in[0,1]$.

%%%%%%%%
\eqref{zx0} is { trivial}.
%%%%%%%%
From Lemma \ref{ke7},
\begin{align*}
\frac1{dnh_n}\sum_{i=1}^{[ns]}
\EEt\left[
\left(\xi_i-\EEt\left[\xi_i\middle|\GG\right]\right)^2
\middle|\GG\right]
%%%
&=\frac1{nh_n}\sum_{i=1}^{[ns]}
\left(\EEt\left[\xi_i^2\middle|\GG\right]+\Rf\right)\\
&=\frac{[ns]}{n}\frac1{d[ns]h_n}\sum_{i=1}^{[ns]}
\left(dh_n+\Rd\right)
\pto s,
\end{align*}
 { i.e.,}  \eqref{zx1} holds.
%%%%%%%%%%%%%
Since $\EEt[\xi_i^4|\GG]=\Rd$,
\begin{align*}
\frac1{(nh_n)^2}\sum_{i=1}^{[ns]}
\EEt\left[
\left(\xi_i-\EEt\left[\xi_i\middle|\GG\right]\right)^4
\middle|\GG\right]
%%%
&\le\frac{C}{(nh_n)^2}\sum_{i=1}^{[ns]}
\EEt\left[
\xi_i^4+R_{i-1}(h_n^8,\theta)
\middle|\GG\right]\\
%%%
&=\frac1{(nh_n)^2}\sum_{i=1}^{[ns]}\Rd\\
%%%
&=\frac 1{n}\cdot\frac1n\sum_{i=1}^{[ns]}
\Ro
=o_p(1).
\end{align*}
Therefore we obtain \eqref{zx2}. This completes the proof of \eqref{dp0}.
%%%%%%%%%%%%%%%%%%%%%%%%%%%%%%%%%%%%%%%%%%%%%%%%%%%%%%%%%%%%%%%%%%%%%%%%%%%%%%%

%%%%%%%%%%%%%%%%%%%%%%%%%%%%%%%%%(4.23)%%%%%%%%%%%%%%%%%%%%%%%%%%%%%%%%%%%%%%
\textit{{ Proofs} of \eqref{dp1} and \eqref{dp2}.}
Noting that
\begin{align*}
\Qi&=
-\sum_{\ell=1}^d\kappa_{i-1}^\ell(\alpha_0)
\partial_\beta b_{i-1}^\ell(\beta_0),
\\
\Qt&=
\sum_{\ell=1}^d
\partial_{\alpha}\kappa_{i-1}^\ell(\alpha_0)(\DeX)^\ell,
\end{align*}
we have, by using Lemma \ref{lem1} and \ref{lem2}, 
\begin{align*}
&\frac{1}{n}\max_{1\le k\le n}
\left\|
\sum_{i=1}^k\Qi
-\frac{k}{n}\sum_{i=1}^n\Qi
\right\|\\
&\le
\sum_{\ell=1}^d
\frac{1}{n}\max_{1\le k\le n}
\left\|
\sum_{i=1}^k
\kappa_{i-1}^\ell(\alpha_0)
\partial_\beta b_{i-1}^\ell(\beta_0)
-\frac{k}{n}\sum_{i=1}^n
\kappa_{i-1}^\ell(\alpha_0)
\partial_\beta b_{i-1}^\ell(\beta_0)
\right\|\\
&=o_p(1)
\end{align*}
and 
\begin{align*}
&\frac1{n\sqrt{h_n}}\max_{1\le k\le n}\left\|
\sum_{i=1}^n\Qt-\frac kn\sum_{i=1}^n\Qt
\right\|\\
&\le
\sum_{\ell=1}^d
\frac1{n\sqrt{h_n}}\max_{1\le k\le n}\left\|
\sum_{i=1}^n
\partial_{\alpha}\kappa_{i-1}^\ell(\alpha_0)(\DeX)^\ell
-\frac kn\sum_{i=1}^n
\partial_{\alpha}\kappa_{i-1}^\ell(\alpha_0)(\DeX)^\ell
\right\|\\
&=o_p(1).
\end{align*} 
%%%%%%%%%%%%%%%%%%%%%%%%%%%%%%%%%%%%%%%%%%%%%%%%%%%%%%%%%%%%%%%%%%%%%%%%%%%%%

%%%%%%%%%%%%%%%%%%%%%%%%%%%%%(4.24),(4.26),(4.27)%%%%%%%%%%%%%%%%%%%%%%%%%%%%%

\textit{{ Proofs} of \eqref{dp1-2}, \eqref{dp2-1} and \eqref{dp3}.} 
Since $\theta_0\in \mathrm{Int}\,\Theta$, there exists an open neighborhood 
$\mathcal O_{\theta_0}$ of $\theta_0$ such that $\mathcal O_{\theta_0}\subset\Theta$.
Note that
\begin{align*}
\Qd
&=\sum_{\ell=1}^d\kappa_{i-1}^\ell(\alpha_0)
\int_0^1(1-u)\partial_\beta^2b_{i-1}^\ell(\beta_0+u(\hat\beta_n-\beta_0))\dd u,
\\
\Qf
&=
-\sum_{\ell=1}^d
\partial_\alpha\kappa_{i-1}^\ell(\alpha_0)b_{i-1}(\hat\beta_n),
\\
\Qg
&=
\sum_{\ell=1}^d
\KK\left(\DeX -h_nb_{i-1}(\hat\beta_n)\right)^\ell\\
&=
\sum_{\ell=1}^d
\int_0^1(1-u)\partial_\alpha^2
\kappa_{i-1}^\ell(\alpha_0+u(\hat\alpha_n-\alpha_0))\dd u
\left(\DeX -h_nb_{i-1}(\hat\beta_n)\right)^\ell.
\end{align*}
If $\hat\theta_n\in\mathcal O_{\theta_0}$, then
%%%%%Q2%%%%%%
\begin{align*}
\|\Qd\|
&\le
\sum_{\ell=1}^d
\sup_{\alpha\in\Theta_A}|\kappa_{i-1}^\ell(\alpha)|\sup_{\beta\in\Theta_B}
\left\|\partial_\beta^2b_{i-1}^\ell(\beta)\right\|,\\
%%%%%%Q4%%%%%
\|\Qf\|
&\le
\sum_{\ell=1}^d
\sup_{\alpha\in\Theta_A}\|\partial_\alpha\kappa_{i-1}^\ell(\alpha)\|
\sup_{\beta\in\Theta_B}|b_{i-1}^\ell(\beta)|,\\
%%%%%%Q5%%%%%
\|\Qg\|
&\le
\sum_{\ell=1}^d
\sup_{\alpha\in\Theta_A}\left\|
\partial_\alpha^2\kappa_{i-1}^\ell(\alpha)
\right\|
\left(|(\DeX)^\ell|
+h_n
\sup_{\beta\in\Theta_B}
|b_{i-1}^\ell(\beta)|\right),
\end{align*}
%%%%%%%%%%%%%%%%
\begin{align*}
%&
\EEt\left[
\frac{1}{n\sqrt{nh_n}}\sum_{i=1}^n\sum_{\ell=1}^d
\sup_{\alpha\in\Theta_A}|\kappa_{i-1}^\ell(\alpha)|\sup_{\beta\in\Theta_B}
\left\|\partial_\beta^2b_{i-1}^\ell(\beta)\right\|
\right]
%\\
%&\le
%\frac{1}{n\sqrt{nh_n}}\sum_{i=1}^n\sum_{\ell=1}^d
%\EEt\left[
%\sup_{\alpha\in\Theta_A}|\kappa_{i-1}^\ell(\alpha)|^2
%\right]^{1/2}
%\EEt\left[
%\sup_{\beta\in\Theta_B}
%\left\|\partial_\beta^2b_{i-1}^\ell(\beta)\right\|^2
%\right]^{1/2}\\
%&
\le \frac{C}{\sqrt{nh_n}}\lto0,
\end{align*}
%%%%%%
\begin{align*}
%&
\EEt\left[
\frac{\sqrt{h_n}}{n}\sum_{i=1}^n\sum_{\ell=1}^d
\sup_{\alpha\in\Theta_A}\|\partial_\alpha\kappa_{i-1}^\ell(\alpha)\|
\sup_{\beta\in\Theta_B}|b_{i-1}^\ell(\beta)|
\right]
%\\
%&\le
%\frac{\sqrt{h_n}}{n}\sum_{i=1}^n\sum_{\ell=1}^d
%\EEt\left[
%\sup_{\alpha\in\Theta_A}\|\partial_\alpha\kappa_{i-1}^\ell(\alpha)\|^2
%\right]^{1/2}
%\EEt\left[
%\sup_{\beta\in\Theta_B}|b_{i-1}^\ell(\beta)|^2
%\right]^{1/2}\\
%&
\le C\sqrt{h_n}\lto0
\end{align*}
%%%%%%
and
\begin{align*}
&\EEt\left[
\frac{1}{n\sqrt{nh_n}}\sum_{i=1}^n
\sum_{\ell=1}^d
\sup_{\alpha\in\Theta_A}\left\|
\partial_\alpha^2\kappa_{i-1}^\ell(\alpha)
\right\|
\left(|(\DeX)^\ell|
+h_n
\sup_{\beta\in\Theta_B}
|b_{i-1}^\ell(\beta)|\right)\right]\\
&\le
\frac{C}{n\sqrt{nh_n}}\sum_{i=1}^n\sum_{\ell=1}^d
\EEt\left[
\sup_{\alpha\in\Theta_A}\left\|
\partial_\alpha^2\kappa_{i-1}^\ell(\alpha)
\right\|^2
\right]^{1/2}
\left(\EEt\left[
((\DeX)^\ell)^2\right]
+h_n^2
\EEt\left[\sup_{\beta\in\Theta_B}
\left|b_{i-1}^\ell(\beta)\right|^2
\right]\right)^{1/2}\\
&\le \frac{C'}{\sqrt n}\lto0.
\end{align*}
%%%%%%%%%%%%%%%%
Hence, 
%from \textbf{[A8]}, we have, 
in the same way as the proof of \eqref{ap3},
we see from  \textbf{[A8]} that
\begin{align*}
\frac{1}{n\sqrt{nh_n}}\sum_{i=1}^n\|\Qd\|=o_p(1),\quad
\frac{\sqrt{h_n}}{n}\sum_{i=1}^n\|\Qf\|=o_p(1),\quad
\frac1{n\sqrt{nh_n}}\sum_{i=1}^n\|\Qg\|=o_p(1).
\end{align*}
%%%%%%%%%%%%%%%%%%%%%%%%%%%%%%%%%%%%%%%%%%%%%%%%%%%%%%%%%%%%%%%%%%%%%%%%%%%%
\end{proof}

%\input{change_point5}
%%%%%%%%%%%%%%%%%%%%%%%%%%%%%%%%%%%%%%%%%%%%%%%%%%%%%%%%%%%%%%%%%%%%%%%%%%%%%%%
\begin{proof}[\bf{Proof of Theorem \ref{th03}}]
Let
\begin{align*}
\zeta_i=
\partial_{\beta}b_{i-1}(\beta_0)^\TT
A_{i-1}^{-1}(\alpha_0)
\left(\DeX-h_nb_{i-1}(\beta_0)\right).
\end{align*}
%%%%%%%%%%%
%%%%%%%%%%%
By the Taylor expansion, we have
\begin{align*}
(A_{i-1}^{-1}(\hat\alpha_n))^{\ell_1,\ell_2}
=
(A_{i-1}^{-1}(\alpha_0))^{\ell_1,\ell_2}
+\partial_\alpha(A_{i-1}^{-1}(\alpha_0))^{\ell_1,\ell_2}(\hat\alpha_n-\alpha_0)
+(\hat\alpha_n-\alpha_0)^\TT\Aa(\hat\alpha_n-\alpha_0),
\end{align*}
where 
\begin{align*}
\Aa=\int_0^1(1-u)\partial_\alpha^2
\left(A_{i-1}^{-1}(\alpha_0+u(\hat\alpha_n-\alpha_0)\right)^{\ell_1,\ell_2}\dd u.
\end{align*}

%%%%%%%%%%%%
Then, we can express
\begin{align*}
\hat{\zeta_i^\ell}
%%%%
&=\sum_{\ell_1,\ell_2=1}^d
\partial_{\beta^\ell}b^{\ell_1}_{i-1}(\hat\beta_n)
(A_{i-1}^{-1}(\hat\alpha_n))^{\ell_1,\ell_2}
\left(\DeX-h_nb_{i-1}(\hat\beta_n)\right)^{\ell_2}\nonumber\\
%%%
&=\sum_{\ell_1,\ell_2=1}^d
\partial_{\beta^\ell}b^{\ell_1}_{i-1}(\hat\beta_n)
(A_{i-1}^{-1}(\alpha_0))^{\ell_1,\ell_2}
\left(\DeX-h_nb_{i-1}(\hat\beta_n)\right)^{\ell_2}\\
%%%
&\qquad+\frac{1}{\sqrt n}\sum_{\ell_1,\ell_2=1}^d
\partial_{\beta^\ell}b^{\ell_1}_{i-1}(\hat\beta_n)
\partial_\alpha(A_{i-1}^{-1}(\alpha_0))^{\ell_1,\ell_2}
\left(\DeX-h_nb_{i-1}(\hat\beta_n)\right)^{\ell_2}
\sqrt n(\hat\alpha_n-\alpha_0)
\\
%%%
&\qquad+\left(\sqrt n(\hat\alpha_n-\alpha_0)\right)^\TT
\left(\frac1n\sum_{\ell_1,\ell_2=1}^d
\partial_{\beta^\ell}b^{\ell_1}_{i-1}(\hat\beta_n)\Aa
\left(\DeX-h_nb_{i-1}(\hat\beta_n)\right)^{\ell_2}\right)
\sqrt n(\hat\alpha_n-\alpha_0)\\
%%%
&=
\sum_{\ell_1,\ell_2=1}^d
\partial_{\beta^\ell}b^{\ell_1}_{i-1}(\hat\beta_n)
(A_{i-1}^{-1}(\alpha_0))^{\ell_1,\ell_2}
(\DeX-h_nb_{i-1}(\beta_0))^{\ell_2}\\
&\qquad+h_n\sum_{\ell_1,\ell_2=1}^d
\partial_{\beta^\ell}b^{\ell_1}_{i-1}(\hat\beta_n)
(A_{i-1}^{-1}(\alpha_0))^{\ell_1,\ell_2}
(b_{i-1}^{\ell_2}(\beta_0)-b_{i-1}^{\ell_2}(\hat\beta_n))\\
%%%
&\qquad+\frac{1}{\sqrt n}\sum_{\ell_1,\ell_2=1}^d
\partial_{\beta^\ell}b^{\ell_1}_{i-1}(\hat\beta_n)
\partial_\alpha(A_{i-1}^{-1}(\alpha_0))^{\ell_1,\ell_2}
\left(\DeX-h_nb_{i-1}(\hat\beta_n)\right)^{\ell_2}
\sqrt n(\hat\alpha_n-\alpha_0)
\\
%%%
&\qquad+\left(\sqrt n(\hat\alpha_n-\alpha_0)\right)^\TT
\left(\frac1n\sum_{\ell_1,\ell_2=1}^d
\partial_{\beta^\ell}b^{\ell_1}_{i-1}(\hat\beta_n)\Aa
\left(\DeX-h_nb_{i-1}(\hat\beta_n)\right)^{\ell_2}\right)
\sqrt n(\hat\alpha_n-\alpha_0)\\
&=:J_1+J_2+J_3+J_4.
\end{align*}
%%%%%%%%%%%%%%%%%%%%%%%%%%%%%%%%%%%%%%%%%%%%%%%%%%%%%%%%%%%%%%%%
We set 
\begin{align*}
A\otimes x^{\otimes k}=\sum_{\ell_1,\ldots,\ell_k=1}^q
A^{\ell_1,\ldots,\ell_k}x^{\ell_1}\cdots x^{\ell_k},\quad
\text{ for }
A\in\underbrace{\mathbb R^q\otimes\cdots\otimes\mathbb R^q}_{k},\ x\in\mathbb R^q.
\end{align*} 
%%%%%%%%%%%%%%%%%%%%%%%%%%%%%%%%%%%%%%%%%%%%%%%%%%%%%%%%%%%%%
Note that, by the Taylor expansion,
\begin{align*}
%\partial_{\beta^\ell}b^{\ell_1}_{i-1}(\hat\beta_n)
%&=
%\partial_{\beta^\ell}b^{\ell_1}_{i-1}(\beta_0)
%+\frac{1}{\sqrt{nh_n}}
%\partial_\beta\partial_{\beta^\ell}b^{\ell_1}_{i-1}(\beta_0)
%\otimes\left(\sqrt{nh_n}(\hat\beta_n-\beta_0)\right)\\
%
%&\qquad+\frac1{nh_n}\Bd
%\otimes\left(\sqrt{nh_n}(\hat\beta_n-\beta_0)\right)^{\otimes 2}\\
%%%%%%
\partial_{\beta^\ell}b^{\ell_1}_{i-1}(\hat\beta_n)
&=
\partial_{\beta^\ell}b^{\ell_1}_{i-1}(\beta_0)
+\sum_{j=1}^{m-1}\frac{1}{(nh_n)^{j/2}}
\partial_\beta^j\partial_{\beta^\ell}b^{\ell_1}_{i-1}(\beta_0)
\otimes\left(\sqrt{nh_n}(\hat\beta_n-\beta_0)\right)^{\otimes j}\\
&\qquad+\frac1{(nh_n)^{m/2}}\mathcal B_{m,i-1}^{\ell,\ell_1}
\otimes\left(\sqrt{nh_n}(\hat\beta_n-\beta_0)\right)^{\otimes m},
\end{align*}
where
\begin{align*}
\mathcal B_{m,i-1}^{\ell,k}
&=
\frac1{(m-1)!}\int_0^1(1-u)^{m-1}
\partial_\beta^m\partial_{\beta^\ell}b_{i-1}^k(\beta_0+u(\hat\beta_n-\beta_0))\dd u,\\
\partial_\beta^j\partial_{\beta^\ell}b^k(x,\beta)
&=\left(\partial_{\beta^{\ell_j}}\cdots\partial_{\beta^{\ell_1}}
\partial_{\beta^\ell}b^k(x,\beta)\right)_{\ell_1,\ldots,\ell_j}
\in\underbrace{\mathbb R^q\otimes\cdots\otimes\mathbb R^q}_{j}.
\end{align*}
Then,
%%%%%%%%%%%%%%%%%%%%%
\begin{align*}
J_1
&=
\sum_{\ell_1,\ell_2=1}^d
\partial_{\beta^\ell}b^{\ell_1}_{i-1}(\hat\beta_n)
(A_{i-1}^{-1}(\alpha_0))^{\ell_1,\ell_2}
(\DeX-h_nb_{i-1}(\beta_0))^{\ell_2}\\
%%%
&=
\sum_{\ell_1,\ell_2=1}^d
\left(
\partial_{\beta^\ell}b^{\ell_1}_{i-1}(\beta_0)
+\sum_{j=1}^{m-1}\frac{1}{(nh_n)^{j/2}}
\partial_\beta^j\partial_{\beta^\ell}b^{\ell_1}_{i-1}(\beta_0)
\otimes\left(\sqrt{nh_n}(\hat\beta_n-\beta_0)\right)^{\otimes j}
\right.\\
%%%%
&\left.\qquad\qquad
+\frac1{(nh_n)^{m/2}}\Bm
\otimes\left(\sqrt{nh_n}(\hat\beta_n-\beta_0)\right)^{\otimes m}
\right)
(A_{i-1}^{-1}(\alpha_0))^{\ell_1,\ell_2}
(\DeX-h_nb_{i-1}(\beta_0))^{\ell_2}\\
%%%%%%%%%%%%%
&=:
\zeta_i^\ell
+\sum_{j=1}^{m-1}\frac{1}{(nh_n)^{j/2}}\Yj
\otimes\left(\sqrt{nh_n}(\hat\beta_n-\beta_0)\right)^{\otimes j}
+\frac{1}{(nh_n)^{m/2}}\Ym
\otimes\left(\sqrt{nh_n}(\hat\beta_n-\beta_0)\right)^{\otimes m}
\end{align*}
%%%%%%%%%%%%%%%%%%%%%%%%%%%%%%%%%%%%%%%%%%%%%%%%%%%%%%%%%%%%%%%%%
{ and}
\begin{align*}
J_2
&=
h_n\sum_{\ell_1,\ell_2=1}^d
\partial_{\beta^\ell}b^{\ell_1}_{i-1}(\hat\beta_n)
(A_{i-1}^{-1}(\alpha_0))^{\ell_1,\ell_2}
(b_{i-1}^{\ell_2}(\beta_0)-b_{i-1}^{\ell_2}(\hat\beta_n))\\
%%%
&=
h_n\sum_{\ell_1,\ell_2=1}^d
\left(
\partial_{\beta^\ell}b^{\ell_1}_{i-1}(\beta_0)
+\frac{1}{\sqrt{nh_n}}
\partial_\beta\partial_{\beta^\ell}b^{\ell_1}_{i-1}(\beta_0)
%\otimes
\left(\sqrt{nh_n}(\hat\beta_n-\beta_0)\right)\right.\\
&\left.\qquad\qquad\qquad+\frac1{nh_n}\Bd
\otimes\left(\sqrt{nh_n}(\hat\beta_n-\beta_0)\right)^{\otimes 2}\right)
(A_{i-1}^{-1}(\alpha_0))^{\ell_1,\ell_2}
(b_{i-1}^{\ell_2}(\beta_0)-b_{i-1}^{\ell_2}(\hat\beta_n))\\
%%%%%%%%%%%%%%%%
&=
h_n\sum_{\ell_1,\ell_2=1}^d
\partial_{\beta^\ell}b^{\ell_1}_{i-1}(\beta_0)
(A_{i-1}^{-1}(\alpha_0))^{\ell_1,\ell_2}
(b_{i-1}^{\ell_2}(\beta_0)-b_{i-1}^{\ell_2}(\hat\beta_n))\\
%%%
&\qquad+
\sqrt{\frac{h_n}{n}}\left(
\sum_{\ell_1,\ell_2=1}^d
\partial_\beta\partial_{\beta^\ell}b^{\ell_1}_{i-1}(\beta_0)
%\otimes
(A_{i-1}^{-1}(\alpha_0))^{\ell_1,\ell_2}
(b_{i-1}^{\ell_2}(\beta_0)-b_{i-1}^{\ell_2}(\hat\beta_n))
\right)\sqrt{nh_n}(\hat\beta_n-\beta_0)
\\
%%%
&\qquad+
\frac1{n}\left(\sum_{\ell_1,\ell_2=1}^d\Bd
(A_{i-1}^{-1}(\alpha_0))^{\ell_1,\ell_2}
(b_{i-1}^{\ell_2}(\beta_0)-b_{i-1}^{\ell_2}(\hat\beta_n))\right)
\otimes\left(\sqrt{nh_n}(\hat\beta_n-\beta_0)\right)^{\otimes 2}
\\
%%%
&=
h_n\sum_{\ell_1,\ell_2=1}^d
\partial_{\beta^\ell}b^{\ell_1}_{i-1}(\beta_0)
(A_{i-1}^{-1}(\alpha_0))^{\ell_1,\ell_2}
\left(
-\frac{1}{\sqrt{nh_n}}\partial_\beta b_{i-1}^{\ell_2}(\beta_0)
\sqrt{nh_n}(\hat\beta_n-\beta_0)\right.\\
&\qquad\qquad\left.
-\frac{1}{nh_n}
\int_0^1(1-u)\partial_\beta^2 b_{i-1}^{\ell_2}(\beta_0+u(\hat\beta_n-\beta_0))\dd u
\otimes\left(\sqrt{nh_n}(\hat\beta_n-\beta_0)\right)^{\otimes2}
\right)\\
%%%%%%%
&\qquad-
\frac{1}{n}\left(\sum_{\ell_1,\ell_2=1}^d
\partial_\beta\partial_{\beta^\ell}b^{\ell_1}_{i-1}(\beta_0)^\TT
(A_{i-1}^{-1}(\alpha_0))^{\ell_1,\ell_2}
\int_0^1\partial_\beta b_{i-1}^{\ell_2}(\beta_0+u(\hat\beta_n-\beta_0))\dd u\right)
\otimes\left(\sqrt{nh_n}(\hat\beta_n-\beta_0)\right)^{\otimes2}
\\
%%%
&\qquad+
\frac1{n}\left(\sum_{\ell_1,\ell_2=1}^d\Bd
(A_{i-1}^{-1}(\alpha_0))^{\ell_1,\ell_2}
(b_{i-1}^{\ell_2}(\beta_0)-b_{i-1}^{\ell_2}(\hat\beta_n))\right)
\otimes\left(\sqrt{nh_n}(\hat\beta_n-\beta_0)\right)^{\otimes 2}\\
%%%%%%%%%%%%%
&=:
\sqrt{\frac{h_n}{n}}\Zi
\left(\sqrt{nh_n}(\hat\beta_n-\beta_0)\right)
+\frac1n
\Zd\otimes\left(\sqrt{nh_n}(\hat\beta_n-\beta_0)\right)^{\otimes2}.
\end{align*}
%%%%%%%%%%%%%%%%%%%%%%%%%%%%%%%%%%%%%%%%%%%%%%%%%%%%%%%%%%%%%%%%%
{ Moreover, }
\begin{align*}
J_3
&=\frac1{\sqrt n}
\sum_{\ell_1,\ell_2=1}^d
\left(\partial_{\beta^\ell}b^{\ell_1}_{i-1}(\beta_0)
+\frac{1}{\sqrt{nh_n}}\int_0^1\partial_\beta\partial_{\beta^\ell}b^{\ell_1}_{i-1}
(\beta_0+u(\hat\beta_n-\beta_0))\dd u
\sqrt{nh_n}(\hat\beta_n-\beta_0)\right)\\
%%%
&\qquad\qquad\cdot\partial_\alpha(A_{i-1}^{-1}(\alpha_0))^{\ell_1,\ell_2}
\left(\DeX-h_nb_{i-1}(\hat\beta_n)\right)^{\ell_2}
\sqrt n(\hat\alpha_n-\alpha_0)\\
%%%%%%
&=
\frac1{\sqrt n}
\sum_{\ell_1,\ell_2=1}^d
\partial_{\beta^\ell}b^{\ell_1}_{i-1}(\beta_0)
\partial_\alpha(A_{i-1}^{-1}(\alpha_0))^{\ell_1,\ell_2}
(\DeX)^{\ell_2}
\sqrt n(\hat\alpha_n-\alpha_0)\\
%%%
&\qquad-\frac{h_n}{\sqrt n}
\sum_{\ell_1,\ell_2=1}^d
\partial_{\beta^\ell}b^{\ell_1}_{i-1}(\beta_0)
\partial_\alpha(A_{i-1}^{-1}(\alpha_0))^{\ell_1,\ell_2}
b_{i-1}^{\ell_2}(\hat\beta_n)
\sqrt n(\hat\alpha_n-\alpha_0)\\
%%%
&\qquad+\left(\sqrt{nh_n}(\hat\beta_n-\beta_0)\right)^\TT\frac1{n\sqrt{h_n}}
\sum_{\ell_1,\ell_2=1}^d
\int_0^1\partial_\beta\partial_{\beta^\ell}b^{\ell_1}_{i-1}
(\beta_0+u(\hat\beta_n-\beta_0))^\TT\dd u\\
&\qquad\qquad\cdot
\partial_\alpha(A_{i-1}^{-1}(\alpha_0))^{\ell_1,\ell_2}
\left(\DeX-h_nb_{i-1}(\hat\beta_n)\right)^{\ell_2}
\sqrt n(\hat\alpha_n-\alpha_0)\\
%%%
&=:\frac1{\sqrt n}\Zt\left(\sqrt n(\hat\alpha_n-\alpha_0)\right)
+\frac{h_n}{\sqrt n}\Zf\left(\sqrt n(\hat\alpha_n-\alpha_0)\right)\\
&\qquad
+\frac{1}{n\sqrt{h_n}}\left(\sqrt{nh_n}(\hat\beta_n-\beta_0)\right)^\TT
\Zg\left(\sqrt n(\hat\alpha_n-\alpha_0)\right),\\
%%%%%%%%%%%%%%%%%%%%%%%%%%%%%%%%%%%%%%%%%%%%%%%%%%%%%%%%%%%%%%%%%
J_4&=:
\frac1n\left(\sqrt n(\hat\alpha_n-\alpha_0)\right)^\TT \Zs
\left(\sqrt n(\hat\alpha_n-\alpha_0)\right).
\end{align*}
%%%%%%%%%%%%%%%%%%%
Note that $\mathcal I_n\pto\mathcal I$ and
\begin{align*}
T_n^\beta
=\frac{1}{\sqrt{nh_n}}\max_{1\le k\le n}
\left\|\mathcal I_n^{-1/2}\left(\sum_{i=1}^k\hat\zeta_i-\frac kn
\sum_{i=1}^n\hat\zeta_i\right)\right\|
=\frac{1}{\sqrt{nh_n}}\max_{1\le k\le n}
\left\|\left(\mathcal I_n^{-1/2}\mathcal I^{1/2}\right)
\mathcal I^{-1/2}\left(\sum_{i=1}^k\hat\zeta_i-\frac kn
\sum_{i=1}^n\hat\zeta_i\right)\right\|.
\end{align*}
Therefore, it is enough to show

%%%%%%%%
\begin{align}
\frac{1}{\sqrt{nh_n}}\max_{1\le k\le n}
\left\|
\mathcal I^{-1/2}\left(\sum_{i=1}^k\zeta_i-\frac{k}{n}\sum_{i=1}^n\zeta_i\right)
\right\|
&\dto \sup_{0\le s\le 1}\| { \boldsymbol B_q^0(s)} \|,
\label{hf1}\\%%%%
\frac{1}{(nh_n)^{(j+1)/2}}\max_{1\le k\le n}
\left\|
\sum_{i=1}^k\Yj-\frac{k}{n}\sum_{i=1}^n\Yj
\right\|&=o_p(1),\qquad(1\le j\le m-1)
\label{hf2}\\%%%%
\frac{1}{(nh_n)^{(m+1)/2}}\max_{1\le k\le n}
\left\|
\sum_{i=1}^k\Ym-\frac{k}{n}\sum_{i=1}^n\Ym
\right\|&=o_p(1),
\label{hf3}\\%%%%
\frac{1}{n}\max_{1\le k\le n}
\left\|
\sum_{i=1}^k\Zi-\frac{k}{n}\sum_{i=1}^n\Zi
\right\|&=o_p(1),
\label{hf4}\\%%%%
\frac{1}{n\sqrt{nh_n}}\max_{1\le k\le n}
\left\|
\sum_{i=1}^k\Zd
-\frac{k}{n}\sum_{i=1}^n\Zd
\right\|&=o_p(1),
\label{hf5}\\%%%%
\frac1{n\sqrt{h_n}}\max_{1\le k\le n}\left\|
\sum_{i=1}^n\Zt-\frac kn\sum_{i=1}^n\Zt
\right\|&=o_p(1),
\label{hf6}\\%%%%
\frac{\sqrt{h_n}}{n}\max_{1\le k\le n}\left\|
\sum_{i=1}^k\Zf-\frac kn\sum_{i=1}^n\Zf
\right\|&=o_p(1),
\label{hf7}\\%%%%%%%
\frac1{n^{3/2}h_n}\max_{1\le k\le n}\left\|
\sum_{i=1}^k\Zg-\frac kn\sum_{i=1}^n\Zg
\right\|&=o_p(1),
\label{hf8}\\%%%%%%%
\frac1{n\sqrt{nh_n}}\max_{1\le k\le n}\left\|
\sum_{i=1}^k\Zs-\frac kn\sum_{i=1}^n\Zs
\right\|&=o_p(1).
\label{hf9}%%%%%%
\end{align}

%%%%%%%%%%%%%%%%%%%%%%%%%%%%%%%%(.)%%%%%%%%%%%%%%%%%%%%%%%%%%%%%%%%%%%%%%%%
\textit{Proof of \eqref{hf1}.} 
If we prove
\begin{align}\label{bm3}
\mathcal W_n(s):=\frac{1}{\sqrt{nh_n}}\sum_{i=1}^{[ns]}\mathcal I^{-1/2}\zeta_i
\wto  { \boldsymbol B_q(s)} \quad\text{in }\mathbb D[0,1],
\end{align}
{ then, it follows from the continuous mapping theorem that}
%we obtain, applying the continuous mapping theorem,
\begin{align*}
\frac{1}{\sqrt{nh_n}}\max_{1\le k\le n}
\left\|
\mathcal I^{-1/2}\left(\sum_{i=1}^k\zeta_i-\frac{k}{n}\sum_{i=1}^n\zeta_i\right)
\right\|
&=
\max_{1\le k\le n}\left\|\frac{1}{\sqrt{nh_n}}\sum_{i=1}^k\mathcal I^{-1/2}\zeta_i
-\frac{k}{n}\frac{1}{\sqrt{nh_n}}\sum_{i=1}^n\mathcal I^{-1/2}\zeta_i\right\|\\
&=
\sup_{0\le s\le 1}\left\|\frac{1}{\sqrt{nh_n}}\sum_{i=1}^{[ns]}\mathcal I^{-1/2}\zeta_i
-\frac{[ns]}{n}\frac{1}{\sqrt{nh_n}}\sum_{i=1}^n\mathcal I^{-1/2}\zeta_i\right\|\\
&=
\sup_{0\le s\le 1}\left\|\mathcal W_n(s)
-\frac{[ns]}{n}\mathcal W_n(1)\right\|\\
&\dto
\sup_{0\le s\le 1}\| { \boldsymbol B_q(s)} -s { \boldsymbol B_q(1)} \|
=
\sup_{0\le s\le 1}\| { \boldsymbol B_q^0(s)} \|.
\end{align*}
%%%%%%%%%%%%%%%%%%%%%%%%%%%%%%%%%%%%%%%%%%%%%%%%%%%%%%%%%%%%

Let us prove \eqref{bm3}.
It is sufficient to show
%It is enough to prove the following.
\begin{align}
\frac1{\sqrt{nh_n}}\sum_{i=1}^{[ns]}
\mathcal I^{-1/2}\left(\zeta_i
-\EEt\left[\zeta_i\middle|\GG\right]\right)
&\wto  { \boldsymbol B_q(s)} \quad{\text{in }}\mathbb D[0,1]{ ,}
\label{qs1}\\
%%%
\frac{1}{\sqrt{nh_n}}\sum_{i=1}^{n}\EEt
\left[\zeta_i\middle|\GG\right]&=o_p(1).
\label{qs2}
\end{align}
%%%%
From Lemma \ref{ke7},
\begin{align*}
\frac{1}{\sqrt{nh_n}}\sum_{i=1}^{n}\EEt
\left[\zeta_i\middle|\GG\right]
=
\frac{1}{\sqrt{nh_n}}\sum_{i=1}^{n}\Rd
=
\sqrt{nh_n^3}\cdot\frac1n\sum_{i=1}^n\Ro
=o_p(1).
\end{align*}
Hence we have \eqref{qs2}.
%%%%%%%%%%%%%%%%%%%%%%%%%%%

According to the Cram\'er-Wold theorem and Corollary 3.8 of McLeish (1974),  
we obtain \eqref{qs1} 
if we show that for all $c\in\mathbb R^q$,
\begin{align*}
\frac1{\sqrt{nh_n}}\sum_{i=1}^{[ns]}
c^\TT\mathcal I^{-1/2}\left(\zeta_i
-\EEt\left[\zeta_i\middle|\GG\right]\right)
&\wto c^\TT { \boldsymbol B_q(s)} \quad{\text{in }}\mathbb D[0,1],
\end{align*}
{ i.e.,} 
\begin{align}
\frac1{\sqrt{nh_n}}\sum_{i=1}^{[ns]}\left\|
\EEt\left[
c^\TT\mathcal I^{-1/2}\left(\zeta_i
-\EEt\left[\zeta_i\middle|\GG\right]\right)
\middle|\GG\right]\right\|&=o_p(1),
\label{zy0}\\
%%%%
\frac1{nh_n}\sum_{i=1}^{[ns]}
\EEt\left[
\left(c^\TT\mathcal I^{-1/2}\left(\zeta_i
-\EEt\left[\zeta_i\middle|\GG\right]\right)\right)^2
\middle|\GG\right]&\pto \|c\|^2s,
\label{zy1}\\
%%%%
\frac1{(nh_n)^2}\sum_{i=1}^{[ns]}
\EEt\left[
\left(c^\TT\mathcal I^{-1/2}\left(\zeta_i
-\EEt\left[\zeta_i\middle|\GG\right]\right)\right)^4
\middle|\GG\right]
&=o_p(1),
\label{zy2}
\end{align}
for all $s\in[0,1]$.

%%%%%%%%
\eqref{zy0} is obvious.
%%%%%%%%
From Lemma \ref{ke7},
\begin{align*}
&\frac1{nh_n}\sum_{i=1}^{[ns]}
\EEt\left[
\left(c^\TT\mathcal I^{-1/2}\left(\zeta_i
-\EEt\left[\zeta_i\middle|\GG\right]\right)\right)^2
\middle|\GG\right]\\
&=\frac1{nh_n}\sum_{i=1}^{[ns]}
\EEt\left[
c^\TT\mathcal I^{-1/2}
\left(\zeta_i-\EEt\left[\zeta_i\middle|\GG\right]\right)^{\otimes2}
\mathcal I^{-1/2}c
\middle|\GG\right]\\
%%%
&=\frac1{nh_n}\sum_{i=1}^{[ns]}
c^\TT\mathcal I^{-1/2}
\left(\EEt\left[\zeta_i\zeta_i^\TT\middle|\GG\right]
-\EEt[\zeta_i|\GG]\EEt[\zeta_i|\GG]^\TT\right)
\mathcal I^{-1/2}c
\\
&=\frac{[ns]}{n}\frac1{[ns]h_n}\sum_{i=1}^{[ns]}
c^\TT\mathcal I^{-1/2}
\left(h_n
\partial_\beta b_{i-1}(\beta_0)^\TT
A_{i-1}^{-1}(\alpha_0)\partial_\beta b_{i-1}(\beta_0)
+\Rd\right)
\mathcal I^{-1/2}c\\
&\pto sc^\TT \mathcal I^{-1/2}\mathcal I \mathcal I^{-1/2}c=\|c\|^2s,
\end{align*}
 { i.e.,}  \eqref{zy1} holds.
%%%%%%%%%%%%%
Since $\EEt[\|\zeta_i\|^4|\GG]=\Rd$,
\begin{align*}
&\frac1{(nh_n)^2}\sum_{i=1}^{[ns]}
\EEt\left[
\left(c^\TT\mathcal I^{-1/2}\left(\zeta_i
-\EEt\left[\zeta_i\middle|\GG\right]\right)\right)^4
\middle|\GG\right]\\
&\le\frac1{(nh_n)^2}\sum_{i=1}^{[ns]}
\EEt\left[
\|c^\TT\mathcal I^{-1/2}\|^4
\left\|\zeta_i-\EEt\left[\zeta_i\middle|\GG\right]\right\|^4
\middle|\GG\right]\\
%%%
&\le\frac{C}{(nh_n)^2}\sum_{i=1}^{[ns]}
\EEt\left[
\|\zeta_i\|^4+R_{i-1}(h_n^8,\theta)
\middle|\GG\right]\\
%%%
&=\frac1{(nh_n)^2}\sum_{i=1}^{[ns]}\Rd\\
%%%
&=\frac 1{n}\cdot\frac1n\sum_{i=1}^{[ns]}
\Ro
=o_p(1).
\end{align*}
Therefore we obtain \eqref{zy2}. This completes the proof of \eqref{hf1}.
%%%%%%%%%%%%%%%%%%%%%%%%%%%%%%%%%%%%%%%%%%%%%%%%%%%%%%%%%%%%%%%%%%%%%%%%%%%%%%%

%%%%%%%%%%%%%%%%%%%%%%%%%%%%%%%%%(.)%%%%%%%%%%%%%%%%%%%%%%%%%%%%%%%%%%%%%%
\textit{{ Proofs} of \eqref{hf2}, \eqref{hf4} and \eqref{hf6}.}
Noting that
\begin{align*}
\Yj
&=
\sum_{\ell_1,\ell_2=1}^d
\partial_\beta^j\partial_{\beta^\ell}b^{\ell_1}_{i-1}(\beta_0)
(A_{i-1}^{-1}(\alpha_0))^{\ell_1,\ell_2}
(\DeX-h_nb_{i-1}(\beta_0))^{\ell_2}\\
&=
\sum_{\ell_1,\ell_2=1}^d
\partial_\beta^j\partial_{\beta^\ell}b^{\ell_1}_{i-1}(\beta_0)
(A_{i-1}^{-1}(\alpha_0))^{\ell_1,\ell_2}
(\DeX)^{\ell_2}
-h_n\sum_{\ell_1,\ell_2=1}^d
\partial_\beta^j\partial_{\beta^\ell}b^{\ell_1}_{i-1}(\beta_0)
(A_{i-1}^{-1}(\alpha_0))^{\ell_1,\ell_2}
b_{i-1}^{\ell_2}(\beta_0),\\
\Zi
&=
-\sum_{\ell_1,\ell_2=1}^d
\partial_{\beta^\ell}b^{\ell_1}_{i-1}(\beta_0)
(A_{i-1}^{-1}(\alpha_0))^{\ell_1,\ell_2}
\partial_\beta b_{i-1}^{\ell_2}(\beta_0),\\
\Zt
&=
\sum_{\ell_1,\ell_2=1}^d
\partial_{\beta^\ell}b^{\ell_1}_{i-1}(\beta_0)
\partial_\alpha(A_{i-1}^{-1}(\alpha_0))^{\ell_1,\ell_2}
(\DeX)^{\ell_2},
\end{align*}
we have, 
by using {\textbf{[A9]}}, Lemmas \ref{lem1} and \ref{lem2}, 
%%%%%%%
\begin{align*}
&\frac{1}{(nh_n)^{(j+1)/2}}\max_{1\le k\le n}
\left\|
\sum_{i=1}^k\Yj-\frac{k}{n}\sum_{i=1}^n\Yj
\right\|\\
&\le
\sum_{\ell_1,\ell_2=1}^d
\frac{1}{nh_n}\max_{1\le k\le n}
\left\|
\sum_{i=1}^k
\partial_\beta^j\partial_{\beta^\ell}b^{\ell_1}_{i-1}(\beta_0)
(A_{i-1}^{-1}(\alpha_0))^{\ell_1,\ell_2}
(\DeX)^{\ell_2}\right.\\
&\qquad\qquad\qquad\qquad\qquad\qquad\left.
-\frac{k}{n}\sum_{i=1}^n
\partial_\beta^j\partial_{\beta^\ell}b^{\ell_1}_{i-1}(\beta_0)
(A_{i-1}^{-1}(\alpha_0))^{\ell_1,\ell_2}
(\DeX)^{\ell_2}
\right\|\\
&\qquad+
\sum_{\ell_1,\ell_2=1}^d
\frac{1}{n}\max_{1\le k\le n}
\left\|
\sum_{i=1}^k
\partial_\beta^j\partial_{\beta^\ell}b^{\ell_1}_{i-1}(\beta_0)
(A_{i-1}^{-1}(\alpha_0))^{\ell_1,\ell_2}
b_{i-1}^{\ell_2}(\beta_0)\right.\\
&\qquad\qquad\qquad\qquad\qquad\qquad\left.
-\frac{k}{n}\sum_{i=1}^n
\partial_\beta^j\partial_{\beta^\ell}b^{\ell_1}_{i-1}(\beta_0)
(A_{i-1}^{-1}(\alpha_0))^{\ell_1,\ell_2}
b_{i-1}^{\ell_2}(\beta_0)
\right\|\\
&=o_p(1),
\end{align*}
%%%%%
\begin{align*}
&\frac{1}{n}\max_{1\le k\le n}
\left\|
\sum_{i=1}^k\Zi-\frac{k}{n}\sum_{i=1}^n\Zi
\right\|\\
&\le
\sum_{\ell_1,\ell_2=1}^d
\frac{1}{n}\max_{1\le k\le n}
\left\|
\sum_{i=1}^k
\partial_{\beta^\ell}b^{\ell_1}_{i-1}(\beta_0)
(A_{i-1}^{-1}(\alpha_0))^{\ell_1,\ell_2}
\partial_\beta b_{i-1}^{\ell_2}(\beta_0)\right.\\
&\qquad\qquad\qquad\qquad\qquad\qquad\left.-\frac{k}{n}\sum_{i=1}^n
\partial_{\beta^\ell}b^{\ell_1}_{i-1}(\beta_0)
(A_{i-1}^{-1}(\alpha_0))^{\ell_1,\ell_2}
\partial_\beta b_{i-1}^{\ell_2}(\beta_0)
\right\|\\
&=o_p(1),
\end{align*}
%%%%%
and
\begin{align*}
&\frac1{n\sqrt{h_n}}\max_{1\le k\le n}\left\|
\sum_{i=1}^n\Zt-\frac kn\sum_{i=1}^n\Zt
\right\|\\
&=
\sum_{\ell_1,\ell_2=1}^d
\frac{1}{n\sqrt{h_n}}\max_{1\le k\le n}
\left\|
\sum_{i=1}^k
\partial_{\beta^\ell}b^{\ell_1}_{i-1}(\beta_0)
\partial_\alpha(A_{i-1}^{-1}(\alpha_0))^{\ell_1,\ell_2}
(\DeX)^{\ell_2}\right.\\
&\qquad\qquad\qquad\qquad\qquad\qquad\left.
-\frac{k}{n}\sum_{i=1}^n
\partial_{\beta^\ell}b^{\ell_1}_{i-1}(\beta_0)
\partial_\alpha(A_{i-1}^{-1}(\alpha_0))^{\ell_1,\ell_2}
(\DeX)^{\ell_2}
\right\|\\
&=o_p(1).
\end{align*}
%%%%%%%%%%%%%%%%%%%%%%%%%%%%%%%%%%%%%%%%%%%%%%%%%%%%%%%%%%%%%%%%%%%%%%%%%%%%%

%%%%%%%%%%%%%%%%%%%%%%%%%%%%%%%(),()%%%%%%%%%%%%%%%%%%%%%%%%%%%%%%%
\textit{{ Proofs} of \eqref{hf3}, \eqref{hf5}, \eqref{hf7}, \eqref{hf8} and \eqref{hf9}.} 
Since $\theta_0\in\mathrm{Int\,}\Theta$, 
there exists an open neighborhood $\mathcal O_{\theta_0}$ of 
$\theta_0$ such that $\mathcal O_{\theta_0}\subset\Theta$. 
Note that
\begin{align*}
\Ym
&=
\sum_{\ell_1,\ell_2=1}^d
\Bm
(A_{i-1}^{-1}(\alpha_0))^{\ell_1,\ell_2}
(\DeX-h_nb_{i-1}(\beta_0))^{\ell_2}\\
&=
\sum_{\ell_1,\ell_2=1}^d
\frac1{(m-1)!}\int_0^1(1-u)^m
\partial_\beta^m\partial_{\beta^\ell}b_{i-1}^{\ell_1}(\beta_0+u(\hat\beta_n-\beta_0))\dd u
(A_{i-1}^{-1}(\alpha_0))^{\ell_1,\ell_2}
(\DeX-h_nb_{i-1}(\beta_0))^{\ell_2},\\
\Zd
&=
%&=
%\sum_{\ell_1,\ell_2=1}^d\Bd
%(A_{i-1}^{-1}(\alpha_0))^{\ell_1,\ell_2}
%(b_{i-1}^{\ell_2}(\beta_0)-b_{i-1}^{\ell_2}(\hat\beta_n))\\
%&=
%\sum_{\ell_1,\ell_2=1}^d
%\int_0^1(1-u)
%\partial_\beta^2\partial_{\beta^\ell}b_{i-1}^{\ell_1}(\beta_0+u(\hat\beta_n-\beta_0))\dd u
%(A_{i-1}^{-1}(\alpha_0))^{\ell_1,\ell_2}
%(b_{i-1}^{\ell_2}(\beta_0)-b_{i-1}^{\ell_2}(\hat\beta_n)),\\
-\sum_{\ell_1,\ell_2=1}^d
\partial_{\beta^\ell}b^{\ell_1}_{i-1}(\beta_0)
(A_{i-1}^{-1}(\alpha_0))^{\ell_1,\ell_2}
\int_0^1(1-u)\partial_\beta^2 b_{i-1}^{\ell_2}(\beta_0+u(\hat\beta_n-\beta_0))\dd u\\
%%%%%%%
&\qquad-
\sum_{\ell_1,\ell_2=1}^d
\partial_\beta\partial_{\beta^\ell}b^{\ell_1}_{i-1}(\beta_0)^\TT
(A_{i-1}^{-1}(\alpha_0))^{\ell_1,\ell_2}
\int_0^1\partial_\beta b_{i-1}^{\ell_2}(\beta_0+u(\hat\beta_n-\beta_0))\dd u
\\
%%%
&\qquad+
\sum_{\ell_1,\ell_2=1}^d
\int_0^1(1-u)
\partial_\beta^2\partial_{\beta^\ell}b_{i-1}^{\ell_1}(\beta_0+u(\hat\beta_n-\beta_0))\dd u
(A_{i-1}^{-1}(\alpha_0))^{\ell_1,\ell_2}
(b_{i-1}^{\ell_2}(\beta_0)-b_{i-1}^{\ell_2}(\hat\beta_n)),\\
\Zf
&=
\sum_{\ell_1,\ell_2=1}^d
\partial_{\beta^\ell}b^{\ell_1}_{i-1}(\beta_0)
\partial_\alpha(A_{i-1}^{-1}(\alpha_0))^{\ell_1,\ell_2}
b_{i-1}^{\ell_2}(\hat\beta_n),\\
\Zg
&=
\sum_{\ell_1,\ell_2=1}^d
\int_0^1\partial_\beta\partial_{\beta^\ell}b^{\ell_1}_{i-1}
(\beta_0+u(\hat\beta_n-\beta_0))^\TT\dd u
\partial_\alpha(A_{i-1}^{-1}(\alpha_0))^{\ell_1,\ell_2}
\left(\DeX-h_nb_{i-1}(\hat\beta_n)\right)^{\ell_2},\\
\Zs
&=
\sum_{\ell_1,\ell_2=1}^d
\partial_{\beta^\ell}b^{\ell_1}_{i-1}(\hat\beta_n)\Aa
\left(\DeX-h_nb_{i-1}(\hat\beta_n)\right)^{\ell_2}\\
&=
\sum_{\ell_1,\ell_2=1}^d
\partial_{\beta^\ell}b^{\ell_1}_{i-1}(\hat\beta_n)
\int_0^1(1-u)\partial_\alpha^2
\left(A_{i-1}^{-1}(\alpha_0+u(\hat\alpha_n-\alpha_0)\right)^{\ell_1,\ell_2}\dd u
\left(\DeX-h_nb_{i-1}(\hat\beta_n)\right)^{\ell_2}.
\end{align*}

%%%%%%%%%%%%%%%%%
If $\hat\theta_n=(\hat\alpha_n,\hat\beta_n)\in\mathcal O_{\theta_0}$, then
\begin{align*}
\|\Ym\|
&\le
\sum_{\ell_1,\ell_2=1}^d
\sup_{\beta\in\Theta_B}\|\partial_\beta^m\partial_{\beta^\ell}b_{i-1}^{\ell_1}(\beta)\|
\sup_{\alpha\in\Theta_A}|(A_{i-1}^{-1}(\alpha))^{\ell_1,\ell_2}|
|(\DeX-h_nb_{i-1}(\beta_0))^{\ell_2}|,
\\
\|\Zd\|
&\le
\sum_{\ell_1,\ell_2=1}^d
\left(
\sup_{\beta\in\Theta_B}|\partial_{\beta^\ell}b^{\ell_1}_{i-1}(\beta)|
\sup_{\alpha\in\Theta_A}|(A_{i-1}^{-1}(\alpha))^{\ell_1,\ell_2}|
\sup_{\beta\in\Theta_B}\|
\partial_\beta^2 b_{i-1}^{\ell_2}(\beta)\|\right.\\
%%%%%%%
&\qquad\qquad+
\sup_{\beta\in\Theta_B}\|\partial_\beta\partial_{\beta^\ell}b^{\ell_1}_{i-1}(\beta)\|
\sup_{\alpha\in\Theta_A}|(A_{i-1}^{-1}(\alpha_0))^{\ell_1,\ell_2}|
\sup_{\beta\in\Theta_B}\|\partial_\beta b_{i-1}^{\ell_2}(\beta)\|\\
%%%%%
&\left.\qquad\qquad+2\sup_{\beta\in\Theta_B}
\|\partial_\beta^2\partial_{\beta^\ell}b_{i-1}^{\ell_1}(\beta)\|
\sup_{\beta\in\Theta_B}|b_{i-1}^{\ell_2}(\beta)|\right),\\
\|\Zf\|
&\le
\sum_{\ell_1,\ell_2=1}^d
\sup_{\beta\in\Theta_B}|\partial_{\beta^\ell}b_{i-1}^{\ell_1}(\beta)|
\sup_{\alpha\in\Theta_A}\|\partial_\alpha(A_{i-1}^{-1}(\alpha))^{\ell_1,\ell_2}\|
\sup_{\beta\in\Theta_B}|b_{i-1}^{\ell_2}(\beta)|,
\\
\|\Zg\|
&\le
\sum_{\ell_1,\ell_2=1}^d
\sup_{\beta\in\Theta_B}\|\partial_\beta\partial_{\beta^\ell}b_{i-1}^{\ell_1}(\beta)\|
\sup_{\alpha\in\Theta_A}\|\partial_\alpha(A_{i-1}^{-1}(\alpha))^{\ell_1,\ell_2}\|
\left(|(\DeX)^{\ell_2}|+h_n\sup_{\beta\in\Theta_B}|b_{i-1}^{\ell_2}(\beta)|\right),
\\
\|\Zs\|
&\le
\sum_{\ell_1,\ell_2=1}^d
\sup_{\beta\in\Theta_B}|\partial_{\beta^\ell}b_{i-1}^{\ell_1}(\beta)|
\sup_{\alpha\in\Theta_A}\|\partial_\alpha^2(A_{i-1}^{-1}(\alpha))^{\ell_1,\ell_2}\|
\left(|(\DeX)^{\ell_2}|+h_n\sup_{\beta\in\Theta_B}|b_{i-1}^{\ell_2}(\beta)|\right),
\end{align*}
%%%%%%%%%%%%%%%%%%
\begin{align*}
&\EEt\left[
\frac{1}{(nh_n)^{(m+1)/2}}
\sum_{i=1}^n\sum_{\ell_1,\ell_2=1}^d
\sup_{\beta\in\Theta_B}\|\partial_\beta^m\partial_{\beta^\ell}b_{i-1}^{\ell_1}(\beta)\|
\sup_{\alpha\in\Theta_A}|(A_{i-1}^{-1}(\alpha))^{\ell_1,\ell_2}|
|(\DeX-h_nb_{i-1}(\beta_0))^\ell|
\right]\\
&\le
\frac{1}{(nh_n)^{(m+1)/2}}
\sum_{i=1}^n\sum_{\ell_1,\ell_2=1}^d
\EEt\left[
\sup_{\beta\in\Theta_B}\|\partial_\beta^m\partial_{\beta^\ell}b_{i-1}^{\ell_1}(\beta)\|^2
\sup_{\alpha\in\Theta_A}|(A_{i-1}^{-1}(\alpha))^{\ell_1,\ell_2}|^2
\right]^{1/2}\\
&\qquad\qquad\qquad\cdot\EEt\left[|(\DeX-h_nb_{i-1}(\beta_0))^{\ell_2}|^2\right]^{1/2}\\
&\le\frac{C}{(nh_n^{m/(m-1)})^{(m-1)/2}}\lto0,
\end{align*}
%%%%%%%
\begin{align*}
&\EEt\left[
\frac{1}{n\sqrt{nh_n}}\sum_{i=1}^n
\sum_{\ell_1,\ell_2=1}^d
\sup_{\beta\in\Theta_B}\|\partial_\beta^2\partial_{\beta^\ell}b_{i-1}^{\ell_1}(\beta)\|
\sup_{\beta\in\Theta_B}|b_{i-1}^{\ell_2}(\beta)|
\right]
%\\
%&\le
%\frac{1}{n\sqrt{nh_n}}\sum_{i=1}^n
%\sum_{\ell_1,\ell_2=1}^d
%\EEt\left[
%\sup_{\beta\in\Theta_B}\|\partial_\beta^2\partial_{\beta^\ell}b_{i-1}^{\ell_1}(\beta)\|
%\sup_{\beta\in\Theta_B}|b_{i-1}^{\ell_2}(\beta)|
%\right]\\
%&
\le\frac{C}{\sqrt{nh_n}}\lto0,
\end{align*}
%%%%%%%
\begin{align*}
&\EEt\left[
\frac{\sqrt{h_n}}{n}\sum_{i=1}^n
\sum_{\ell_1,\ell_2=1}^d
\left(
\sup_{\beta\in\Theta_B}|\partial_{\beta^\ell}b^{\ell_1}_{i-1}(\beta)|
\sup_{\alpha\in\Theta_A}|(A_{i-1}^{-1}(\alpha))^{\ell_1,\ell_2}|
\sup_{\beta\in\Theta_B}\|
\partial_\beta^2 b_{i-1}^{\ell_2}(\beta)\|\right.\right.\\
%%%%%%%
&\qquad\qquad+
\sup_{\beta\in\Theta_B}\|\partial_\beta\partial_{\beta^\ell}b^{\ell_1}_{i-1}(\beta)\|
\sup_{\alpha\in\Theta_A}|(A_{i-1}^{-1}(\alpha_0))^{\ell_1,\ell_2}|
\sup_{\beta\in\Theta_B}\|\partial_\beta b_{i-1}^{\ell_2}(\beta)\|\\
%%%%%
&\left.\left.\qquad\qquad+2\sup_{\beta\in\Theta_B}
\|\partial_\beta^2\partial_{\beta^\ell}b_{i-1}^{\ell_1}(\beta)\|
\sup_{\beta\in\Theta_B}|b_{i-1}^{\ell_2}(\beta)|\right)
\right]\\
%\\
%&\le
%\frac{\sqrt{h_n}}{n}\sum_{i=1}^n
%\sum_{\ell_1,\ell_2=1}^d
%\EEt\left[
%\sup_{\beta\in\Theta_B}|\partial_{\beta^\ell}b_{i-1}^{\ell_1}(\beta)|
%\sup_{\alpha\in\Theta_A}\|\partial_\alpha(A_{i-1}^{-1}(\alpha))^{\ell_1,\ell_2}\|
%\sup_{\beta\in\Theta_B}|b_{i-1}^{\ell_2}(\beta)|
%\right]\\
%&
&\le C\sqrt{h_n}\lto0,
\end{align*}
%%%%%%%
\begin{align*}
&\EEt\left[
\frac{1}{n^{3/2}h_n}\sum_{i=1}^n
\sum_{\ell_1,\ell_2=1}^d
\sup_{\beta\in\Theta_B}\|\partial_\beta\partial_{\beta^\ell}b_{i-1}^{\ell_1}(\beta)\|
\sup_{\alpha\in\Theta_A}\|\partial_\alpha(A_{i-1}^{-1}(\alpha))^{\ell_1,\ell_2}\|
\left(|(\DeX)^{\ell_2}|+h_n\sup_{\beta\in\Theta_B}|b_{i-1}^{\ell_2}(\beta)|\right)
\right]\\
&\le
\frac{C}{n^{3/2}h_n}\sum_{i=1}^n
\sum_{\ell_1,\ell_2=1}^d
\EEt\left[
\sup_{\beta\in\Theta_B}\|\partial_\beta\partial_{\beta^\ell}b_{i-1}^{\ell_1}(\beta)\|^2
\sup_{\alpha\in\Theta_A}\|\partial_\alpha(A_{i-1}^{-1}(\alpha))^{\ell_1,\ell_2}\|^2
\right]^{1/2}\\
&\qquad\qquad\qquad\cdot\EEt\left[
|(\DeX)^{\ell_2}|^2+h_n^2\sup_{\beta\in\Theta_B}|b_{i-1}^{\ell_2}(\beta)|^2
\right]^{1/2}\\
&\le\frac{C'}{\sqrt{nh_n}}\lto0,
\end{align*}
%%%%%%%
\begin{align*}
&\EEt\left[
\frac{1}{n\sqrt{nh_n}}\sum_{i=1}^n
\sum_{\ell_1,\ell_2=1}^d
\sup_{\beta\in\Theta_B}|\partial_{\beta^\ell}b_{i-1}^{\ell_1}(\beta)|
\sup_{\alpha\in\Theta_A}\|\partial_\alpha^2(A_{i-1}^{-1}(\alpha))^{\ell_1,\ell_2}\|
\left(|(\DeX)^{\ell_2}|+h_n\sup_{\beta\in\Theta_B}|b_{i-1}^{\ell_2}(\beta)|\right)
\right]\\
&\le
\frac{C}{n\sqrt{nh_n}}\sum_{i=1}^n
\sum_{\ell_1,\ell_2=1}^d
\EEt\left[
\sup_{\beta\in\Theta_B}|\partial_{\beta^\ell}b_{i-1}^{\ell_1}(\beta)|^2
\sup_{\alpha\in\Theta_A}\|\partial_\alpha^2(A_{i-1}^{-1}(\alpha))^{\ell_1,\ell_2}\|^2
\right]^{1/2}\\
&\qquad\qquad\qquad\cdot\EEt\left[
|(\DeX)^{\ell_2}|^2+h_n^2\sup_{\beta\in\Theta_B}|b_{i-1}^{\ell_2}(\beta)|^2
\right]^{1/2}\\
&\le\frac{C'}{\sqrt n}\lto0.
\end{align*}
%%%%%%%%%%%%%%%%%%%%
Hence, 
{ in the same way as the proof of \eqref{ap3}, 
it follows from \textbf{[A8]} that} 
\begin{align*}
\frac1{(nh_n)^{(m+1)/2}}\sum_{i=1}^n\|\Ym\|&=o_p(1),\quad
\frac{1}{n\sqrt{nh_n}}\sum_{i=1}^n\|\Zd\|=o_p(1),\quad
\frac{\sqrt{h_n}}{n}\sum_{i=1}^n\|\Zf\|=o_p(1),\\
\frac1{n^{3/2}h_n}\sum_{i=1}^n\|\Zg\|&=o_p(1),\quad
\frac1{n\sqrt{nh_n}}\sum_{i=1}^n\|\Zs\|=o_p(1).
\end{align*}
\end{proof}
%%%%%%%%%%%%%%%%%%%%%%%%%%%%%%%%%%%%%%%%%%%%%%%%%%%%%%%%%%%%%%%%%%%%%%%%%%%%
\begin{proof}[\bf{Proof of Corollary \ref{th3}}]
Note that by \textbf{[A9]}, we can show \eqref{hf3} 
in the proof of Theorem \ref{th03}.
If we assume \textbf{[A9']}, then one has  
\begin{align*}
\frac{1}{(nh_n)^{(M+1)/2}}\max_{1\le k\le n}
\left\|
\sum_{i=1}^k\mathcal Y_{M,i}^\ell-\frac{k}{n}\sum_{i=1}^n\mathcal Y_{M,i}^\ell
\right\|=0
\end{align*}
corresponding to \eqref{hf3}. 
Therefore, this corollary is shown as in the proof of Theorem \ref{th03}.
\end{proof}
%\input{change_point6}
%%%%%%%%%%%%%%%%%%%%%%%%%%%%%%%%%%%%%%%%%%%%%%%%%%%%%%%%%%%%%%%%%%%%%%%%%%%
\begin{proof}[\bf{Proof of Theorem \ref{th4}}]
If we prove
\begin{align}
&\frac{1}{[nt^*]}\sum_{i=1}^{[nt^*]}\hat\eta_i\pto F(\alpha_1^*),\label{kk1}\\
&\frac{1}{n-[nt^*]}\sum_{i=[nt^*]+1}^n\hat\eta_i\pto F(\alpha_2^*),\label{kk2}
\end{align}
{ then, we see from {\textbf{[B2]}} that}
%we obtain, from {\textbf{[B2]}}, 
\begin{align*}
\frac{1}{n}\sum_{i=1}^n\hat\eta_i
&=
\frac{[nt^*]}{n}\frac{1}{[nt^*]}
\sum_{i=1}^{[nt^*]}\hat\eta_i
+\frac{n-[nt^*]}{n}\frac{1}{n-[nt^*]}\sum_{i=[nt^*]+1}^n\hat\eta_i\\
&\pto
t^*F(\alpha_1^*)+(1-t^*)F(\alpha_2^*)
\end{align*}
and
\begin{align*}
\frac{1}{n}\sum_{i=1}^{[nt^*]}\hat\eta_i-\frac{[nt^*]}{n}\frac{1}{n}\sum_{i=1}^n\hat\eta_i
&=
\frac{[nt^*]}{n}\left(
\frac{1}{[nt^*]}\sum_{i=1}^{[nt^*]}\hat\eta_i-\frac{1}{n}\sum_{i=1}^n\hat\eta_i
\right)\\
&\pto
t^*(F(\alpha_1^*)-(t^*F(\alpha_1^*)+(1-t^*)F(\alpha_2^*)))\\
&=t^*(1-t^*)(F(\alpha_1^*)-F(\alpha_2^*))\neq0,
\end{align*}
 { i.e.,} 
\begin{align*}
T_n^\alpha
=\frac{1}{\sqrt{2dn}}\max_{1\le k\le n}
\left|
\sum_{i=1}^k\hat\eta_i-\frac{k}{n}\sum_{i=1}^n\hat\eta_i
\right|
&=\frac{1}{\sqrt{2dn}}\sup_{0\le s\le 1}
\left|
\sum_{i=1}^{[ns]}\hat\eta_i-\frac{[ns]}{n}\sum_{i=1}^n\hat\eta_i
\right|\\
&\ge
\frac{1}{\sqrt{2dn}}
\left|
\sum_{i=1}^{[nt^*]}\hat\eta_i-\frac{[nt^*]}{n}\sum_{i=1}^n\hat\eta_i
\right|\\
&=\sqrt{\frac{n}{2d}}
\left|
\frac{1}{n}\sum_{i=1}^{[nt^*]}\hat\eta_i-\frac{[nt^*]}{n}\frac{1}{n}\sum_{i=1}^n\hat\eta_i
\right|\lto\infty.
\end{align*}
{
Therefore, we have $P(T_n^\alpha > w_1(\epsilon))\lto1$. 
}
%which completes the proof of Theorem \ref{th4}.
%%%%%%%%%%%%%%%%

Let us show \eqref{kk1}. We can express
\begin{align*}
&\frac{1}{[nt^*]}\sum_{i=1}^{[nt^*]}\hat\eta_i\\
&=\frac{1}{[nt^*]}\sum_{i=1}^{[nt^*]}\tr\left(
A^{-1}_{i-1}(\hat\alpha_n)\frac{(\DeX)^{\otimes2}}{h_n}
\right)\\
&=\frac{1}{[nt^*]}\sum_{i=1}^{[nt^*]}\sum_{\ell_1,\ell_2=1}^d
(A^{-1}_{i-1}(\hat\alpha_n))^{\ell_1,\ell_2}
\frac{(\DeX)^{\ell_1}(\DeX)^{\ell_2}}{h_n}\\
&=\frac{1}{[nt^*]}\sum_{i=1}^{[nt^*]}\sum_{\ell_1,\ell_2=1}^d
\left((A^{-1}_{i-1}(\bar\alpha_*))^{\ell_1,\ell_2}
+\int_0^1\partial_\alpha 
(A^{-1}_{i-1}(\bar\alpha_*+u(\hat\alpha_n-\bar\alpha_*)))^{\ell_1,\ell_2}\dd u
(\hat\alpha_n-\bar\alpha_*)\right)
\frac{(\DeX)^{\ell_1}(\DeX)^{\ell_2}}{h_n}\\
&=\frac{1}{[nt^*]}\sum_{i=1}^{[nt^*]}
\tr\left(
A^{-1}_{i-1}(\bar\alpha_*)\frac{(\DeX)^{\otimes2}}{h_n}
\right)\\
&\qquad+\frac{1}{[nt^*]}\sum_{i=1}^{[nt^*]}
\left(\sum_{\ell_1,\ell_2=1}^d
\int_0^1\partial_\alpha 
(A^{-1}_{i-1}(\bar\alpha_*+u(\hat\alpha_n-\bar\alpha_*)))^{\ell_1,\ell_2}\dd u
\frac{(\DeX)^{\ell_1}(\DeX)^{\ell_2}}{h_n}
\right)
(\hat\alpha_n-\bar\alpha_*).
\end{align*}
%%%%%%%%%
From Lemma \ref{ke7}, we have 
\begin{align*}
\frac{1}{[nt^*]}\sum_{i=1}^{[nt^*]}\EE_{\alpha_1^*}\left[
\tr\left(
A^{-1}_{i-1}(\bar\alpha_*)\frac{(\DeX)^{\otimes2}}{h_n}
\right)\middle|\GG
\right]
&=
\frac{1}{[nt^*]}\sum_{i=1}^{[nt^*]}
\tr\left(
A^{-1}_{i-1}(\bar\alpha_*)\EE_{\alpha_1^*}\left[\frac{(\DeX)^{\otimes2}}{h_n}\middle|\GG\right]
\right)\\
&=
\frac{1}{[nt^*]}\sum_{i=1}^{[nt^*]}
\tr\left(
A^{-1}_{i-1}(\bar\alpha_*)A_{i-1}(\alpha_1^*)+\Ri
\right)\\
&\pto
\int_{\mathbb R^d}\tr(A^{-1}(x,\bar\alpha_*)A(x,\alpha_1^*))\dd\mu_{\alpha_1^*}(x)
=F(\alpha_1^*)
\end{align*}
%%%%
and
\begin{align*}
&\frac{1}{[nt^*]^2}\sum_{i=1}^{[nt^*]}\EE_{\alpha_1^*}\left[
\left(\tr\left(
A^{-1}_{i-1}(\bar\alpha_*)\frac{(\DeX)^{\otimes2}}{h_n}
\right)\right)^2\middle|\GG
\right]\\
&=
\frac{1}{[nt^*]^2}\sum_{i=1}^{[nt^*]}\EE_{\alpha_1^*}\left[
\sum_{\ell_1,\ell_2=1}^d\sum_{\ell_3,\ell_4=1}^d
(A_{i-1}^{-1}(\bar\alpha_*))^{\ell_1,\ell_2}(A_{i-1}^{-1}(\bar\alpha_*))^{\ell_3,\ell_4}
\frac{(\DeX)^{\ell_1}(\DeX)^{\ell_2}(\DeX)^{\ell_3}(\DeX)^{\ell_4}}{h_n^2}
\middle|\GG\right]\\
&=
\frac{1}{[nt^*]^2}\sum_{i=1}^{[nt^*]}
\sum_{\ell_1,\ell_2=1}^d\sum_{\ell_3,\ell_4=1}^d
(A_{i-1}^{-1}(\bar\alpha_*))^{\ell_1,\ell_2}(A_{i-1}^{-1}(\bar\alpha_*))^{\ell_3,\ell_4}
\EE_{\alpha_1^*}\left[
\frac{(\DeX)^{\ell_1}(\DeX)^{\ell_2}(\DeX)^{\ell_3}(\DeX)^{\ell_4}}{h_n^2}
\middle|\GG\right]\\
&=
\frac{1}{[nt^*]^2}\sum_{i=1}^{[nt^*]}
\sum_{\ell_1,\ell_2=1}^d\sum_{\ell_3,\ell_4=1}^d
(A_{i-1}^{-1}(\bar\alpha_*))^{\ell_1,\ell_2}(A_{i-1}^{-1}(\bar\alpha_*))^{\ell_3,\ell_4}
\Ro\\
&=
\frac{1}{[nt^*]}\cdot\frac{1}{[nt^*]}\sum_{i=1}^{[nt^*]}\Ro
\pto0.
\end{align*}
Therefore, from Lemma 9 of Genon-Catalot and Jacod (1993),  
\begin{align}
\frac{1}{[nt^*]}\sum_{i=1}^{[nt^*]}
\tr\left(
A^{-1}_{i-1}(\bar\alpha_*)\frac{(\DeX)^{\otimes2}}{h_n}
\right)
\pto
%\int_{\mathbb R^d}\tr(A^{-1}(x,\bar\alpha_*)A(x,\alpha_1^*))\dd\mu_{\alpha_1^*}(x)=
F(\alpha_1^*).
\label{rt1}
\end{align}
%%%%%%%%%%%%%%%
On the other hand, from {\textbf{[B1]}} and
%since $\bar\alpha_*\in\Theta_A$, there exists an open neighborhood $\mathcal O_{\bar\alpha_*}$
%of $\bar\alpha_*$ such that $\mathcal O_{\bar\alpha_*}\subset\Theta_A$.
%If $\hat\alpha_n\in\mathcal O_{\bar\alpha_*}$, then
\begin{align*}
\frac{1}{[nt^*]}\sum_{i=1}^{[nt^*]}
\sum_{\ell_1,\ell_2=1}^d
\int_0^1\partial_\alpha 
(A_{i-1}^{-1}(\bar\alpha_*+u(\hat\alpha_n-\bar\alpha_*)))^{\ell_1,\ell_2}\dd u
\frac{(\DeX)^{\ell_1}(\DeX)^{\ell_2}}{h_n}=O_p(1),
\end{align*}
%and {\textbf{[B1]}},
we have
\begin{align}
\left(\frac{1}{[nt^*]}\sum_{i=1}^{[nt^*]}
\sum_{\ell_1,\ell_2=1}^d
\int_0^1\partial_\alpha 
(A_{i-1}^{-1}(\bar\alpha_*+u(\hat\alpha_n-\bar\alpha_*)))^{\ell_1,\ell_2}\dd u
\frac{(\DeX)^{\ell_1}(\DeX)^{\ell_2}}{h_n}\right)(\hat\alpha_n-\bar\alpha_*)=o_p(1).
\label{rt2}
\end{align}
Hence, { it follows from \eqref{rt1} and \eqref{rt2} that we obtain \eqref{kk1}}.
In the same way, noting that $\alpha_0=\alpha_2^*$ 
if $[nt^*]+1 \le i\le n$, we have \eqref{kk2},
which completes the proof of Theorem \ref{th4}.
\end{proof}
%%%%%%%%%%%%%%%%%%%%%%%%%%%%%%%%%%%%%%%%%%%%%%%%%%%%%%

%%%%%%%%%%%%%%%%%%%%%%%%%%%%%%%%%%%%%%%%%%%%%%%%%%%%%%%%%%%%%%%%%%%%%%%%%%%
\begin{proof}[\bf{Proof of Theorem \ref{th5}}]
If we prove
\begin{align}
&\frac{1}{[nt^*]h_n}\sum_{i=1}^{[nt^*]}\hat\xi_i\pto G(\beta_1^*),\label{ll1}\\
&\frac{1}{(n-[nt^*])h_n}\sum_{i=[nt^*]+1}^n\hat\xi_i\pto G(\beta_2^*),\label{ll2}
\end{align}
{ then, it follows from {\textbf{[B4]}} that} 
\begin{align*}
\frac{1}{nh_n}\sum_{i=1}^n\hat\xi_i
&=
\frac{[nt^*]}{n}\frac{1}{[nt^*]h_n}
\sum_{i=1}^{[nt^*]}\hat\xi_i
+\frac{n-[nt^*]}{n}\frac{1}{(n-[nt^*])h_n}\sum_{i=[nt^*]+1}^n\hat\xi_i\\
&\pto
t^*G(\beta_1^*)+(1-t^*)G(\beta_2^*)
\end{align*}
and
\begin{align*}
\frac{1}{nh_n}\sum_{i=1}^{[nt^*]}\hat\xi_i-\frac{[nt^*]}{n}\frac{1}{nh_n}\sum_{i=1}^n\hat\xi_i
&=
\frac{[nt^*]}{n}\left(
\frac{1}{[nt^*]h_n}\sum_{i=1}^{[nt^*]}\hat\xi_i-\frac{1}{nh_n}\sum_{i=1}^n\hat\xi_i
\right)\\
&\pto
t^*(G(\beta_1^*)-(t^*G(\beta_1^*)+(1-t^*)G(\beta_2^*)))\\
&=t^*(1-t^*)(G(\beta_1^*)-G(\beta_2^*))\neq0,
\end{align*}
 { i.e.,} 
\begin{align*}
T_{1,n}^\beta
=\frac{1}{\sqrt{dnh_n}}\max_{1\le k\le n}
\left|
\sum_{i=1}^k\hat\xi_i-\frac{k}{n}\sum_{i=1}^n\hat\xi_i
\right|
&=\frac{1}{\sqrt{dnh_n}}\sup_{0\le s\le 1}
\left|
\sum_{i=1}^{[ns]}\hat\xi_i-\frac{[ns]}{n}\sum_{i=1}^n\hat\xi_i
\right|\\
&\ge
\frac{1}{\sqrt{dnh_n}}
\left|
\sum_{i=1}^{[nt^*]}\hat\xi_i-\frac{[nt^*]}{n}\sum_{i=1}^n\hat\xi_i
\right|\\
&=\sqrt{\frac{nh_n}{d}}
\left|
\frac{1}{nh_n}\sum_{i=1}^{[nt^*]}\hat\xi_i-\frac{[nt^*]}{n}\frac{1}{nh_n}\sum_{i=1}^n\hat\xi_i
\right|\lto\infty.
\end{align*}
{ Hence, one has $P(T_{1,n}^\beta > w_1(\epsilon))\lto1$.}
%Therefore, we have $P(T_{1,n}^\beta > w_1(\epsilon))\lto1$, 
%which completes the proof of Theorem \ref{th5}.

%%%%%%%%%%%%%%%%
Let us show \eqref{ll1}. We can express
\begin{align*}
&\frac{1}{[nt^*]h_n}\sum_{i=1}^{[nt^*]}\hat\xi_i\\
&=\frac{1}{[nt^*]h_n}\sum_{i=1}^{[nt^*]}
\kappa_{i-1}(\hat\alpha_n)(\Xt-\Xs-h_nb_{i-1}(\hat\beta_n))\\
&=\frac{1}{[nt^*]h_n}\sum_{i=1}^{[nt^*]}
\left(\kappa_{i-1}(\alpha_0)+(\hat\alpha_n-\alpha_0)^\TT
\int_0^1\partial_\alpha\kappa_{i-1}(\alpha_0+u(\hat\alpha_n-\alpha_0))^\TT\dd u
\right)
(\Xt-\Xs-h_nb_{i-1}(\hat\beta_n))\\
&=\frac{1}{[nt^*]}\sum_{i=1}^{[nt^*]}
\kappa_{i-1}(\alpha_0)\frac{\DeX}{h_n}
-\frac{1}{[nt^*]}\sum_{i=1}^{[nt^*]}\kappa_{i-1}(\alpha_0)b_{i-1}(\hat\beta_n)\\
&\qquad+\sqrt n(\hat\alpha_n-\alpha_0)^\TT
\left(
\frac{1}{\sqrt{nh_n}[nt^*]}\sum_{i=1}^{[nt^*]}
\int_0^1\partial_\alpha\kappa_{i-1}(\alpha_0+u(\hat\alpha_n-\alpha_0))^\TT\dd u
\left(\frac{\DeX}{\sqrt{h_n}}-\sqrt{h_n}b_{i-1}(\hat\beta_n)\right)
\right).
\end{align*}
%%%%%%%%%
From Lemma \ref{ke7}, we have 
\begin{align*}
\frac{1}{[nt^*]}\sum_{i=1}^{[nt^*]}\EE_{\theta_1^*}\left[
\kappa_{i-1}(\alpha_0)\frac{\DeX}{h_n}
\middle|\GG
\right]
&=
\frac{1}{[nt^*]}\sum_{i=1}^{[nt^*]}
\kappa_{i-1}(\alpha_0)(b_{i-1}(\beta_1^*)+\Ri)\\
&\pto
\int_{\mathbb R^d}\kappa(x,\alpha_0)b(x,\beta_1^*)\dd\mu_{\theta_1^*}(x)
\end{align*}
%%%%
and
\begin{align*}
\frac{1}{[nt^*]^2}\sum_{i=1}^{[nt^*]}\EE_{\theta_1^*}\left[
\left(\kappa_{i-1}(\alpha_0)\frac{\DeX}{h_n}\right)^2
\middle|\GG
\right]
&=
\frac{1}{[nt^*]^2h_n}\sum_{i=1}^{[nt^*]}
\kappa_{i-1}(\alpha_0)
\EE_{\theta_1^*}\left[
\frac{(\DeX)^{\otimes2}}{h_n}
\middle|\GG\right]
\kappa_{i-1}(\alpha_0)^\TT\\
&=
\frac{1}{[nt^*]h_n}\frac{1}{[nt^*]}\sum_{i=1}^{[nt^*]}\Ro
\pto0.
\end{align*}
Therefore, from Lemma 9 of Genon-Catalot and Jacod (1993),  
\begin{align}
\frac{1}{[nt^*]}\sum_{i=1}^{[nt^*]}
\kappa_{i-1}(\alpha_0)\frac{\DeX}{h_n}
\pto
%\int_{\mathbb R^d}\tr(A^{-1}(x,\bar\alpha_*)A(x,\alpha_1^*))\dd\mu_{\alpha_1^*}(x)=
\int_{\mathbb R^d}\kappa(x,\alpha_0)b(x,\beta_1^*)\dd\mu_{\theta_1^*}(x).
\label{rs1}
\end{align}
%%%%%%%%%%%%%%%
On the other hand, from \textbf{[B3']} and 
%since $\bar\alpha_*\in\Theta_A$, there exists an open neighborhood $\mathcal O_{\bar\alpha_*}$
%of $\bar\alpha_*$ such that $\mathcal O_{\bar\alpha_*}\subset\Theta_A$.
%If $\hat\alpha_n\in\mathcal O_{\bar\alpha_*}$, then
\begin{align*}
&\frac{1}{[nt^*]}\sum_{i=1}^{[nt^*]}
\kappa_{i-1}(\alpha_0)
\int_0^1\partial_\beta b_{i-1}(\bar\beta_*+u(\hat\beta_n-\bar\beta_*))\dd u=O_p(1),\\
&\frac{1}{[nt^*]}\sum_{i=1}^{[nt^*]}
\int_0^1\partial_\alpha\kappa_{i-1}(\alpha_0+u(\hat\alpha_n-\alpha_0))^\TT\dd u
\left(\frac{\DeX}{\sqrt{h_n}}-\sqrt{h_n}b_{i-1}(\hat\beta_n)\right)
=O_p(1), 
\end{align*}
%and {\textbf{[B3]}},
we have
\begin{align}
&\frac{1}{[nt^*]}\sum_{i=1}^{[nt^*]}\kappa_{i-1}(\alpha_0)b_{i-1}(\hat\beta_n)
\nonumber\\
&=
\frac{1}{[nt^*]}\sum_{i=1}^{[nt^*]}\kappa_{i-1}(\alpha_0)
\left(
b_{i-1}(\bar\beta_*)
+\int_0^1\partial_\beta b_{i-1}(\bar\beta_*+u(\hat\beta_n-\bar\beta_*))\dd u
(\hat\beta_n-\bar\beta_*)
\right)\nonumber\\
&=
\frac{1}{[nt^*]}\sum_{i=1}^{[nt^*]}\kappa_{i-1}(\alpha_0)b_{i-1}(\bar\beta_*)
+
\left(
\frac{1}{[nt^*]}\sum_{i=1}^{[nt^*]}
\kappa_{i-1}(\alpha_0)
\int_0^1\partial_\beta b_{i-1}(\bar\beta_*+u(\hat\beta_n-\bar\beta_*))\dd u
\right)(\hat\beta_n-\bar\beta_*)\nonumber\\
&\pto
\int_{\mathbb R^d}\kappa(x,\alpha_0)b(x,\bar\beta_*)\dd\mu_{\theta_1^*}(x)
\label{rs3}
\end{align}
%%%%
and
\begin{align}
\sqrt n(\hat\alpha_n-\alpha_0)^\TT
\left(
\frac{1}{\sqrt{nh_n}[nt^*]}\sum_{i=1}^{[nt^*]}
\int_0^1\partial_\alpha\kappa_{i-1}(\alpha_0+u(\hat\alpha_n-\alpha_0))^\TT\dd u
\left(\frac{\DeX}{\sqrt{h_n}}-\sqrt{h_n}b_{i-1}(\hat\beta_n)\right)
\right)=o_p(1).
\label{rs2}
\end{align}
Hence, from \eqref{rs1}, \eqref{rs3} and \eqref{rs2}, we obtain \eqref{ll1}.
In the same manner, noting that $\beta_0=\beta_2^*$ 
if $[nt^*]+1 \le i\le n$, we have \eqref{ll2},
which completes the proof of Theorem \ref{th5}.
\end{proof}
%%%%%%%%%%%%%%%%%%%%%%%%%%%%%%%%%%%%%%%%%%%%%%%%%%%%%%

%%%%%%%%%%%%%%%%%%%%%%%%%%%%%%%%%%%%%%%%%%%%%%%%%%%%%%%%%%%%%%%%%%%%%%%%%%%
\begin{proof}[\bf{Proof of Theorem \ref{th6}}]
If we prove
\begin{align}
&\frac{1}{[nt^*]h_n}\sum_{i=1}^{[nt^*]}\hat\zeta_i\pto H(\beta_1^*),\label{mm1}\\
&\frac{1}{(n-[nt^*])h_n}\sum_{i=[nt^*]+1}^n\hat\zeta_i\pto H(\beta_2^*),\label{mm2}
\end{align}
we obtain, from {\textbf{[B5]}},
\begin{align*}
\frac{1}{nh_n}\sum_{i=1}^n\hat\zeta_i
&=
\frac{[nt^*]}{n}\frac{1}{[nt^*]h_n}
\sum_{i=1}^{[nt^*]}\hat\zeta_i
+\frac{n-[nt^*]}{n}\frac{1}{(n-[nt^*])h_n}\sum_{i=[nt^*]+1}^n\hat\zeta_i\\
&\pto
t^*H(\beta_1^*)+(1-t^*)H(\beta_2^*)
\end{align*}
and
\begin{align*}
\frac{1}{nh_n}\sum_{i=1}^{[nt^*]}\hat\zeta_i-\frac{[nt^*]}{n}\frac{1}{nh_n}\sum_{i=1}^n\hat\zeta_i
&=
\frac{[nt^*]}{n}\left(
\frac{1}{[nt^*]h_n}\sum_{i=1}^{[nt^*]}\hat\zeta_i-\frac{1}{nh_n}\sum_{i=1}^n\hat\zeta_i
\right)\\
&\pto
t^*(H(\beta_1^*)-(t^*H(\beta_1^*)+(1-t^*)H(\beta_2^*)))\\
&=t^*(1-t^*)(H(\beta_1^*)-H(\beta_2^*))\neq0,
\end{align*}
 { i.e.,} 
\begin{align*}
T_{2,n}^\beta
&=\frac{1}{\sqrt{nh_n}}\max_{1\le k\le n}
\left\|
\mathcal I_n^{-1/2}\left(\sum_{i=1}^k\hat\zeta_i-\frac{k}{n}\sum_{i=1}^n\hat\zeta_i
\right)
\right\|\\
&=\frac{1}{\sqrt{nh_n}}\sup_{0\le s\le 1}
\left\|
\mathcal I_n^{-1/2}\left(
\sum_{i=1}^{[ns]}\hat\zeta_i-\frac{[ns]}{n}\sum_{i=1}^n\hat\zeta_i
\right)\right\|\\
&\ge
\frac{1}{\sqrt{nh_n}}
\left\|
\mathcal I_n^{-1/2}\left(
\sum_{i=1}^{[nt^*]}\hat\zeta_i-\frac{[nt^*]}{n}\sum_{i=1}^n\hat\zeta_i
\right)\right\|\\
&=\sqrt{nh_n}
\left\|
\mathcal I_n^{-1/2}\left(
\frac{1}{nh_n}\sum_{i=1}^{[nt^*]}\hat\zeta_i
-\frac{[nt^*]}{n}\frac{1}{nh_n}\sum_{i=1}^n\hat\zeta_i
\right)\right\|\lto\infty.
\end{align*}

%%%%%%%%%%%%%%%%

Let us show \eqref{mm1}. 
One has that %We can express, 
from \textbf{[A9]},
\begin{align*}
&\frac{1}{[nt^*]h_n}\sum_{i=1}^{[nt^*]}\hat\zeta_i^\ell\\
&=\frac{1}{[nt^*]h_n}\sum_{i=1}^{[nt^*]}
\sum_{\ell_1,\ell_2=1}^d
\partial_{\beta^\ell} b_{i-1}^{\ell_1}(\hat\beta_n)(A_{i-1}^{-1}(\hat\alpha_n))^{\ell_1,\ell_2}
(\Xt-\Xs-h_nb_{i-1}(\hat\beta_n))^{\ell_2}\\
&=\frac{1}{[nt^*]h_n}\sum_{i=1}^{[nt^*]}
\sum_{\ell_1,\ell_2=1}^d
\partial_{\beta^\ell} b_{i-1}^{\ell_1}(\hat\beta_n)
\left(
(A_{i-1}^{-1}(\alpha_0))^{\ell_1,\ell_2}
+\int_0^1\partial_\alpha(A_{i-1}^{-1}(\alpha_0+u(\hat\alpha_n-\alpha_0)))^{\ell_1,\ell_2} \dd u
(\hat\alpha_n-\alpha_0)
\right)\\
&\qquad\qquad\qquad\cdot((\DeX)^{\ell_2}-h_nb_{i-1}^{\ell_2}(\hat\beta_n))\\
&=\frac{1}{[nt^*]}\sum_{i=1}^{[nt^*]}
\sum_{\ell_1,\ell_2=1}^d
\partial_{\beta^\ell} b_{i-1}^{\ell_1}(\hat\beta_n)
(A_{i-1}^{-1}(\alpha_0))^{\ell_1,\ell_2}\frac{(\DeX)^{\ell_2}}{h_n}\\
&\qquad-\frac{1}{[nt^*]}\sum_{i=1}^{[nt^*]}
\sum_{\ell_1,\ell_2=1}^d
\partial_{\beta^\ell} b_{i-1}^{\ell_1}(\hat\beta_n)
(A_{i-1}^{-1}(\alpha_0))^{\ell_1,\ell_2}
b_{i-1}^{\ell_2}(\hat\beta_n)\\
&\qquad+
\left(
\frac{1}{\sqrt{nh_n}[nt^*]}\sum_{i=1}^{[nt^*]}
\sum_{\ell_1,\ell_2=1}^d
\partial_{\beta^\ell} b_{i-1}^{\ell_1}(\hat\beta_n)
\left(\frac{(\DeX)^{\ell_2}}{\sqrt{h_n}}-\sqrt{h_n}b_{i-1}^{\ell_2}(\hat\beta_n)\right)
\right.\\
&\qquad\qquad\qquad\left.
\cdot
\int_0^1\partial_\alpha (A_{i-1}^{-1}(\alpha_0+u(\hat\alpha_n-\alpha_0)))^{\ell_1,\ell_2}\dd u
\right)\sqrt n(\hat\alpha_n-\alpha_0)\\
&=
\frac{1}{[nt^*]}\sum_{i=1}^{[nt^*]}\sum_{\ell_1,\ell_2=1}^d
\left(
\partial_{\beta^\ell}b_{i-1}^{\ell_1}(\bar\beta_*)
+\sum_{j=1}^{m-1}\partial_\beta^j\partial_{\beta^\ell}
b_{i-1}^{\ell_1}(\bar\beta_*)\otimes (\hat\beta_n-\bar\beta_*)^{\otimes j}
\right.\\
&\qquad\qquad\qquad\left.
+\int_0^1\partial_\beta^m\partial_{\beta^\ell}b_{i-1}^{\ell_1}
(\bar\beta_*+u(\hat\beta_n-\bar\beta_*))\dd u\otimes
(\hat\beta_n-\bar\beta_*)^{\otimes m}
\right)
(A_{i-1}^{-1}(\alpha_0))^{\ell_1,\ell_2}\frac{(\DeX)^{\ell_2}}{h_n}\\
&\qquad-
\frac{1}{[nt^*]}\sum_{i=1}^{[nt^*]}
\sum_{\ell_1,\ell_2=1}^d
\partial_{\beta^\ell} b_{i-1}^{\ell_1}(\hat\beta_n)
(A_{i-1}^{-1}(\alpha_0))^{\ell_1,\ell_2}
b_{i-1}^{\ell_2}(\hat\beta_n)\\
&\qquad+
\left(
\frac{1}{\sqrt{nh_n}[nt^*]}\sum_{i=1}^{[nt^*]}
\sum_{\ell_1,\ell_2=1}^d
\partial_{\beta^\ell} b_{i-1}^{\ell_1}(\hat\beta_n)
\left(\frac{(\DeX)^{\ell_2}}{\sqrt{h_n}}-\sqrt{h_n}b_{i-1}^{\ell_2}(\hat\beta_n)\right)
\right.\\
&\qquad\qquad\qquad\left.
\cdot
\int_0^1\partial_\alpha (A_{i-1}^{-1}(\alpha_0+u(\hat\alpha_n-\alpha_0)))^{\ell_1,\ell_2}\dd u
\right)\sqrt n(\hat\alpha_n-\alpha_0)\\
%%%%%
&=
\frac{1}{[nt^*]}\sum_{i=1}^{[nt^*]}\sum_{\ell_1,\ell_2=1}^d
\left(
\partial_{\beta^\ell}b_{i-1}^{\ell_1}(\bar\beta_*)
+\sum_{j=1}^{m-1}\partial_\beta^j\partial_{\beta^\ell}
b_{i-1}^{\ell_1}(\bar\beta_*)\otimes (\hat\beta_n-\bar\beta_*)^{\otimes j}
\right)
(A_{i-1}^{-1}(\alpha_0))^{\ell_1,\ell_2}\frac{(\DeX)^{\ell_2}}{h_n}\\
&\qquad
+\frac{1}{(nh_n)^{m/2}[nt^*]}\sum_{i=1}^{[nt^*]}\sum_{\ell_1,\ell_2=1}^d
\left(
\int_0^1\partial_\beta^m\partial_{\beta^\ell}b_{i-1}^{\ell_1}
(\bar\beta_*+u(\hat\beta_n-\bar\beta_*))\dd u\otimes
(\sqrt{nh_n}(\hat\beta_n-\bar\beta_*))^{\otimes m}
\right)\\
&\qquad\qquad\qquad
\cdot(A_{i-1}^{-1}(\alpha_0))^{\ell_1,\ell_2}\frac{(\DeX)^{\ell_2}}{h_n}\\
&\qquad-
\frac{1}{[nt^*]}\sum_{i=1}^{[nt^*]}
\sum_{\ell_1,\ell_2=1}^d
\partial_{\beta^\ell} b_{i-1}^{\ell_1}(\hat\beta_n)
(A_{i-1}^{-1}(\alpha_0))^{\ell_1,\ell_2}
b_{i-1}^{\ell_2}(\hat\beta_n)\\
&\qquad+
\left(
\frac{1}{\sqrt{nh_n}[nt^*]}\sum_{i=1}^{[nt^*]}
\sum_{\ell_1,\ell_2=1}^d
\partial_{\beta^\ell} b_{i-1}^{\ell_1}(\hat\beta_n)
\left(\frac{(\DeX)^{\ell_2}}{\sqrt{h_n}}-\sqrt{h_n}b_{i-1}^{\ell_2}(\hat\beta_n)\right)
\right.\\
&\qquad\qquad\qquad\left.
\cdot
\int_0^1\partial_\alpha (A_{i-1}^{-1}(\alpha_0+u(\hat\alpha_n-\alpha_0)))^{\ell_1,\ell_2}\dd u
\right)\sqrt n(\hat\alpha_n-\alpha_0).
\end{align*}
%%%%%%%%%
It follows from Lemma \ref{ke7} that for $f\in C_{\uparrow}^{4,1}(\mathbb R^d\times \Theta_B)$,
\begin{align*}
&\frac{1}{[nt^*]}\sum_{i=1}^{[nt^*]}\EE_{\theta_1^*}\left[
f_{i-1}(\bar\beta_*)(A_{i-1}^{-1}(\alpha_0))^{\ell_1,\ell_2}\frac{(\DeX)^{\ell_2}}{h_n}
\middle|\GG
\right]\\
&=
\frac{1}{[nt^*]}\sum_{i=1}^{[nt^*]}
f_{i-1}(\bar\beta_*)(A_{i-1}^{-1}(\alpha_0))^{\ell_1,\ell_2}
(b_{i-1}^{\ell_2}(\beta_1^*)+\Ri)\\
&\pto
\int_{\mathbb R^d}
f(x,\bar\beta_*)(A^{-1}(x,\alpha_0))^{\ell_1,\ell_2}
b^{\ell_2}(x,\beta_1^*)\dd\mu_{\theta_1^*}(x)
\end{align*}
%%%%
and
\begin{align*}
&\frac{1}{[nt^*]^2}\sum_{i=1}^{[nt^*]}\EE_{\theta_1^*}\left[
\left(
f_{i-1}(\bar\beta_*)(A_{i-1}^{-1}(\alpha_0))^{\ell_1,\ell_2}\frac{(\DeX)^{\ell_2}}{h_n}
\right)^2
\middle|\GG
\right]\\
&=
\frac{1}{[nt^*]^2h_n}\sum_{i=1}^{[nt^*]}
(f_{i-1}(\bar\beta_*)(A_{i-1}^{-1}(\alpha_0))^{\ell_1,\ell_2})^2
\EE_{\theta_1^*}\left[
\frac{((\DeX)^{\ell_2})^2}{h_n}
\middle|\GG\right]\\
&=
\frac{1}{[nt^*]h_n}\frac{1}{[nt^*]}\sum_{i=1}^{[nt^*]}\Ro
\pto0.
\end{align*}
Therefore, by Lemma 9 of Genon-Catalot and Jacod (1993),  
\begin{align*}
\frac{1}{[nt^*]}\sum_{i=1}^{[nt^*]}
f_{i-1}(\bar\beta_*)(A_{i-1}^{-1}(\alpha_0))^{\ell_1,\ell_2}\frac{(\DeX)^{\ell_2}}{h_n}
\pto
\int_{\mathbb R^d}
f(x,\bar\beta_*)(A^{-1}(x,\alpha_0))^{\ell_1,\ell_2}
b^{\ell_2}(x,\beta_1^*)\dd\mu_{\theta_1^*}(x),
\end{align*}
From this and {\textbf{[B3]}}, we obtain
\begin{align*}
&\frac{1}{[nt^*]}\sum_{i=1}^{[nt^*]}\sum_{\ell_1,\ell_2=1}^d
\left(
\partial_{\beta^\ell}b_{i-1}^{\ell_1}(\bar\beta_*)
+\sum_{j=1}^{m-1}\partial_\beta^j\partial_{\beta^\ell}
b_{i-1}^{\ell_1}(\bar\beta_*)\otimes (\hat\beta_n-\bar\beta_*)^{\otimes j}
\right)
(A_{i-1}^{-1}(\alpha_0))^{\ell_1,\ell_2}\frac{(\DeX)^{\ell_2}}{h_n}\\
&\pto
\int_{\mathbb R^d}
\sum_{\ell_1,\ell_2=1}^d\partial_{\beta^\ell}b^{\ell_1}(x,\bar\beta_*)
(A^{-1}(x,\alpha_0))^{\ell_1,\ell_2}
b^{\ell_2}(x,\beta_1^*)\dd\mu_{\theta_1^*}(x),
\end{align*}
 { i.e.,} 
\begin{align}
\frac{1}{[nt^*]}\sum_{i=1}^{[nt^*]}
\partial_{\beta}b_{i-1}(\hat\beta_n)^\TT
A_{i-1}^{-1}(\alpha_0)\frac{\DeX}{h_n}
\pto
\int_{\mathbb R^d}
\partial_\beta b(x,\bar\beta_*)^\TT
A^{-1}(x,\alpha_0)
b(x,\beta_1^*)\dd\mu_{\theta_1^*}(x).
\label{rr1}
\end{align}

%%%%%%%%%%%%%%%
On the other hand, from {\textbf{[B3]}} and
\begin{align*}
&\frac{1}{[nt^*]}\sum_{i=1}^{[nt^*]}
\sum_{\ell_1,\ell_2=1}^d
\partial_{\beta^\ell}b_{i-1}^{\ell_1}(\bar\beta_*)(A_{i-1}^{-1}(\alpha_0))^{\ell_1,\ell_2}
\int_0^1\partial_\beta b_{i-1}^{\ell_2}(\bar\beta_*+u(\hat\beta_n-\bar\beta_*))\dd u
=O_p(1),\\
&\frac{1}{[nt^*]}\sum_{i=1}^{[nt^*]}
\sum_{\ell_1,\ell_2=1}^d
b_{i-1}^{\ell_2}(\bar\beta_*)(A_{i-1}^{-1}(\alpha_0))^{\ell_1,\ell_2}
\int_0^1\partial_\beta\partial_{\beta^\ell} b_{i-1}^{\ell_1}(\bar\beta_*+u(\hat\beta_n-\bar\beta_*))\dd u
=O_p(1),\\
&
\frac{1}{[nt^*]}\sum_{i=1}^{[nt^*]}\sum_{\ell_1,\ell_2=1}^d
\int_0^1\partial_\beta^m\partial_{\beta^\ell}b_{i-1}^{\ell_1}
(\bar\beta_*+u(\hat\beta_n-\bar\beta_*))\dd u
(A_{i-1}^{-1}(\alpha_0))^{\ell_1,\ell_2}\frac{(\DeX)^{\ell_2}}{\sqrt{h_n}}
=O_p(1),\\
&\frac{1}{[nt^*]}\sum_{i=1}^{[nt^*]}
\sum_{\ell_1,\ell_2=1}^d
\int_0^1\partial_\beta \partial_{\beta^\ell} 
b_{i-1}^{\ell_1}(\bar\beta_*+u(\hat\beta_n-\bar\beta_*))^\TT\dd u
(A_{i-1}^{-1}(\alpha_0))^{\ell_1,\ell_2}
\int_0^1\partial_\beta b_{i-1}^{\ell_2}(\bar\beta_*+u(\hat\beta_n-\bar\beta_*))\dd u
=O_p(1),\\
&\frac{1}{[nt^*]}\sum_{i=1}^{[nt^*]}
\sum_{\ell_1,\ell_2=1}^d
\partial_{\beta^\ell}b_{i-1}^{\ell_1}(\hat\beta_n)
\left(\frac{(\DeX)^{\ell_2}}{\sqrt{h_n}}-\sqrt{h_n}b_{i-1}^{\ell_2}(\hat\beta_n)\right)
\int_0^1\partial_\alpha (A_{i-1}^{-1}(\alpha_0+u(\hat\alpha_n-\alpha_0)))^{\ell_1,\ell_2}\dd u
=O_p(1),
\end{align*}
%and {\textbf{[B3]}}, 
we have 
\begin{align*}
&\frac{1}{[nt^*]}\sum_{i=1}^{[nt^*]}
\sum_{\ell_1,\ell_2=1}^d
\partial_{\beta^\ell} b_{i-1}^{\ell_1}(\hat\beta_n)
(A_{i-1}^{-1}(\alpha_0))^{\ell_1,\ell_2}
b_{i-1}^{\ell_2}(\hat\beta_n)\\
&=
\frac{1}{[nt^*]}\sum_{i=1}^{[nt^*]}
\sum_{\ell_1,\ell_2=1}^d
\left(
\partial_{\beta^\ell} b_{i-1}^{\ell_1}(\bar\beta_*)
+\int_0^1\partial_\beta\partial_{\beta^\ell} 
b_{i-1}^{\ell_1}(\bar\beta_*+u(\hat\beta_n-\bar\beta_*))\dd u (\hat\beta_n-\bar\beta)
\right)
(A_{i-1}^{-1}(\alpha_0))^{\ell_1,\ell_2}\\
&\qquad\cdot\left(
b_{i-1}^{\ell_2}(\bar\beta_*)
+\int_0^1\partial_\beta b_{i-1}^{\ell_2}(\beta_1^*+u(\hat\beta_n-\bar\beta_*))\dd u
(\hat\beta_n-\bar\beta_*)
\right)\\
&=
\frac{1}{[nt^*]}\sum_{i=1}^{[nt^*]}
\sum_{\ell_1,\ell_2=1}^d
\partial_{\beta^\ell} b_{i-1}^{\ell_1}(\bar\beta_*)
(A_{i-1}^{-1}(\alpha_0))^{\ell_1,\ell_2}
b_{i-1}^{\ell_2}(\bar\beta_*)\\
&\qquad+
\frac{1}{[nt^*]}\sum_{i=1}^{[nt^*]}
\sum_{\ell_1,\ell_2=1}^d
\left(
\partial_{\beta^\ell} b_{i-1}^{\ell_1}(\bar\beta_*)
(A_{i-1}^{-1}(\alpha_0))^{\ell_1,\ell_2}
\int_0^1\partial_\beta b_{i-1}^{\ell_2}(\bar\beta_*+u(\hat\beta_n-\bar\beta_*))\dd u
\right.\\
&\left.\qquad\qquad\qquad+ b_{i-1}^{\ell_2}(\bar\beta_*)
(A_{i-1}^{-1}(\alpha_0))^{\ell_1,\ell_2}
\int_0^1\partial_\beta \partial_{\beta^\ell} 
b_{i-1}^{\ell_1}(\bar\beta_*+u(\hat\beta_n-\bar\beta_*))\dd u
\right)(\hat\beta_n-\bar\beta_*)\\
&\qquad+(\hat\beta_n-\bar\beta_*)^\TT
\left(
\frac{1}{[nt^*]}\sum_{i=1}^{[nt^*]}
\sum_{\ell_1,\ell_2=1}^d
\int_0^1\partial_\beta \partial_{\beta^\ell} 
b_{i-1}^{\ell_1}(\bar\beta_*+u(\hat\beta_n-\bar\beta_*))^\TT\dd u\right.\\
&\left.\qquad\qquad\qquad\qquad\cdot
(A_{i-1}^{-1}(\alpha_0))^{\ell_1,\ell_2}
\int_0^1\partial_\beta b_{i-1}^{\ell_2}(\bar\beta_*+u(\hat\beta_n-\bar\beta_*))\dd u
\right)
(\hat\beta_n-\bar\beta_*)\\
%%%%
&\pto
\int_{\mathbb R^d}
\sum_{\ell_1,\ell_2=1}^d
\partial_{\beta^\ell}b^{\ell_1}(x,\bar\beta_*)(A^{-1}(x,\alpha_0))^{\ell_1,\ell_2}
b^{\ell_2}(x,\bar\beta_*)\dd\mu_{\theta_1^*}(x),
\end{align*}
 { i.e.,} 
\begin{align}
\frac{1}{[nt^*]}\sum_{i=1}^{[nt^*]}
\partial_{\beta}b_{i-1}(\hat\beta_n)^\TT A_{i-1}^{-1}(\alpha_0)b_{i-1}(\hat\beta_n)
\pto
\int_{\mathbb R^d}
\partial_{\beta}b(x,\bar\beta_*)^\TT A^{-1}(x,\alpha_0)
b(x,\bar\beta_*)\dd\mu_{\theta_1^*}(x),
\label{rr2}
\end{align}
%%%%
and
\begin{align}
&\frac{1}{\sqrt{nh_n}[nt^*]}\sum_{i=1}^{[nt^*]}
\sum_{\ell_1,\ell_2=1}^d
\partial_{\beta^\ell}b^{\ell_1}_{i-1}(\hat\beta_n)
\left(\frac{(\DeX)^{\ell_2}}{\sqrt{h_n}}-\sqrt{h_n}b_{i-1}^{\ell_2}(\hat\beta_n)\right)
\int_0^1\partial_\alpha (A_{i-1}^{-1}(\alpha_0+u(\hat\alpha_n-\alpha_0)))^{\ell_1,\ell_2}\dd u
\nonumber\\
&=o_p(1).
\label{rr3}
\end{align}
Moreover, since $nh_n^{m/(m-1)}< nh_n^{(m+1)/m}$, we have, 
as $nh_n^{m/(m-1)}\lto\infty$ 
\begin{align}
&\frac{(nh_n)^{-m/2}}{[nt^*]}\sum_{i=1}^{[nt^*]}\sum_{\ell_1,\ell_2=1}^d
\left(
\int_0^1\partial_\beta^m\partial_{\beta^\ell}b_{i-1}^{\ell_1}
(\bar\beta_*+u(\hat\beta_n-\bar\beta_*))\dd u\otimes
(\sqrt{nh_n}(\hat\beta_n-\bar\beta_*))^{\otimes m}
\right)
(A_{i-1}^{-1}(\alpha_0))^{\ell_1,\ell_2}\frac{(\DeX)^{\ell_2}}{h_n}
\nonumber\\
&=O_p\left((nh_n^{(m+1)/m})^{-m/2}\right)
=o_p(1).
\label{rr4}
\end{align}
Hence, from \eqref{rr1}, \eqref{rr2}, \eqref{rr3} and \eqref{rr4}, 
we obtain \eqref{mm1}.
By an analogous way, noting that $\beta_0=\beta_2^*$ 
if $[nt^*]+1 \le i\le n$, we have \eqref{mm2},
which completes the proof of Theorem \ref{th6}.
\end{proof}
%%%%%%%%%%%%%%%%%%%%%%%%%%%%%%%%%%%%%%%%%%%%%%%%%%%%%%
\begin{proof}[\bf{Proof of Corollary \ref{th7}}]
Note that by \textbf{[A9]} and \textbf{[B3]}, 
we can show \eqref{rr4} in the proof of Theorem \ref{th6}.
If we assume \textbf{[A9']}, then one has  
\begin{align*}
\frac{1}{[nt^*]}\sum_{i=1}^{[nt^*]}\sum_{\ell_1,\ell_2=1}^d
\left(
\int_0^1\partial_\beta^M\partial_{\beta^\ell}b_{i-1}^{\ell_1}
(\bar\beta_*+u(\hat\beta_n-\bar\beta_*))\dd u\otimes
(\hat\beta_n-\bar\beta_*)^{\otimes M}
\right)
(A_{i-1}^{-1}(\alpha_0))^{\ell_1,\ell_2}\frac{(\DeX)^{\ell_2}}{h_n}
=0
\end{align*}
corresponding to \eqref{rr4}.
%Sinceis sufficient 
Moreover,  in order to show \eqref{rr1}, \eqref{rr2} and \eqref{rr3} in Theorem \ref{th6}, 
it is sufficient  to assume $\hat\beta_n-\beta_*=o_p(1)$ in \textbf{[B3']}. 
Hence, this corollary is shown as in the proof of Theorem \ref{th6}.
\end{proof}
%%%%%%%%% REFERENCES %%%%%%%%%%%%%%%%%%%%%%%%%

%%%%%%%%%% END OF REFERENCE %%%%%%%%%%%%%%%%%%%%%%%%%%%%%%%%%%%%%%%%


\begin{thebibliography}{99}
%\bibitem{Bill}
\harvarditem{}{}{} 
Billingsley, P. (1999).
Convergence of probability measures. Second edition.
Willy, New York.
%%
%\bibitem{DGI}
\harvarditem{}{}{} 
De Gregorio, A., Iacus, S. M. (2008). 
Least squares volatility change point estimation 
for partially observed diffusion processes. 
Communications in Statistics - Theory and Methods, {\bf 37}, 2342-2357.
%%
%\bibitem{GCJ}
\harvarditem{}{}{} 
Genon-Catalot, V., Jacod, J. (1993).
On the estimation of the diffusion coefficient for multidimensional diffusion processes.
Annales de l'Institut Henri Poincar\'e Probabilit\'es et Statistiques, {\bf 29}, 119-151.
%%
%\bibitem{HH}
\harvarditem{}{}{} 
Hall, P., Heyde, C. C. (1980).
Martingale limit theory and its application.
Academic Press, New York.
%%
%\bibitem{IY}
\harvarditem{}{}{} 
Iacus, S. M., Yoshida, N. (2012).
Estimation for the change point of volatility in a stochastic differential equation.
Stochastic Processes and their Applications, {\bf 122}, 1068-1092.
%%
%\bibitem{KaU}
\harvarditem{}{}{} 
Kaino, Y., Uchida, M. (2018).
Hybrid estimators for stochastic differential equations from reduced data. 
Statistical Inference for Stochastic Processes, {\bf 21}, 435-454.
%%
%\bibitem{Kes95}
\harvarditem{}{}{} 
Kessler, M. (1995).
Estimation des param\`etres d'une diffusion par des contrastes corrig\'es.
Comptes rendus de l'Acad\'emie des sciences. Paris Serie I  {\bf 320}, 359-362.
%%
%\bibitem{Kes97}
\harvarditem{}{}{} 
Kessler, M. (1997).
Estimation of an Ergodic Diffusion from Discrete Observations.
Scandinavian Journal of Statistics, {\bf 24}, 211-229.
%%
%\bibitem{KiU}
\harvarditem{}{}{} 
Kitagawa, H., Uchida, M. (2014).
Adaptive test statistics for ergodic diffusion processes sampled at discrete times.
Journal of Statistical Planning and Inference, {\bf 150}, 84-110.
%%
%\bibitem{Kuto}
\harvarditem{}{}{} 
Kutoyants, Y. A. (2004).
Statistical inference for ergodic diffusion processes, Springer.
%%
%\bibitem{Lee}
\harvarditem{}{}{} 
Lee, S. (2011). 
Change point test for dispersion parameter based on discretely observed sample 
from SDE models.
Bulletin of the Korean Mathematical Society, {\bf 48}, 839-845.
%%
%\bibitem{LNY}
\harvarditem{}{}{} 
Lee, S., Nishiyama, Y., Yoshida, N. (2006).
Test for parameter change in diffusion processes by cusum statistics 
based on one-step estimators.
Annals of the Institute of Statistical Mathematics, {\bf 58}, 211-222.
%%
%\bibitem{McL}
\harvarditem{}{}{} 
McLeish, D., L. (1974).
Dependent central limit theorems and invariance principles.
The Annals of Probability, {\bf 2}, 620-628.
%%
%\bibitem{NN12}
\harvarditem{}{}{} 
Negri, I., Nishiyama, Y. (2012).
Asymptotically distribution free test for parameter change in a diffusion process model.
Annals of the Institute of Statistical Mathematics, {\bf 64}, 911-918.
%%
%\bibitem{NN17}
\harvarditem{}{}{} 
Negri, I., Nishiyama, Y. (2017).
Z-process method for change point problems with applications to discretely observed
diffusion processes.
Statistical Methods and Applications, {\bf 26,} 231-250.
%%
%\bibitem{Song}
\harvarditem{}{}{} 
Song, J. (2020).
Robust test for dispersion parameter change in discretely observed diffusion processes.
Computational Statistics and Data Analysis, {\bf 142}, 106832.
%%
%\bibitem{SL}
\harvarditem{}{}{} 
Song, J., Lee, S. (2009). 
Test for parameter change in discretely observed diffusion processes. 
Statistical Inference for Stochastic Processes, {\bf 12}, 165-183.
%%
%\bibitem{Tsukuda}
\harvarditem{}{}{} 
Tsukuda, K. (2017).
A change detection procedure for an ergodic diffusion process.
Annals of the Institute of Statistical Mathematics, {\bf 69}, 833-864.
%%
%
\harvarditem{}{}{} 
Uchida, M. and Yoshida, N. (2011). Estimation for misspecified ergodic diffusion processes from discrete observations. European Series in Applied and Industrial Mathematics: Probability and Statistics, Volume 15, 270-290. 
%\bibitem{UY}
\harvarditem{}{}{} 
Uchida, M., Yoshida, N. (2012).
Adaptive estimation of an ergodic diffusion process based on sampled data.
Stochastic Processes and their Applications, {\bf 122}, 2885-2924.
%%
\end{thebibliography}
\end{document}